\documentclass[11 pt]{article}
\usepackage{amsthm}
\usepackage{amsfonts}
\usepackage{amsmath,amscd,setspace}
\usepackage{fullpage}
\usepackage{amssymb}
\usepackage{centernot} 
\usepackage{enumerate} 
\usepackage{tikz-cd}
\usepackage{mathrsfs}
\usepackage[OT2,T1]{fontenc}
\usepackage{seqsplit}
\usepackage{color}
\usepackage{array}
\usepackage{verbatim}
\usepackage{url}

\usepackage{bbm}

\DeclareSymbolFont{cyrletters}{OT2}{wncyr}{m}{n}
\DeclareMathSymbol{\Sha}{\mathalpha}{cyrletters}{"58}

\definecolor{refkey}{rgb}{1,1,1}
\definecolor{labelkey}{rgb}{1,1,1}
\definecolor{cite}{rgb}{0.9451,0.2706,0.4941}
\definecolor{ruri}{rgb}{0.0078,0.4022,0.8010}

\usepackage[%
bookmarks=true,bookmarksnumbered=true,%
colorlinks=true,linkcolor=ruri,citecolor=red%
 ]{hyperref}

\makeindex 

\renewcommand{\vec}{{}}
\def\F{{\rm \mathbb{F}}}
\def\Z{{\rm \mathbb{Z}}}
\def\N{{\rm \mathbb{N}}}
\def\Q{{\rm \mathbb{Q}}}

\def\G{{\rm \mathbb{G}}}

\def\C{{\rm \mathbb{C}}}

\def\FF{{\rm \mathcal{F}}}
\def\R{{\rm \mathbb{R}}}

\def\L{{{L}}}
\def\P{{\rm \mathbb{P}}}
\def\J{{\rm \mathcal{J}}}
\def\p{{\rm \mathfrak{p}}}

\def\O{{\rm \mathcal{O}}}

\def\A{{\rm \mathbb{A}}}

    \newcommand{\BQ}{{\mathbb {Q}}}

     \newcommand{\BZ}{{\mathbb {Z}}}

    \newcommand{\ra}{\rightarrow} 

\def\Nm{{\rm Nm}}

\def\irr{{\rm irr}}
\def\red{{\rm red}}

\def\ord{{\rm ord}}

\def\avg{{\rm avg}}

\def\sgn{{\rm sgn}}

\def\Aut{{\rm Aut}}

\def\Jac{{\rm Jac}}

\def\Pic{{\rm Pic}}

    \newcommand{\loc}{{\mathrm{loc}}}

\def\Disc{{\rm Disc}}

\def\SL{{\rm SL}}
\def\SO{{\rm SO}}
\def\Vol{{\rm Vol}}

\def\Sym{{\rm Sym}}

\def\GL{{\rm GL}}
\def\Gal{{\rm Gal}}

\def\rk{{\rm rk}}

\newcommand{\corank}{{\mathrm{corank}}}

\def\Sel{{\rm Sel}}

\def\id{{\rm id}}

\numberwithin{equation}{section}

\newtheorem{theorem}{Theorem}[section]
\newtheorem{lemma}[theorem]{Lemma}

\newtheorem{remark}[theorem]{Remark}
\newtheorem{definition}[theorem]{Definition}
\newtheorem{example}[theorem]{Example}

\newtheorem{corollary}[theorem]{Corollary}
\newtheorem{proposition}[theorem]{Proposition}

\usepackage{marginnote,color}

\usepackage{tikz}

\def\shownotes{\def\inline##1##2##3{ \begin{adjustwidth}{3mm}{7mm}\mbox{}\par \noindent
{\color{##1}\hspace{-1.9cm}{\large ##2}\vspace{-\baselineskip}\\##3}
\newline\end{adjustwidth}} \def\inlinewide##1##2##3{ \begin{adjustwidth}{0mm}{0cm}\mbox{}\par \noindent
{\color{##1}\hspace{-1.6cm}{\large ##2}\vspace{-\baselineskip}\\##3}
\newline\end{adjustwidth}}  \def\marg##1##2##3{\marginnote{\color{##1}{\large ##2}\\{\small ##3}}[-.8cm]}}

\shownotes

\let\phi\varphi

\makeatletter
\DeclareRobustCommand{\pmods}[1]{\mkern4mu({\operator@font mod}\mkern6mu#1)}
\makeatother

\newcommand{\rank}{\mathop{\mathrm{rank}}}
\renewcommand{\C}{\mathbb{C}}

\newcommand{\Stab}{\mathop{\mathrm{Stab}}}

    \newcommand{\st}{{\mathrm{st}}}\newcommand{\supp}{{\mathrm{supp}}}

\newcommand{\eps}{\varepsilon}

\renewcommand{\sgn}{\mathrm{sgn}}
\newcommand{\disjcup}{\bigsqcup}

\newcommand{\K}{\mathbb{K}}

\renewcommand{\d}{\partial}

\renewcommand{\P}{\mathbb{P}}

\newcommand{\diag}{\mathrm{diag}}
\renewcommand{\Jac}{\mathrm{Jac}}
\newcommand{\smalltwobytwo}[4]{\left(\begin{smallmatrix} #1 & #2\\ #3 & #4\end{smallmatrix}\right)}

\begin{document}

\title{\vspace{-.5in}Integers expressible as the sum of two rational cubes}

\vspace{-.05in}
\author{Levent Alp\"oge, Manjul Bhargava, and Ari Shnidman \vspace{.085in}\\ (with an appendix by Ashay Burungale and Christopher Skinner)}
\date{}
\newcommand{\Addresses}{{
  \bigskip
  \footnotesize

  \textsc{Harvard University, Cambridge, MA, USA }\par\nopagebreak
   \texttt{alpoge@fas.harvard.edu}

  \medskip

  \textsc{Princeton University,
Princeton, NJ, USA}\par\nopagebreak
 \texttt{bhargava@math.princeton.edu}

  \medskip

  \textsc{Institute of Advanced Studies, Princeton, NJ, USA}\par\nopagebreak
  \texttt{ari.shnidman@gmail.com}

}}

\newcommand{\AddressesA}{{
  \bigskip
  \footnotesize

  \textsc{University of Texas at Austin,
    USA}\par\nopagebreak
  \texttt{ashayk@utexas.edu}

\medskip
  
  \textsc{Princeton University, Princeton, NJ, USA}\par\nopagebreak
  \texttt{cmcls@math.princeton.edu}

}}

\maketitle

\vspace{-.35in}
\begin{abstract}
We prove that a positive proportion of integers are expressible as the sum of two rational cubes, and a positive proportion are not so expressible, thus proving a conjecture of Davenport. More generally, we prove that a positive proportion (in fact, at least one sixth) of elliptic curves in any cubic twist family have rank $0$, and a positive proportion (in fact, at least one sixth) of elliptic curves with good reduction at~$2$ in any cubic twist family have rank~$1$.

Our method involves proving that the average size of the $2$-Selmer group of elliptic curves in any cubic twist family, having any given root number, is $3$.  We accomplish this by generalizing a parametrization,  due to the second author and Ho, of elliptic curves with extra structure by pairs of binary cubic forms. We then count integer points satisfying suitable congruences on a quadric hypersurface in the space of real pairs of binary cubic forms in a fundamental domain for the action of $\SL_2(\Z)\times \SL_2(\Z)$, using a novel combination of geometry-of-numbers methods and the circle method that builds on earlier work of Ruth and the first author. In particular, we make use of a new interpretation of the singular integral and singular series arising in the circle method in terms of real and $p$-adic  Haar measures on the relevant group~$\SL_2\times \SL_2$.
We prove a new uniformity estimate for integral points on such a quadric,  which along with a sieve allows us to prove that the average size of the $2$-Selmer group over the full cubic twist family~is~$3$. 
By suitably partitioning the subset of curves in the family with given root number, we carry out a further sieve to show that the root number is equidistributed and that the same average, now taken over only those curves of given root number, is again~$3$.
Finally, we apply the $p$-parity theorem of Dokchitser and Dokchitser and a new $p$-converse theorem, proven by Burungale and Skinner in the Appendix, to conclude. 

We also prove the analogue of the above results for the sequence of square numbers: namely, we prove that a positive proportion of square integers are expressible as the sum of two rational cubes, and a positive proportion are not so expressible.
\end{abstract}
\vspace{-.15in}

\begin{spacing}{0.15}
\tableofcontents
\end{spacing}

\section{Introduction}

It has long been known which integers can be expressed as the sum of two rational squares. As  was first observed by Girard in $1625$ and Fermat in $1638$, and finally proven by Euler in~$1749$ \cite[pp.~227--231]{Dickson}, 
they are those positive integers whose prime factorizations have all primes that are congruent to $3\pmods{4}$ occurring with even exponent. Nowadays, this can also be deduced from the Hasse--Minkowski local-global principle for quadratic forms. 
Using this precise description, we see that a density of~$0\%$ of integers are the sum of two rational squares.  Moreover, an integer is the sum of two rational~squares if and only if it is the sum of two integer squares. 

In contrast, the integers that are the sum of two rational cubes do not seem to follow any simple pattern: 
$$\!\!\!\!\!1, 2, 6, 7, 8, 9, 12, 13, 15, 16, 17, 19, 20, 22, 26, 27, 28, 30, 31, 33, 34, 35, \ldots$$
The study of this special set of integers has a long history, going back to Fermat, Lagrange, Euler, Legendre, and Dirichlet
(see~\cite[pp.~572--578]{Dickson} for a comprehensive history up to $1918$).
It was conjectured by Davenport and others \cite[pp.~92]{1983icmproceedings} that these integers have positive density.  Indeed, based on the predictions of Goldfeld~\cite{goldfeld}, Katz--Sarnak~\cite{katzsarnak}, and Bektemirov--Mazur--Stein--Watkins~\cite{bektemirovmazursteinwatkins}, it is natural to conjecture that the integers that can be expressed as the sum of two rational cubes should have natural density exactly $1/2$ (see also Zagier--Kramarz \cite{zagierkramarz} and Watkins \cite{watkins} for a discussion of and computations related to this particular  family). However, it has not previously been known whether this density is even  greater than~0 or even less than~1.

Unlike the case of the sum of two rational/integer  squares, it is possible for an integer to be the sum of two rational cubes but {\it not} the sum of two integer cubes, the smallest example being $$6=\left(\frac{17}{21}\right)^3+\left(\frac{37}{21}\right)^3.$$ In fact, it is easy to see that the integers that can be expressed as the sum of two integer cubes have density zero.\footnote{If $|x^3 + y^3|=|x+y|\cdot |x^2-xy+y^2|\leq X$ with $|x|\geq |y|$, then $|x+y|\ll {X}/{|x|^2}$;  hence the number of $(x,y)\in \Z\times \Z$ with $0\neq |x^3 + y^3|\leq X$ is $\,\ll \sum_{|x|\ll  X^{1/3}}X^{1/3}+
\sum_{X^{1/3}\ll|x|\ll X^{1/2}}
X/|x|^2
\ll X^{\frac{2}{3}}$.}

A number of mathematicians, including Euler, Pepin, and Sylvester, also considered
the related problem of representing {\it square} numbers as the sum of two rational cubes. Such numbers also do not seem to follow any simple pattern:
$$\!\!\!\!1^2, 3^2, 4^2, 7^2, 8^2, 13^2, 17^2, 18^2, 20^2, 21^2, 22^2, 24^2, 26^2, 27^2, 28^2, 30^2, 31^2, 32^2, 34^2, 35^2, \ldots $$ 
It is natural to analogously conjecture that such squares should have positive density in the set of all square integers, and indeed by the same heuristics should have density exactly $1/2$ in the set of all squares.

The purpose of this paper is to prove that the density of integers expressible as the sum of two rational cubes is strictly positive and strictly less than $1$---thus proving Davenport's conjecture.  We also prove the analogue of Davenport's conjecture for squares by proving that the density of integers whose square is the sum of two rational cubes is strictly positive and strictly less than $1$. 

We note that there is never any local obstruction for an integer to be the sum of two rational cubes, so proving these theorems must necessarily involve global arguments. 

\subsection{Main results}
We prove the following theorems:
\begin{theorem}\label{main}
A positive proportion of integers are the sum of two rational cubes, and a positive proportion are not.  
\end{theorem}

\begin{theorem}\label{mainsquare}
A positive proportion of square integers are the sum of two rational cubes, and a positive proportion are not. 
\end{theorem}

More precisely, we prove that 
\begin{equation}\label{maineq}\liminf_{X\to\infty} \dfrac{\#\left\{n\in\Z : |n|<X \text{ and } n  \text{ is the sum of two rational cubes} \right\}}{\#\left\{n\in\Z \colon |n| < X\right\}}\geq \frac{2}{21},
\end{equation}
\begin{equation}\label{maineq2}\liminf_{X\to\infty} \dfrac{\#\left\{n\in\Z : |n|<X \text{ and } n  \text{ is not the sum of two rational cubes} \right\}}{\#\left\{n\in\Z \colon |n| < X\right\}}\geq \frac16,
\end{equation}
and 
\begin{equation}\label{mainsquareeq}\liminf_{X\to\infty} \dfrac{\#\left\{n\in\Z : |n|<X \text{ and } n^2  \text{ is the sum of two rational cubes} \right\}}{\#\left\{n\in\Z \colon |n| < X\right\}}\geq \frac{2}{21},
\end{equation}
\begin{equation}\label{mainsquareeq2}\liminf_{X\to\infty} \dfrac{\#\left\{n\in\Z : |n|<X \text{ and } n^2  \text{ is not the sum of two rational cubes} \right\}}{\#\left\{n\in\Z \colon |n| < X\right\}}\geq \frac16.
\end{equation}
In fact, we prove the stronger claim that among the cubic twists $x^3+y^3=nz^3$ (resp.\ $x^3+y^3=n^2z^3$) of the Fermat cubic, at least $1/6$ of these twists have rank $0$ and at least $2/21$ have rank $1$.

More generally, we consider general families of cubic twists of elliptic curves.  Any cubic twist family of elliptic curves over $\Q$ takes the form $E_{d,n}:y^2=x^3+dn^2$, where $d\in\Z$ is fixed  and $n\in\Z$ varies.  Since the elliptic curve $x^3+y^3=n$ can be expressed in Weierstrass form as $y^2=x^3-432n^2$, the family of twists of the Fermat cubic corresponds to the case $d=-432$.\footnote{Note that $E_{d,n}$ and $E_{-27d,n}$ are $3$-isogenous, so their ranks (and $2$-Selmer ranks) agree.}

We prove the following generalization of Theorems~\ref{main} and \ref{mainsquare}.

\begin{theorem}\label{main2}
Fix $d \neq 0$. When $n$ varies 
over all integers $($resp.\ squares$)$, at least $1/6$ of the elliptic curves in the cubic twist family $E_{d,n}:y^2=x^3+dn^2$  have rank $0$, and at least $1/6$ of the elliptic curves $E_{d,n}$ with good reduction at $2$ have rank $1$. In particular, if the squarefree part of $d$ is congruent to $1 \pmods 4$, then a proportion of at least $\frac{1}{21}2^{r-1}$ of the curves $E_{d,n}$ have rank $1$,
where $r$ is the least residue of $2v_2(d)$ {\em (resp.\ $v_2(d)+1$)} modulo $3$.
\end{theorem}

We prove Theorem~\ref{main2} for the curves $E_{d,n}$ (as $n \in \Z$ varies) via a determination of the average size of the $2$-Selmer group of elliptic curves in any cubic twist family satisfying any finite---or any acceptable infinite---set of congruence conditions. Moreover, the flexibility of our method also  allows us to handle the thin families $E_{d,n^2}$ with $n$ varying, by exploiting the coincidence that $\Sel_2(E_{d,n^2})$ is contained in the ambient space $H^1(\Q, E_{d^2,n}[2])$ of $\Sel_2(E_{d^2,n})$.

Let us say that a subset $\Sigma\subset \Z$ is {\it acceptable} if it is defined by congruence conditions modulo prime powers, where for sufficiently large $p$, the congruence conditions include all integers with $p$-adic valuation at most $1$.  Then we prove the following theorem.

\begin{theorem}\label{theorem:avg2Sel}
Fix $d \neq 0$, and let $\Sigma\subset\Z$ be any acceptable subset. When $n$ varies over elements of $\Sigma$ $($resp.\ squares of elements of $\Sigma)$, the average size of $\Sel_2(E_{d,n})$ is~$3$. 
\end{theorem}
It follows from Theorem~\ref{theorem:avg2Sel} that, when $n$ varies over $\Sigma$ (or over the squares of elements of~$\Sigma$), the average rank of elliptic curves in the cubic twist family  $E_{d,n}$ is bounded; 
indeed, since $\rk\, E_{d,n} \leq 2^{\rk\,E_{d,n}-1} \leq \frac12 \#\Sel_2(E_{d,n})$, it follows that the average rank of $E_{d,n}$ is less than $\frac{3}{2} = 1.5$.


To improve this bound further, we show that the set of $n$ such that $\dim_{\F_2} \Sel_2(E_{d,n})$ (resp.\ $\dim_{\F_2} \Sel_2(E_{d,n^2})$) has a given parity has density $1/2$ and is a union of acceptable sets.  Moreover, we confirm a general prediction of Poonen--Rains~\cite[2.22(c)]{poonenrains} that the average size of the 2-Selmer group is equal to 3 even if one restricts to just those elliptic curves having a given $2$-Selmer parity!


\begin{theorem}\label{rootnumbers}
Fix $d\neq 0$. When $n$ ranges over all integers, the density of the set $\Sigma_{\mathrm{even}}$ $($resp.\ $\Sigma_{\mathrm{odd}})$ of integers $n$ such that $\dim \Sel_2(E_{d,n})$ is even $($resp.\ odd$)$ is $1/2$. Moreover, the average size of $\Sel_2(E_{d,n})$ for $n \in \Sigma_{\mathrm{even}}$ $($resp.\ $n \in \Sigma_{\mathrm{odd}})$ is $3$. 
The same statements hold with $\Sel_2(E_{d,n^2})$ in place of $\Sel_2(E_{d,n})$.

\end{theorem}
\noindent
We prove Theorem~\ref{rootnumbers} by carrying out an analysis of root numbers, and relating root numbers to $2$-Selmer parity via the 
$p$-parity theorem of Dokchitser--Dokchitser~\cite{DDparity}.\footnote{Many important cases of the $p$-parity theorem were proved by Kim \cite{Kim07} and by Nekov\'{a}\v{r}~\cite{Nekovarparity}; in fact, we only use the case $p = 2$ which was proved by Monsky \cite{Monsky}.}  Theorems~\ref{main}--\ref{main2} are then  deduced using Theorem~\ref{rootnumbers} together with the $p$-converse theorem of Burungale--Skinner in the Appendix (see $\S\ref{sec:sketch}$ for more details). 

Theorem~\ref{rootnumbers} also implies the following bounds on (the limsup and the liminf of) the average rank of elliptic curves in cubic twist families:

\begin{theorem}\label{main3}
Fix $d\neq 0$, and let $\Sigma\subset \Z$ be any acceptable subset. The average rank in the cubic twist family of elliptic curves $E_{d,n}$ $($resp.\ $E_{d,n^2}$$)$, $n\in \Sigma$, is at most $4/3$.
Furthermore, if the squarefree part of $d$ is congruent to $1 \pmods{4}$, then the average rank in the cubic twist family of elliptic curves $E_{d,n}$ $($resp.\ $E_{d,n^2})$, $n\in \Z$, is at least $\frac{1}{21}2^{r-1}$,
where $r$ is the least residue of $2v_2(d)$ $($resp.\ $v_2(d) + 1$$)$ modulo $3$.
\end{theorem}

Theorem~\ref{main3} shows, for the first time, the boundedness (and, in many cases, the positivity) of the average rank in arbitrary cubic twist families (both $E_{d,n}$ and $E_{d,n^2}$ as $n\in\Z$ varies). The question of the boundedness of the average rank in twist families of elliptic curves has been studied extensively. The unique sextic twist family was handled by Elkies and the second and third authors \cite{j=0}.  
The quadratic case has been studied by many authors  (see, e.g., \cite{Heath-Brown,RubinSilverberg, Silverbergsurvey,bkls,sd,kane,KMR}), and most recently by Smith \cite{Smith-thesis}, whose work covers most quadratic twist families. Meanwhile, significant progress on the unique quartic twist family was made by Kane and Thorne \cite{KaneThorne}.

\subsection{Variations and related results}

One may ask which positive integers can be expressed as the sum of two {\it positive} rational cubes. 

\begin{theorem}\label{positivecubes}
A positive proportion of positive integers are expressible as the sum of two {positive} rational cubes, and a positive proportion are not.
\end{theorem}
\noindent
Indeed, the same lower bounds on the proportions as in (\ref{maineq}) and (\ref{maineq2}) hold for Theorem~\ref{positivecubes}: if an elliptic curve over $\Q$ has positive rank, then the rational points are dense in the real component of the identity; thus if a non-cube positive integer $n$ is the sum of two rational cubes, then it is also the sum of two positive rational cubes, because the elliptic curve $x^3+y^3=n$ then has positive rank and possesses an arc of real points in the positive~quadrant. 

Our methods also imply the following result about integers that are the product of three rational numbers in arithmetic progression:
\begin{theorem}\label{mainv}
A positive proportion of integers $($resp.\ squares$)$ are expressible as the product of three rational numbers in arithmetic progression, and a positive proportion are not.
\end{theorem}
\noindent
Again, by the same arguments, the same lower bounds on the proportions in Theorem~\ref{mainv} hold as in (\ref{maineq}) and (\ref{maineq2}); and the same lower bounds on the proportions hold for the set of positive integers that are the product of three {\it positive} rational numbers in arithmetic progression.

More generally, our results imply that a positive proportion of integers cannot be represented by any given reducible binary cubic form over $\Q$.

\begin{theorem}\label{mainv2}
Let $f(x,y)$ be any binary cubic form over $\Q$ with a linear factor.  Then a positive proportion of integers $($resp.\ squares$)$ cannot be expressed as $f(x,y)$ with $x,y\in\Q$. Furthermore, if the squarefree part of $\Disc(f)$ is $1\pmods 4$, then a positive proportion of integers $($resp.\ squares$)$ can be expressed as $f(x,y)$ with $x,y\in\Q$. 
\end{theorem}
\noindent
Theorems~\ref{main}--\ref{mainsquare} and Theorem \ref{mainv} are the special cases of Theorem~\ref{mainv2} where we set  $f(x,y)=x^3+y^3$ and $f(x,y)=x(x+y)(x+2y)$, respectively. Theorem~\ref{mainv2} follows from Theorem~\ref{main2}, since the  curve $f(x,y)=n$ is isomorphic to the curve $E_{d,n}$ where $d=16\,\Disc(f)$. 

When $f(x,y)$ is irreducible, the curve $f(x,y)=nz^3$ is not necessarily an elliptic curve, as it then often fails to even have local points. 
Indeed, in the irreducible case, the density of integers $n$, such that $f(x,y)=n$ has points  everywhere locally, is $0$.  More precisely:

\begin{theorem}
Let $f(x,y)$ be an irreducible binary cubic form over $\Q$.  The number of integers $n$ with $|n|<X$ such that the curve $f(x,y)=n$ $($resp.\ $f(x,y) = n^2)$ has points  everywhere locally is on the order of  either $X/\log^{1/3}X$ or $X/\log^{2/3}X$, depending on whether $\Disc(f)$ is or is not a square.
\end{theorem}

\noindent 
Indeed, for sufficiently large $p$, the curve $f(x,y)=n$ (resp.\ $f(x,y) = n^2$) fails to have a solution over $\Q_p$ precisely when
$v_p(n)$ is not a multiple of 3 and  $f(x,y)$ is irreducible (mod $p$).   
By the Chebotarev density theorem, the density of primes $p$ such that $f(x,y)$ is irreducible (mod $p$) is  $\frac23$ or $\frac13$, depending on whether the Galois group of $f$ is $C_3$ or $S_3$. The theorem then  follows from standard counting results of Selberg--Delange type (e.g., Theorem \ref{odoni}).

In cases where it does have points locally, it is natural to ask how often the curve $f(x,y)=n$ 
has a global rational point. 
The genus one curve $f(x,y)=n$
naturally corresponds to an element of the ``$\sqrt{-3}$-Selmer group'' of $E_{d,n}$, where again $d=16\,\Disc(f)$.
Indeed, the cubic twist families $E_{d,n}:y^2=x^3+dn^2$ have traditionally been studied via 
the Selmer groups associated to a natural isogeny $\sqrt{-3}:E_{d,n}\to E_{-27d,n}$ defined over $\Q$. 

This approach allows one to determine the $3$-Selmer group of any such curve. In 1879, Sylvester \cite[\S 2]{Sylvester} (see also Selmer~\cite{Selmer}) used this $\sqrt{-3}$-descent to show that the 3-Selmer rank of $x^3+y^3=p$ for $p$ a prime is   $0$ if $p\equiv 2,5\pmods9$  and is $1$ if $p\equiv 4,7,8\pmods9$. 
This proved that primes $p\equiv 2,5\pmods9$ are not the sum of two cubes, and led Sylvester to conjecture that primes $p\equiv 4,7,8\pmods9$ {are} the sum of two cubes, a proof of which was recently announced by Kriz~\cite{kriz}; the strongest known results to date in this direction are due to Dasgupta and Voight~\cite{DV18} and~Yin~\cite{Yin22}.

However, the $\sqrt{-3}$-Selmer group is not so useful in studying the cubic twist families $E_{d,n}$ for general integers~$d$, 
as the size of the $\sqrt{-3}$-Selmer typically grows with the number of prime factors of $n$. 
Indeed, for any fixed nonsquare $d\in \Z$, we show that the average number of  curves $f(x,y)=n$ locally having a point 
{grows} with~$|n|$. In particular, we prove:

\begin{theorem}\label{theorem:3Sel}
Suppose $d\in\Z$ is not a square.  Then  $\avg_n \#\Sel_3(E_{d,n}) = \infty$.
\end{theorem}
\noindent 

Since the average rank of $E_{d,n}$ remains bounded by Theorem~\ref{main2} as $n$ varies, we conclude that for most cubic twist families, the average size of $\Sha(E_{d,n})[3]$ is unbounded. 


\begin{corollary}\label{shacor}
Suppose $d\in\Z$ is not a square.  Then $\avg_n \#\Sha(E_{d,n})[3] = \infty$.   
\end{corollary}
In \cite{ABScubics}, we proved a similar but less extreme result for the entire twist family $E^k \colon y^2 = x^3 + k$. There we found that, for any integer $r$, a positive proportion of the curves $E^k$ have $\#\Sha(E^k)[3] > 3^r$, but the average size of $\Sha(E^k)[3]$ is still bounded.\footnote{In the special case where $d$ is a square, our method for proving Theorem~\ref{theorem:3Sel} does not apply and we are not sure what to expect for the average size of $\Sha(E_{d,n})[3]$.} 

Corollary \ref{shacor} can be reformulated in more concrete terms as follows:

\begin{corollary}\label{hasse}
Suppose $d\in\Z$ is not a square. When binary cubic forms $f(x,y)$ over $\Z$ of discriminant $dn^2$ are ordered by $|n|$, $100\%$ of the plane curves $z^3=f(x,y)$ that are locally soluble fail the Hasse principle.
\end{corollary}

\subsection{A higher-dimensional example}

We prove Theorem~\ref{theorem:avg2Sel}  in the more general context of cubic twist families of abelian varieties having a $\mu_3$-action, where $\mu_3$ is the group of third roots of unity; see Theorem \ref{thm:abvarsintro} for more details. 

As an example of this more general result, let $C \colon y^3 = x^4 + ax^2 + b$ be a smooth plane quartic curve with $\mu_6$-action, where $a,b \in \Q$. The quotient $C/\mu_2$ is the elliptic curve $E \colon y^3 = x^2 + ax + b$.  Let $A = \ker(\mathrm{Jac}(C) \to E)$ be the Prym abelian surface associated to the double cover. For each non-zero integer $n$, consider the cubic twist $C_n \colon ny^3 = x^4 + ax^2 + b$ and the associated Prym~surface~$A_n$. The natural polarization $\lambda_n \colon A_n \to \widehat{A}_n$ is a $(2,2)$-isogeny. 
Our more general result determines the average size of the Selmer group $\Sel_{\lambda_n}(A_n)$, from which we can deduce information about the average ranks of the abelian varieties $A_n$ and $\Jac(C_n)$. Specifically, we prove:

\begin{theorem}\label{thm: bielliptic Picard curves}
As $n$ varies ordered by $|n|$, the average size of $\Sel_{\lambda_n}(A_n)$ is equal to $3$. The average rank of $A_n(\Q)$ is at most $3$ and the average rank of $\Jac(C_n)(\Q)$ is at most $13/3$. 
\end{theorem}



For more results on the arithmetic of bielliptic Picard curves see \cite{lagashnidman,ShnidmanWeissSexticTwist}.
See also the recent work of~Laga~\cite{laga}, who proves an average rank bound in a universal family of Prym surfaces.

\section{Sketch of the proofs of Theorems~\ref{main}--\ref{main3}}\label{sec:sketch} 

\subsubsection*{1. Parametrization of $2$-Selmer elements of elliptic curves in cubic twist families}\label{subsec:param} To prove Theorems~\ref{main}--\ref{main3}, we make use of, and extend, a parametrization of $2$-Selmer elements in the family of elliptic curves $E(a_1,a_3) :y^2+a_1xy+a_3y=x^3$ via pairs $(F_1,F_2)$ of integral binary cubic forms, as studied by the second author and Ho~\cite{BhargavaHo,BhargavaHo2}.  The curves $E = E(a_1,a_3)$ form the universal family of elliptic curves having a marked rational 3-torsion point. Elements of $\Sel_2(E)$ can be represented by pairs $(C,D)$, where $C$ is a genus 1 curve that is {\it locally soluble} (i.e., $C(\Q_p) \neq \emptyset$ for all $p$, $\Pic^1(C) \simeq E$),  and $D$ is a degree $2$ divisor on $C$.     

Let $V$ denote the space of pairs of integer-matrix binary cubic forms, i.e., pairs $(F_1,F_2)$ where $F_1(x,y)=r_1x^3+3r_2x^2y+3r_3xy^2+r_4y^3$ and $F_2(x,y)=r_5x^3+3r_6x^2y+3r_7xy^2+r_8y^3$. We say that a pair $(F_1,F_2) \in V(\Q)$ is {\it locally soluble} if the genus one curve $z^2=\Disc_{x,y}( w_1F_1(x,y)- w_2F_2(x,y))$ has points everywhere locally.\footnote{Here, $\Disc_{x,y}$ is $-1/27$ times the usual polynomial discriminant of the binary form in the variables $x$ and $y$, following the conventions used in \cite{BhargavaHo,BhargavaHo2}. This curve lives in the  weighted projective space $\P_{2,1,1}= \P_{z,w_1,w_2}$.} Note that the isomorphism class of this curve is invariant under the action of $\SL_2^2(\Q)$ on $V(\Q)$.
There are two polynomial invariants $A_1$ and $A_3$ for the action of $\SL_2^2$ on $V$, having degrees $2$ and $6$, respectively (for explicit formulas, see (\ref{A1formula}) and (\ref{A3formula})).

\begin{theorem}[{\cite[Thm.\ 4.2]{BhargavaHo2}}]\label{2selpar1}
The elements in the $2$-Selmer group $\Sel_2(E)$ of $$E:y^2 + a_1 xy + a_3 y =x^3$$are in bijection with $\SL_2^2(\Q)$-equivalence classes of locally soluble 
pairs of {\it integral} binary cubic forms having invariants $A_1=M a_1$ and $A_3=M^3 a_3$ for some fixed nonzero integer $M$.
\end{theorem}
 
Restricting Theorem~\ref{2selpar1} to the case where $A_1=a_1=0$ yields the following corollary. 

\begin{corollary}\label{2selpar2}
The elements in the $2$-Selmer group $\Sel_2(E_n)$ of $E_n:y^2 + ny = x^3$ are in bijection with $\SL_2^2(\Q)$-equivalence classes of locally soluble pairs of {\it integral} binary cubic forms satisfying $A_1=0$ and $A_3 =M^3n$ for some fixed nonzero integer $M$.
\end{corollary}

The family $E_n$ above is isomorphic to the family $E_{16,n} \colon y^2 = x^3 + 16n^2$ from the introduction.  To handle general cubic twist families $E_{d,n}$, 
 we prove a generalization of 
 Corollary \ref{2selpar2} for $E_n$ that  gives some flexibility in the local conditions used to define a Selmer group inside $H^1(\Q, E_n)[2]$.  We then use the fact that $E_{d,n}:y^2=x^3+dn^2$ is the $d$-th quadratic twist of $E_{2d^2n}$, and hence we have isomorphisms $E_{d,n}[2] \simeq E_{2d^2n}[2]$ and $H^1(\Q,E_{d,n}[2]) \simeq H^1(\Q,E_{2d^2n}[2])$; this enables us to handle all cubic twist families $E_{d,n}$. We prove:

\begin{theorem}\label{thm:2seltwist}
The elements in the $2$-Selmer group $\Sel_2(E)$ of $E_{d,n}:y^2  =x^3+dn^2$ are in bijection with $\SL_2^2(\Q)$-equivalence classes of integral pairs of binary cubic forms satisfying certain congruence conditions with $A_1=0$ and $A_3 =M^3n$ for some fixed nonzero integer $M$ depending only on $d$. 
\end{theorem}
\noindent 

We prove this result in the more general context of cubic twist families of abelian varieties 
$A$ admitting a $\mu_3$-action (see Section \ref{sec:param}).
We assume $A$ carries an ample line bundle $\L$ fixed by the $\mu_3$-action and such that the corresponding polarization $\lambda \colon A\to \widehat{A}$ has kernel of order $4$.  For each such $A$, we show that there exists an elliptic curve $E = E(0,a_3)$ and an isomorphism of central extensions $\Theta(L) \simeq \Theta(\O_E(2\infty))$,  where $\Theta(L) = \Aut(L/A)$ is Mumford's theta group \cite[\S23]{MumfordAV}. 
We use this isomorphism to view elements of $\Sel_{\lambda_n}(A_n)$ as elements of $H^1(\Q, E_n[2])$, which we then show correspond to orbits of pairs of integral binary cubic forms satisfying certain congruence conditions, in a manner that is analogous to the statement of Theorem~\ref{thm:2seltwist}.  

\subsubsection*{2. The number of $\SL_2(\Z)^2$-orbits on the invariant quadric $A_1=0$ with $|A_3|<X$}

Theorem~\ref{thm:2seltwist} shows that in order to determine the average size of the $2$-Selmer group in a family of cubic twists $E_{d,n}$, where $d\in\Z$ is fixed and $n$ varies over an acceptable set of integers, we must solve a certain counting problem. Namely, we must estimate the (weighted) number of $G(\Z)$-orbits on $V(\Z)$ lying on the quadric $Y = \{A_1 = 0\} \subset V$ such that $|A_3| < X$ and with $A_3$ satisfying certain congruence conditions. Here, we take $G := \SL_2^2/\mu_2$ to be the group acting faithfully on $V$.

In this direction, we prove the following result. Let us  say that $(F_1,F_2) \in Y(\Z)$ is {\it irreducible} if the corresponding binary quartic form $\Disc(w_1F_1- w_2F_2)$ in $w_1$ and $w_2$ does not have a linear factor over $\Q$.  Non-trivial elements of the $2$-Selmer group correspond to irreducible orbits, so it suffices to count the latter.
 
\begin{theorem}\label{mainctgthm}
Let $\Sigma \subset \Z$ be an acceptable subset, let $S(\Sigma)$ denote the set of $y\in Y(\Z_p)$ such that $A_3(y)\in \Sigma$, and let $S_p(\Sigma)$ denote the $p$-adic closure of $S(\Sigma)$ in $Y(\Z_p)$.  Let $N(S(\Sigma);X)$ denote the number of irreducible $G(\Z)$-orbits of $y \in Y(\Z)$ such that $A_3(y) \in \Sigma$ and $|A_3(y)| < X$. Then
\begin{equation}
N(S(\Sigma);X)
   = 
     X  \cdot 
 \int_{\scriptstyle{y\in G(\Z) \backslash Y(\R)}\atop\scriptstyle{|A_3(y)|<1}}
dy \cdot \prod_{p}
  \int_{y\in S_p(\Sigma)}\,dy\,+\,O_\Sigma(X^{1-c})
\end{equation}
for an absolute constant $c > 0$;   here, $dy$ is the $G(\R)$-invariant $($resp.\ $G(\Z_p)$-invariant$)$ measure on~$Y(\R)$ $($resp.\ $Y(\Z_p))$ given by $dr_2\,dr_3\cdots dr_8/(\partial A_1/\partial r_1)$,
  where $r_1,\ldots,r_8$ are the coordinates on~$V$.
\end{theorem}

\pagebreak
To prove Theorem \ref{mainctgthm}, we wish to count integer points on the quadric $Y$ within a fundamental domain for the action of $G(\Z)$ of $V(\R)$. To accomplish this, we combine geometry-of-numbers techniques as in \cite{Bhargavadensityofquinticfields,BS1} with the circle/delta  method as in \cite{HB}. This combination of techniques was first studied by Ruth~\cite{Ruth} in his thesis, while a more generally applicable method for this counting problem was sketched by the first author in \cite{leventthesis,leventcircle}. This paper represents the first published account of the method; here, we fill in the details and also take the method considerably further in order to carry out the desired asymptotic count (both upper and {\it lower} bounds) and thereby prove Theorem~\ref{mainctgthm}. 
We expect that the method, now developed fully and in this generality, will be useful in a number of other contexts. 



One novel  aspect in our use of the circle method is that we prove a new interpretation of the singular integral and singular
series arising in the circle method in terms of real and $p$-adic integrals with respect to Haar measure on $G$.
Specifically, we prove that the $G$-invariant measure $dy$ (also known as the Gelfand--Leray form) on the quadric $Y$ can be re-expressed as $c\: dA_3\, dg$ for some constant $c$, where $dg$ denotes a Haar measure on $G$. This allows us to compute the global constants much more easily and directly.

We also prove a generalization of  Theorem~\ref{mainctgthm} where we allow weighted counts, where the weights are defined by suitable congruence conditions. The local densities for such weighted counts are then computed in terms of integrals with respect to the measure $dy$, and re-expressed in terms of Haar measure on $G$, using our aforementioned interpretations for the singular integral and singular series. 

These expressions for weighted counts in terms of integrals with respect to the real and $p$-adic Haar measures on $G$ play a key role in the application to average sizes of Selmer groups.

\subsubsection*{3. The average size of the $2$-Selmer group of elliptic curves in cubic twist families}

For applications to Selmer groups, we choose the weight function to be the characteristic function of the locally soluble orbits, weighted appropriately to account for the number of $G(\Z)$-orbits in a given $G(\Q)$-orbit.  The corresponding local densities then take a particularly nice form when expressed in terms of a Haar measure on $G$, allowing us to combine Theorems \ref{thm:2seltwist} and \ref{mainctgthm}. 
We thereby obtain an upper bound of 3 on the average size of the 2-Selmer group in these cubic twist families. 

To determine an exact average instead of just an upper bound, we must prove a uniformity estimate for the number of elements with $A_1=0$ and  $A_3$ divisible by the square of a large prime. This can be a difficult problem in general, and in fact its analogue for the larger family $y^2+a_1 xy+a_3y=x^3$ is not known \cite[\S8]{BhargavaHo2}. We prove a suitable estimate in the case $A_1=0$ by using a geometric sieve for quadrics due to Browning and Heath-Brown~\cite{browning-heath-brown} in those cases where $A_3$ is a multiple of $p^2$ for ``mod $p$ reasons''; we then use the condition $A_1=0$ to construct an $A_3$-preserving transformation that changes the condition that $A_3$ is a multiple of $p^2$ for ``mod~$p^2$ reasons'' to being so for ``mod~$p$ reasons'', thereby reducing to the cases already handled.
Applying the uniformity estimate and a sieve then shows that the average size of the $2$-Selmer group in any cubic twist family $E_{d,n}$ is $3$, even when $n$ varies within acceptable sets in $\Z$. 

To handle the thinner families of the form $E_{d,n^2}$, we use the same parameterization as in \S 2.\ref{subsec:param} but replace $d$ with $d^2$, using the fact that $H^1(\Q, E_{d,n^2}[2])\simeq H^1(\Q, E_{d^2,n}[2])$. We prove integrality of these Selmer elements (with a suitable change in $M$) even under the corresponding changes in local conditions.

We will in fact prove the following result for general families of cubic twists of abelian varieties:
\begin{theorem}\label{thm:abvarsintro}
Let $A$ be an abelian variety over $\Q$ with a  degree $4$ polarization $\lambda  \colon A \to \widehat{A}$ induced by a symmetric line bundle $\L \in \Pic(A)$.  Suppose the pair $(A,\L)$ admits a fixed-point-free $\mu_3$-action.
For each nonzero $n \in \Z$, let $\lambda_n \colon A_n \to \widehat{A}_n$ be the cubic twist of $\lambda$.  Let $\Sigma \subset \Z$ be any acceptable set. Then, as $|n| \to \infty$, the average size of $\#\Sel_{\lambda_n}(A_n)$ over $n \in \Sigma$ is $3$.  
\end{theorem}

As a special case, we deduce Theorem~\ref{theorem:avg2Sel}. The generality of Theorem \ref{thm:abvarsintro} allows us to treat both cases of Theorem~\ref{theorem:avg2Sel} on the same footing, since they represent the same cubic twist family, but with  {inverse} $\mu_3$-actions.\footnote{The parameter $n^2 \equiv n^{-1} \in \Q^\times/\Q^{\times 3} \simeq H^1(\Q, \mu_3) \to H^1(\Q, \Aut(A_1))$ becomes $n$ when using the inverse action.}

\subsubsection*{4. Analysis of root numbers}
Theorem~\ref{theorem:avg2Sel} alone is not sufficient to deduce Theorems~\ref{main}--\ref{main2}. 
To proceed further, we carry out an analysis of the root numbers $w_{d,n} \in \{\pm 1\}$ of the elliptic curves $E_{d,n}$. Recall that $w_{d,n}$ is the sign appearing in the functional equation for the $L$-function of $E_{d,n}$, and the parity conjecture predicts that $(-1)^{\rk\, E_{d,n}(\Q)} = w_{d,n}$.  We prove:

\begin{theorem}\label{thm:rootnumberequidistribution}
Fix $d$, and let $\Sigma$ be any acceptable subset of $\Z$ defined by prime-to-$3$ congruence conditions. Then the root numbers $w_{d,n}$ and $w_{d,n^2}$ are equidistributed as $n \in \Sigma$ goes to infinity. In other words,
$$\sum_{n\in \Sigma: |n|\leq X} w_{d,n} = o_{d,S}(X)$$ and similarly $$\sum_{n\in \Sigma: |n|\leq X} w_{d,n^2} = o_{d,S}(X).$$
\end{theorem}
We will in fact prove Theorem \ref{thm:rootnumberequidistribution} for more general $\Sigma$, namely those $\Sigma$ defined by prime-to-$3$ congruence conditions which have a natural density. Moreover, for $\Sigma$ defined by finitely many congruence conditions mod $m$ we will obtain a bound of shape $\ll d^{O(1)} m^{O(1)} X^{1-c}$ for an explicit absolute constant $c\in \R^+$.

From this equidistribution result, we see that the parity conjecture would imply that $50\%$ of twists $E_{d,n}$ have odd rank and $50\%$ of twists $E_{d,n}$ have even rank (with $d$ fixed and $n$ varying), even if we impose certain congruence conditions on $n$.  However, the prime-to-3 hypothesis in the theorem is crucial, since one can construct acceptable subsets $\Sigma$ such that the root number $w_{d,n}$ is constant for $n \in \Sigma$.  Indeed, we prove:

\begin{theorem}\label{thm:rootnumbercountableunion}
Fix $d$ and let $\Sigma$ be an acceptable subset of $\Z$. The set of $n\in \Sigma$ such that $E_{d,n}$ $($resp.\ $E_{d,n^2})$ has a given root number is a countable union of acceptable sets.
\end{theorem}

These theorems are proved using an explicit formula for the root numbers $w_{d,n}$ due to Varilly-Alvarado (based on work of Rohrlich) \cite{VA11}. Aside from multiplicative factors at $2$ and $3$, the root number of $y^2=x^3-432n^2$ is equal to  $(-1)^{\omega'(n)}$ where $\omega'(n)$ is the number of primes $p \equiv 2\pmods{3}$ dividing $n$.  Theorem~\ref{thm:rootnumberequidistribution} follows from suitable asymptotic control of $\sum_n (-1)^{\omega'(n)}$ over arithmetic progressions.

Theorem~\ref{rootnumbers} follows from combining  Theorem~\ref{theorem:avg2Sel} with Theorems \ref{thm:rootnumberequidistribution} and \ref{thm:rootnumbercountableunion}.

\subsubsection*{5. The proportion of curves with rank $0$ and rank $1$ in cubic twist families}

The parity conjecture is still open, so our equidistribution results for root numbers do not immediately imply anything about ranks of elliptic curves. However, we may instead apply the $p$-parity theorem of  Dokchitser--Dokchitser \cite{DDparity}. 

\begin{theorem}[$p$-parity]
Let $E/\Q$ be an elliptic curve and let $w(E)$ be its root number. Then for every prime $p$, we have
\[w(E) = (-1)^{\dim_{\F_p}\!\Sel_p(E) + \dim_{\F_p}\!E[p](\Q)}.\]
\end{theorem} 
It is easy to see that in any cubic twist family, we have $E_{d,n}[2](\Q) = 0$ for $100\%$ of integers $n$. For such $n$, the case $p = 2$ of the $p$-parity theorem reads $w_{d,n} = (-1)^{\dim_{\F_2}\!\Sel_2(E_{d,n})}$. Across the root number $+1$ curves, the $2$-Selmer group size will be an even power of $2$ (i.e., $1$, $4$, $16$, etc.), while for the root number $-1$ curves, the $2$-Selmer group size will be an odd power of $2$ (i.e., $2$, $8$, $32$, etc.).  Since the average size of $\Sel_2(E_{d,n})$ is $3$ across those $E_{d,n}$ having root number $+1$,  we conclude that at least $1/3$ of these curves~must have $2$-Selmer rank $0$  (note that $3$ is the average of $1,4,4$). Since $\Sel_2(E) = 0$ implies $\rk\, E(\Q) = 0$, we deduce that at least $1/6$ of elliptic curves in every cubic twist family have rank $0$.

Similarly, since the average size of $\Sel_2(E_{d,n})$ is $3$ across those $E_{d,n}$ having root number $-1$,  we conclude that at least $5/6$ of these curves~must have $2$-Selmer rank $1$  (note that $3$ is the average of $2,2,2,2,2,8$). The identical arguments also apply to $E_{d,n^2}$ as $n$ varies. We therefore deduce:

\begin{corollary}
Fix $d\neq 0$. Then at least $5/12$ of the elliptic curves $E_{d,n}$ $($resp.\ $E_{d,n^2})$ have $2$-Selmer rank $1$.
\end{corollary}

Moreover, we prove a similar result as $n$ varies over sets of integers defined by congruence conditions that are prime to $3$. We then apply the following $p$-converse theorem of Burungale and Skinner (Corollary \ref{cor:pcm2} of the Appendix).
\begin{theorem}
Let $p$ be prime and let $E/\Q$ be a CM elliptic curve with supersingular reduction at $p$. If $\#\Sel_p(E) = p$ and the localization map $\Sel_p(E) \to E(\Q_p)/pE(\Q_p)$ is injective, then $\rk\,E(\Q) = 1$. 
\end{theorem}
This allows us to prove that a positive proportion of cubic twists $E_{d,n}$ (resp.\ $E_{d,n^2}$) with good reduction at~$2$ have rank~$1$.  We handle the extra $2$-adic condition by proving an appropriate equidistribution theorem for $2$-Selmer elements under the $2$-adic localization map (Theorem \ref{thm:equidistribution}), though this causes our lower bound on the proportion of rank $1$ twists to drop from $5/12$ to $1/6$.  

This completes the sketch of the proofs of Theorems~\ref{main}--\ref{main2} and \ref{main3}.  If the theorem of Burungale--Skinner is eventually generalized to cover any CM  elliptic curve with \emph{potentially} supersingular reduction satisfying $\#\Sel_p(E) = p$  (so that bad reduction is allowed and without any hypothesis on the localization map), then our results would imply that at least $5/12$ of cubic twists have rank~$1$ in any cubic twist family of elliptic curves, and thus that at least $5/12$ of integers (resp.\ square integers) are the sum of two rational cubes.

\section{Pairs of binary cubic forms and $2$-Selmer elements in cubic twist families}\label{sec:param}

In \cite[\S 6.3.2]{BhargavaHo}, the second author and Ho  gave a functorial bijection between  orbits of pairs of binary cubic forms and isomorphism classes of genus $1$ curves with extra data. We review this parametrization and then specialize it to certain genus one curves with $j$-invariant $0$. We then generalize this specialization to arbitrary cubic twist families of polarized abelian varieties. We use this parametrization to classify certain $2$-Selmer elements in terms of orbits of pairs of binary cubic forms and show that orbits corresponding to $2$-Selmer elements have integral representatives.

\subsection{Pairs of binary cubic forms over a field}

Let $F$ be a ground field of characteristic neither $2$ nor $3$. We consider the two groups $\widetilde G = \GL_2^2$ and $G = \SL_2^2/\mu_2$, where we view $\mu_2 \subset \SL_2^2$ via the diagonal scalar embedding. Both groups act on  the space $V = F^2\otimes\Sym_3F^2$ of pairs of integer-matrix binary cubic forms, with the first copy of $\GL_2$ (resp.\ $\SL_2$) acting on $F^2$ by the standard representation and the second copy acting by the symmetric cube of the standard representation.  It is known that the ring of polynomial invariants for the action of $G$ on $V$ is freely generated by two invariants $A_1$ and $A_3$, of degrees $2$ and $6$ respectively \cite{BhargavaHo}; see the next section for some explicit formulas.  These are only {\it relative} invariants for the action of $\widetilde G$.  Let $Y \subset V$ be the quadric hypersurface defined by the equation $A_1 = 0$. 

For each $n \in F^\times$, let $Y(F)_n$ be the subset of $y \in Y(F)$ having $A_3$-invariant $n$.  The group $G(F)$ acts on both $Y(F)$ and $Y(F)_n$.  We will make use of the following bijection due to the second author and Ho~\cite[Thm.~2.3]{BhargavaHo}, which relates the $G(F)$-orbits on $Y(F)_n$ to isomorphism classes of pairs $(C,L)$, where $C$ is a genus one curve whose Jacobian $\Pic^0(C)$ is the elliptic curve $E_n \colon y^2 + ny = x^3$, and where $L$ is a degree $2$ line bundle on~$C$. Such a genus one curve $C$ is automatically an $E_n$-torsor \cite[\S X]{SilvermanAEC}.

\begin{theorem}\label{theorem:2Selparam}
Fix $n \in F^\times$. There is a natural bijection between $G(F)$-orbits on $Y(F)_n$ and isomorphism classes of pairs $(C, L)$, where $C$ is an $E_n$-torsor and $L$ is a degree two line bundle on~$C$.  If $v \in Y(F)_n$ corresponds to $(C, L)$, then $\mathrm{Stab}_{\widetilde G}(v) \simeq \Aut(C, L)$ and $\mathrm{Stab}_G(v) \simeq E_n[2]$. 
\end{theorem}
\begin{proof}
This follows from the more general result cited earlier (see also \cite[Thm.\ 3.1]{BhargavaHo2}). Given a pair of cubic forms $v = (F_1,F_2) \in V(F)$, the corresponding curve $C$ is the hyperelliptic curve $z^2 = \Disc(w_1F_1(x,y) -  w_2F_2(x,y)),$ and the line bundle $L$ is the pullback of $\O_{\P^1}(1)$ along the map $z \colon C \to \P^1$. The Jacobian of $C$ is the elliptic curve $E \colon y^2 + A_1(v)xy + A_3(v)y = x^3$. Setting $A_1(v) = 0$ and $A_3(v) = n$, we obtain an $y^2 +ny= x^3$.   
\end{proof}

For each $n \in F^\times$, let $v_n \in Y(F)_n$ be a pair of cubic forms corresponding under the bijection of Theorem \ref{theorem:2Selparam} to the pair $(E_n, L_n)$, where $L_n = \O_{E_n}(2\infty)$ is the line bundle corresponding to the divisor $2\infty = 2[0 \colon 1 \colon 0]$.  We refer to the $G(F)$-orbit of $v_n$ as the {reducible orbit} having $A_3$-invariant~$n$.  Explicitly, we may take $v_n$ to be the pair $(3xy^2,x^3 + ny^3).$ 

The stabilizer $\mathrm{Stab}_{\widetilde G}(v_n)$ is isomorphic to Mumford's theta group $\Theta(L_n):= \Aut(E_n,L_n)$ of automorphisms of the 
total space of $L_n$ lifting those of $E_n$. This is an infinite noncommutative group scheme, best described via the exact sequence
\begin{equation}\label{eq:thetasequence}
    0 \to \G_m \to \Theta(L_n) \to E_n[2] \to 0. 
\end{equation}
The subgroup $\G_m$ corresponds to the automorphisms of $L_n$ given by scalar multiplication in each fiber, and the map $\Theta(L_n) \to E_n[2]$ records the underlying automorphism of $E_n$, which is necessarily translation by a point in $E_n[2]$.  

Applying the long exact sequence in cohomology to (\ref{eq:thetasequence}) gives the {\it obstruction map} \[\mathrm{ob} \colon H^1(F,E_n[2]) \to H^2(F, \G_m).\] 
This is only a map of pointed sets, despite the fact that both source and target are abelian groups. By Hilbert's Theorem $90$, we have $H^1(F, \G_m) = 0$, and hence the kernel of the obstruction map is $H^1(F, \Theta(L_n))$, viewed as a subset of $H^1(F, E_n[2])$.

\begin{corollary}\label{cor:thetaparam}
There is a functorial bijection between the $G(F)$-orbits on $Y(F)_n$ and the pointed set $H^1(F, \Theta(L_n))$.  Under this bijection, the reducible orbit corresponds to the trivial class. 
\end{corollary}
\begin{proof}
By arithmetic invariant theory \cite[Prop.\ 1]{BG}, the set of $G(F)$-orbits on $Y(F)_n$ is in bijection with the kernel of 
$H^1(F, \mathrm{Stab}_G(v_n)) \to H^1(F, \SL_2^2/\mu_2)$. Since $H^1(F, \SL_2^2) = 0$, we have  $H^1(F, \SL_2^2/\mu_2) \simeq H^2(F, \mu_2)$. Thus the $G(F)$-orbits on $Y(F)_n$ are in bijection with the kernel of $H^1(F, E_n[2]) \to H^2(F, \mu_2)$, through which the obstruction map factors. 
Alternatively, we may apply the theory to $\widetilde G$ and use that $\Theta(L_n) \simeq \mathrm{Stab}_{\widetilde G}(v_n)$ and $H^2(F, \widetilde G) = 0$.  
\end{proof}
\subsection{Parametrizing $A[\lambda]$-torsors via pairs of binary cubics}\label{subsec:A-torsors}

We next generalize Theorem \ref{theorem:2Selparam} to more general cubic twist families of elliptic curves.  Since it is not any harder, we will  allow abelian varieties of higher dimension as well. We will make explicit the case of elliptic curves, for the benefit of readers who are not comfortable with the language of abelian varieties, but also to emphasize that even in this case we are doing something new. Indeed, the particular cubic twist family $E_n \colon y^2+ ny = x^3$ is special in the sense that it admits the rational point $(0,0)$ of order 3 (which is crucial for the parametrization in \cite[\S 6.3.2]{BhargavaHo}),  whereas in a general cubic twist family of elliptic curves $E_{d,n} \colon y^2 = x^3 + dn^2$, with $d$ a fixed non-square, the curve $E_{d,n}$ has no rational $3$-torsion point.  An easy computation shows that  $E_n$ is isomorphic to $E_{16,n}$. 

In this paper, a {\it polarized abelian variety over $F$} is a pair $(A, \L)$, where $A$ is an abelian variety over $F$ and $\L$ is an ample line bundle on $A$. 

\begin{remark}{\em 
Every abelian variety $A$ admits a dual abelian variety $\widehat{A} = \Pic^0(A)$ which parametrizes algebraically trivial line bundles on $A$.  If $A = E$ is an elliptic curve, then algebraically trivial is equivalent to degree 0, and there is an isomorphism $E \simeq \widehat{E} = \Pic^0(E)$ defined by $P \mapsto \O_E(P - \infty)$.  In higher dimension $A$ and $\widehat A$ need not be isomorphic, but any ample line bundle $\L$ on $A$ induces a map $\lambda_\L \colon A \to \widehat{A}$, sending a point $P$ to $t_P^*\L \otimes \L^{-1}$, where $t_P \colon A \to A$ is translation by $P$. The map $\lambda_\L$ is called the {\it polarization} associated to $(A, \L)$.  
}
\end{remark}

If $(A,L)$ is a polarized abelian variety, we write $\Aut(A, \L)$ for the group of automorphisms $\alpha \in \Aut(A)$ such that $\alpha^*\L \simeq \L$.    Let $\mu_3$ be the group of third roots of unity.  A {\it $\mu_3$-action} on $(A,\L)$ is an inclusion of $F$-group schemes $\iota \colon \mu_3 \hookrightarrow \Aut(A,\L)$.  The action has {\it isolated fixed points} if for each nontrivial $\zeta \in \mu_3$, the endomorphism $1 - \iota(\zeta)$ has finite kernel, or in other words, is an isogeny.  This condition is automatically satisfied if $A$ is simple, e.g., if $A$ is an elliptic curve. 

\begin{example}{\em
If $A = E$ is an elliptic curve with $\mu_3$-action, then $E$ has a model $E\colon y^2 = x^3 + d$ for some $d \in F^\times$, and the $\mu_3$-action is given by $(x,y) \mapsto (\zeta x,y)$.  The line bundles $\O_E(k\infty)$ are preserved by this action since $\infty$ is sent to $\infty$. The kernel of $1 - \iota(\zeta)$ is the order three group generated by $(0,\pm\sqrt{d})$.}
\end{example}

We now suppose that $(A, \L)$ admits a $\mu_3$-action $\iota$.  For each $n \in F^\times$, we define $(A_n, \L_n)$ to be the twist of $(A,\L)$ by the cocycle  $G_F \to \Aut(A, \L)$ sending $g \mapsto \iota(g(\sqrt[3]{n})/\sqrt[3]{n})$. The isomorphism class of $(A_n,\L_n)$ depends only on the image of $n$ in $F^\times/F^{\times 3} \simeq H^1(F, \mu_3)$.  We have $\dim_F H^0(A, \L) = \dim_F H^0(A_n, \L_n)$ for all $n$, and the polarizations $\lambda_n  \colon A_n \to \widehat{A}_n$ all have degree $(\dim_FH^0(A, \L))^2$.  

We assume for the rest of the section that $(A, \L)$ has a $\mu_3$-action with isolated fixed points and that $\dim_F H^0(A, \L) = 2$.  We also assume that $\L$ is symmetric, in the sense that $[-1]^*\L \simeq \L$.  The groups $A_n[\lambda_n](\bar F)$ are  isomorphic to $(\Z/2\Z)^2$, so that $\lambda_n \colon A_n \to \widehat{A}_n$ is a family of $(2,2)$-isogenies. 
Let $\lambda = \lambda_1 \colon A \to \widehat{A}$ be the initial polarization. 

\begin{example}{\em 
If  $A = E \colon y^2 = x^3 + d$ is an elliptic curve, we may take $L = \O_E(2\infty)$. After making the identification $E \simeq \widehat{E}$, the polarization $\lambda \colon E \to \widehat{E}$ becomes the multiplication-by-minus-two map $[-2] \colon E \to E$. Indeed, $\lambda$ sends $P$ to the divisor $2(-P) -2\infty \sim 2(\infty - P)$. The twist $E_n$ is the elliptic curve $E_{d,n} \colon y^2 = x^3 + dn^2$ from the introduction. 
}
\end{example}

\begin{lemma}\label{lem:muaction}
There exists $d \in F^\times$ such that  $A[\lambda] \simeq \mathrm{Stab}_G(y_d)$.
\end{lemma}

\begin{proof}
The $\mu_3$-action on $(A,\L)$ induces a $\mu_3$-action on $A[\lambda]$. The hypothesis that $1- \zeta$ is an isogeny implies that the action of $\mu_3$ on $A[\lambda] \setminus 0$ is simply transitive. Indeed, the degree of $1-\zeta$ is a power of $3$, so $1 - \zeta$ cannot annihilate any point of order 2.  We also have $\zeta^g P^g = (\zeta P)^g$ for all $g \in \Gal(\bar F/F)$, $\zeta\in \mu_3(\bar F)$, and $P \in A[\lambda](\bar F)$. From this we see that $F(\mu_3)$ is contained in $F(A[\lambda])$, which is a Galois extension of $F$ whose Galois group is a subgroup of $S_3 \simeq \GL_2(\F_2)$. It follows from elementary Galois theory (see \cite[Lem.\ 33]{BShn}) that $F(A[\lambda]) \simeq F(\sqrt[3]{d}, \zeta_3) = F(\sqrt[3]{d^2}, \zeta_3)$ for some $d \in F^\times$. We conclude that $A[\lambda] \simeq E_{2d}[2]$, since $\Q(E_{2d}[2])$ is the splitting field of the cubic polynomial $x^3 + 64d^2$, and $E_{2d}[2]$ is the unique twist of $(\Z/2\Z)^2$ with that splitting field.\footnote{Recall that $E_n$ is isomorphic to $y^2 = x^3 + 16n^2$ so $E_{2n}$ is isomorphic to $y^2 = x^3 + 64n^2$.} The lemma now follows from Theorem~\ref{theorem:2Selparam}. 
\end{proof}

\begin{example}{\em 
If $A \colon y^2 = x^3 + d$ is an elliptic curve then $A[2] \simeq E_{d'}[2] \simeq \mathrm{Stab}_G(y_{d'})$, where $d' = 2d$, so the notation $E_{d,n}$ from the introduction does not quite match the notation for the parameter $d$ whose existence is guaranteed by Lemma \ref{lem:muaction}. The choice $d' = 4d^2$ also works.
For our purposes, the exact choice of $d$ satisfying Lemma \ref{lem:muaction} will not matter.}      
\end{example}

We fix once and for all an element $d \in F^\times$ as in  Lemma \ref{lem:muaction}. Just as in the previous subsection there is an exact sequence defining the theta group $\Theta(L) = \Aut(L)$:
\[0 \to \G_m \to \Theta(\L) \to A[\lambda] \to 0.\]
We will sometimes think of points of $\Theta(L)$ as pairs $(P, \varphi)$, with $P \in A[\lambda]$ and $\varphi$ an isomorphism $t_P^*L \simeq L$, where $t_P \colon A \to A$ is translation-by-$P$. If $P = 0$, then $\varphi$ may be viewed as a scalar in $F^\times$.  Taking Galois cohomology, we obtain a map of pointed sets $H^1(F, A[\lambda]) \to H^2(F, \G_m)$.  As before, the kernel of this map is $H^1(F, \Theta(L))$. 

We can now state our general parametrization result. 

\begin{theorem}\label{theorem:thetaparamgeneral}
Let $d \in F^\times$ be defined as in Lemma $\ref{lem:muaction}$, and for each $n \in F^\times$,  let $(A_n, \L_n)$ be the $n$-th cubic twist of $(A, \L)$. Then there is a natural bijection between the $G(F)$-orbits on $Y(F)_{dn}$ and the pointed set $H^1(F, \Theta(\L_n))$, sending the reducible orbit $v_{dn}$ to the identity element. For any $v \in Y(F)_{dn}$, the stabilizer $\mathrm{Stab}_G(v)$ is isomorphic to $A_n[\lambda_n]$. 
\end{theorem}

\begin{proof}
Let $E = E_d$ and $L_E = \O_E(2\infty)$.  We will show that there is an isomorphism of central extensions $\Theta(\L_1) \simeq \Theta(L_E)$. The case $n = 1$ of the theorem then follows from Corollary \ref{cor:thetaparam}. The general case follows too, since taking cubic twists we obtain $\Theta(\L_n) \simeq \Theta((L_E)_n)$, for all $n$.
  
Let $M$ be the line bundle $p_1^*\L \otimes p_2^*\L_E$ which gives rise to the product polarization on $A \times E$.  
The theta group of $M$ is related to the theta groups of $\L$ and $\L_E$ in a simple way:
\[\Theta(M) \simeq (\Theta(\L) \times \Theta(\L_E))/\Delta, \]
where $\Delta = \{(1, t, 1, t^{-1})\} \subset \G_m \times \G_m \subset \Theta(\L) \times \Theta(\L_E)$. Thus there is a short exact sequence 
\[1 \longrightarrow \G_m \longrightarrow \Theta(M) \stackrel{p}{\longrightarrow} A[\lambda] \times E[2] \longrightarrow 1.\]
Now choose an isomorphism $\eta \colon A[\lambda] \simeq E[2]$ and let $B$ be the abelian variety which is the quotient of $A \times E$ by the graph $\Gamma_\eta \subset A[\lambda] \times E[2] \subset A \times E$.
Let $\pi \colon A \times E \to B$ be the quotient map. The subgroup $\Gamma_\eta \subset (A \times E)[\lambda_M] = A[\lambda] \times E[2]$ is isotropic with respect to the skew-symmetric Weil pairing induced by $M$, since
\[\langle(P, \eta(P)), (Q, \eta(Q)) \rangle_M = \langle P, Q \rangle_\L \langle \eta(P), \eta(Q) \rangle_{\L_E} = \langle P, Q \rangle_\L^2 = 1.\]
Here we have used that $\eta$ is (automatically!) a symplectic isomorphism with respect to the $\mu_2$-valued Weil pairings induced by $\L$ and $\L_E$.  
\begin{lemma}
There is a line bundle $L_B$ on $B$ such that $\pi^*\L_B \simeq M$.
\end{lemma}
\begin{proof}
Define the group scheme $\widehat{\Gamma}_\eta = \ker(\Pic(B) \stackrel{\pi^*}{\longrightarrow} \Pic(A \times E))$, abstractly isomorphic to the self-dual group scheme $E[2] \simeq \Gamma_\eta = \ker(\pi)$.  Since $\Gamma_\eta$ is isotropic, there is a line bundle $\widetilde L$ on $B_{\overline F}$ such that $\pi^*\widetilde L \simeq M_{\overline F}$ \cite[\S 23]{MumfordAV}. We show that we can choose $\widetilde L$ such that it descends to $B$. Since $B$ has a rational point, this is equivalent to $\widetilde L^\sigma \simeq \widetilde L$ for all $\sigma \in G_F$.   The collection of all such line bundles $\widetilde L$ is a $\widehat{\Gamma}_\eta$-torsor over $\Q$, hence gives a class  $c \in H^1(F, E[2])$.
Moreover, $\widetilde L$ descends to $B$ if and only if this torsor is trivial.  On the other hand, the $\mu_3$-action on $A \times E$ descends to a $\mu_3$-action on $B$ which interwines the isogeny $\pi$. Since $L$ and $L_E$ are fixed by $\mu_3$, we have $\zeta^*\widetilde L \otimes \widetilde L^{-1} \in \widehat{\Gamma}_\eta$ and hence $c =\zeta(c)$ in $H^1(F(\zeta_3), E[2])$. Let $K = F(\zeta_3, E[2])$ be the cubic \'etale-algebra over $F(\zeta_3)$ cut out by $E[2] \setminus \{0\}$.\footnote{This is either a $\Z/3\Z$-extension or $F(\zeta_3)^3$.  The map  $H^1(F, E[2]) \to H^1(F(\zeta_3), E[2])$ is injective in both cases.}  Under the isomorphism (\cite[Prop.\ 5.1]{BG})
\[H^1(F(\zeta_3), E[2]) \simeq \ker\left(K^\times/K^{\times2} \stackrel{\Nm}{\longrightarrow} F(\zeta_3)^\times/F(\zeta_3)^{\times2}\right),\] the action of $\mu_3$ on $H^1(F(\zeta_3), E[2])$ is identified with the action of $\Gal(K/F(\zeta_3))$ on the elements of $K^\times/K^{\times 2}$ of square norm.   Since the latter group action is easily seen to have no nontrivial fixed points, it follows that $c$ is trivial and hence $\widetilde L$ descends to $B$.
\end{proof}

\begin{remark}{\em
    We remark for later use that  Riemann-Roch and the formula $\chi(M) = \deg(\pi)\chi(\L_B)$ together imply that $\L_B$ is a principal polarization.}
\end{remark} 
The existence of this line bundle $\L_B$ implies, by \cite[Thm.\ 2 \S23]{MumfordAV}, that there is a subgroup $H \subset \Theta(M)$ and an isomorphism $\psi \colon \Gamma_\eta \simeq H$ such that $p \circ \psi = \mathrm{id}$.  This data determines an isomorphism of theta groups
\[
\begin{tikzcd}
0  \arrow[r] & \G_m \arrow[r]\arrow[d,"\mathrm{id}"] & \Theta(\L) \arrow[r]\arrow[d, "\widetilde\eta"] & A[\lambda] \arrow[r]\arrow[d, "\eta"] & 0 \\
0  \arrow[r] & \G_m \arrow[r] & \Theta(\L_E) \arrow[r] & E[2] \arrow[r] & 0. 
\end{tikzcd}
\]
Explicitly, on $S$-valued points, if $\psi(P, \eta(P)) = (P,s_0,\eta(P),r_0) \in H \subset \Theta(M)$, then  
\[\widetilde \eta(P,s) = (\eta(P), (s_0^{-1}s)r_0), \]
where we view $s_0^{-1}s$ as a scalar in $\Aut(\L) \simeq \G_m \simeq \Aut(\L_E)$. This proves the claimed  isomorphism of central extensions $\Theta(\L) \simeq \Theta(\L_E)$.
\end{proof}

\begin{remark}{\em 
The bijections of Theorem \ref{theorem:thetaparamgeneral} (one for each $n$) seem to depend on the initial choice of isomorphism $\Theta(\L) \simeq \Theta(L_E)$, which itself depends on a choice of isomorphism $A[\lambda] \simeq E_d[2]$. In fact, Jef Laga has pointed out to us that any automorphism of $\Theta(L)$ (as central extensions) that commutes with the $\mu_3$-action and which induces the identity on $A[\lambda]$ is the identity. This uniqueness gives another way to prove the existence of an isomorphism $\Theta(L) \simeq \Theta(L_E)$ over $F$.   
}
\end{remark}

\begin{example}{\em 
If $A \colon y^2 = x^3 + d$ is an elliptic curve with $L = \O_A(2\infty)$, then we can make the bijection $H^1(F, \Theta(L)) \simeq H^1(F, \Theta(L_E))$ in the proof of Theorem \ref{theorem:thetaparamgeneral} very explicit. Recall that in this case we may take $E = E_{2d^2} \colon y^2 = x^3 + d^4$.  Then $H^1(F, \Theta(L_E))$ parametrizes orbits of pairs of binary cubic forms $y = (F_1,F_2)$ with $A_1(y) = 0$ and $A_3(y) = 2d^2$. On the other hand, $H^1(F, \Theta(L))$ parametrizes isomorphism classes of curves of the form $z^2 = f(x,y)$ with Jacobian $A$.  The explicit map between these two sets sends $(F_1,F_2)$ to the curve $dz^2 = \Disc( w_1F_1 -  w_2F_2)$.}
\end{example}

Finally, we make an explicit connection between rational points on $A_n(F)$ and $G(F)$-orbits on $Y(F)_{dn}$. In fact, the more direct connection is with $\widehat{A}_n(F)$ not $A_n(F)$, since the short exact sequence 
\[0 \to A_n[\lambda_n] \to A_n \to \widehat A_n \to 0\]
induces a map $\delta \colon \widehat A_n(F) \to H^1(F, A_n[\lambda_n])$.
\begin{proposition}\label{prop:kummerobstruction}
The composition $\widehat A_n(F) \stackrel{\delta}{\to} H^1(F, A_n[\lambda]) \stackrel{\mathrm{ob}}{\to}H^2(F, \G_m)$ is $0$.  In particular, the map $\delta$ factors through a map $\widehat{A}_n(F) \to H^1(F, \Theta(\L_n))$.
\end{proposition}

\begin{proof}
This is \cite[Prop.\ 6.4]{poonenrains} and is where we use the fact that $\L$ is symmetric.
\end{proof}

Thus, to each point $P \in \widehat{A}_n(F)$, there is an associated $G(F)$-orbit of pairs of binary cubic forms $v \in Y(F)_{dn}$.

\subsection{Parametrization over local fields}\label{subsec:localfields}

We specialize the preceding discussion to the case $F = \Q_p$ for some prime number $p$ or $F = \R = \Q_\infty$. If $p < \infty$, we assume, without loss of generality, that the fixed element $d \in F^\times$ lies in $\Z_p$.  

For each $n \in \Q_p^\times$, we say that $v \in Y(\Q_p)_{nd}$ is {\it locally soluble} if the corresponding element of $H^1(\Q_p, \Theta(\L_n))$, via Theorem \ref{theorem:thetaparamgeneral}, lies in the image of the Kummer map $\widehat{A}_n(\Q_p) \to H^1(\Q_p, \Theta(\L_n))$ given by Proposition \ref{prop:kummerobstruction}.  Write $Y(\Q_p)_{dn}^\mathrm{ls}$ for the set of locally soluble $v \in Y(\Q_p)$ having $A_3$-invariant $dn$.  This notion of local solubility of course depends on the pair $(A,\L)$.    

\begin{example}{\em 
If $A \colon y^2 = x^3 + d$ is an elliptic curve, then $v = (F_1,F_2) \in Y(\Q_p)$ is locally soluble if and only if the curve $C \colon dz^2 = \Disc(w_1F_1 - w_2F_2)$ has $C(\Q_p) \neq \emptyset$.  
}
\end{example}
For $p < \infty$, we need the following integrality result, which is due to the second author and Ho in the case where $A$ is an elliptic curve $E$ with a 3-torsion point and $\L = \O_E(2\infty)$  \cite[Thm.\ 4.2]{BhargavaHo2}. Write $\mathfrak{f}_A$ for the conductor of $A$, so that $p \mid \mathfrak{f}_A$ if and only if $A$ has bad reduction. If $A = E$ is an elliptic curve with minimal Weierstrass model, then $p \mid \mathfrak{f}_E$ if and only if $p$ divides the discriminant of a minimal Weierstrass model of $E$.  

\begin{theorem}\label{theorem:integrality}
Assume $p \nmid 6d\mathfrak{f}_A\infty$ and $n \in \Z_p$. If $v \in Y(\Q_p)_{dn}^\mathrm{ls}$, then $Y(\Z_p)\cap G(\Q_p) y\neq \emptyset$.  In other words, the $G(\Q_p)$-orbit of $y$ contains an integral representative.
\end{theorem}

The proof uses the following lemma.

\begin{lemma}\label{lem:H1vanishing}
Suppose $p > 3$ and $m \in p\Z_p$ has valuation $1$ or $2$.  Then $H^1(\Q_p,E_m[2]) = 0$.
\end{lemma}
\begin{proof}
The cubic field $L = \Q_p(\sqrt[3]{m^2})$ is totally ramified at $p$.  Since, by  \cite[Prop.\ 5.1]{BG}, 
\[E_m[2] \simeq \ker \left(\mathrm{Res}^L_{\Q_p} \mu_2 \stackrel{\Nm}{\longrightarrow} \mu_2\right),\]
we compute 
\begin{align*}
H^1(\Q_p, E_m[2]) &\simeq \ker(L^\times/L^{\times 2} \stackrel{\Nm}{\longrightarrow}\Q_p^\times/\Q_p^{\times 2})\\
& \simeq\ker(\O_L^\times/\O_L^{\times 2} \stackrel{\Nm}{\longrightarrow}\Z_p^\times/\Z_p^{\times 2})\\
&\simeq\ker(\F_p^\times/\F_p^{\times 2} \stackrel{x \mapsto x^3}{\longrightarrow}\F_p^\times/\F_p^{\times 2}) = 0.
\end{align*}
\end{proof}

\begin{proof}[Proof of Theorem $\ref{theorem:integrality}$]
Recall that $A[\lambda] \simeq E_d[2]$.  Twisting both sides, we have $A_n[\lambda_n] \simeq E_{dn}[2]$, and hence $H^1(\Q_p, A_n[\lambda_n]) \simeq H^1(\Q_p, E_{dn}[2])$.  If $v_p(n) \not\equiv 0 \pmods{3}$, then 
\[H^1(\Q_p, \Theta(\L_n)) \subset H^1(\Q_p, A_n[\lambda_n]) \simeq H^1(\Q_p, E_{dn}) = 0\]
by Lemma \ref{lem:H1vanishing}. It follows that there is a unique $G(\Q_p)$-orbit in $Y(\Q_p)_{nd}$, namely the reducible orbit.  This orbit has an explicit integral representative, namely, the pair of binary cubic forms $(3xy^2, x^3 + ndy^3)$; see \cite[\S 4.6(b)]{BhargavaHo2} or the formulas in the next section.

If $v_p(n) \equiv 0 \pmods{3}$, then because $p \nmid 6d\mathfrak{f}_A\Disc(\Q_p(\sqrt[3]{n}))$, the twist $A_n$ has good reduction at $p$.  The image of the Kummer map $\widehat{A}_n \to H^1(\Q_p, A_n[\lambda_n])$ is therefore the subgroup of unramified classes \cite[Prop.\ 2.7(d)]{CesnaviciusFlat}. By the same reasoning, this is also the image of the Kummer map $E_{dn}(\Q_p) \to H^1(\Q_p, E_{dn}[2]) \simeq H^1(\Q_p, A_n[\lambda_n]);$  in this case, the result follows  from \cite[Theorem~ 4.2]{BhargavaHo2}. 
\end{proof}

Finally, we record a near-converse to Theorem \ref{theorem:integrality}. 
\begin{proposition}\label{prop:cubefreeimpliessoluble}
If $p \nmid 6d\mathfrak{f}_A\infty$ and $n \in \Z_p$ has valuation $v_p(n) \leq 2$, then $Y(\Z_p)_{dn} \subset Y(\Q_p)^\mathrm{ls}_{dn}$.
\end{proposition}
\begin{proof}
If $1 \leq v_p(n) \leq 2$, then in the proof of Theorem \ref{theorem:integrality} we saw that all $v \in Y(\Q_p)$ having $A_3$-invariant $dn$ are locally soluble. So it remains to show that if $v_p(n) = 0$ and $v \in Y(\Z_p)_{dn}$, then the class $v' \in H^1(\Q_p, E_{dn}[2])\simeq H^1(\Q_p, A_n[\lambda_n])$ corresponding to $v$ is in the image of the Kummer map (recall that in this case, the image of the Kummer map is the subgroup of unramified classes, for both $A_n$ and $E_{dn}$).  Now, the pair of binary cubic forms $v = (F_1,F_2)$ determines a genus one curve $C/\Q_p$ with integral model $\mathcal{C} \colon z^2 = \Disc(w_1F_1 - w_2F_2)$. 
The special fiber $\mathcal{C}_p$ is a hyperelliptic curve over~$\F_p$ which comes from a pair of binary cubic forms over $\F_p$ and has discriminant which is nonzero in $\F_p$.  
It follows that $\mathcal{C}_p$ is smooth and hence $\mathcal{C}/\Z_p$ is smooth.  Since $C/\Q_p$ has good reduction, it has a $\Q_p$-point, and hence $C \simeq E_{dn}$. Equivalently, the class $v'$ lies in $\ker(H^1(\Q_p, E_{dn}[2]) \to H^1(\Q_p, E_{dn}))$.  Since this kernel is also the image of the Kummer map, we have proven the desired claim.
\end{proof}

\subsection{Parametrization over global fields}
Now let us specialize the situation of Section \ref{subsec:A-torsors} to the case $F = \Q$.  Recall that the Selmer group $\Sel_{\lambda_n}(A_n)$ is the kernel of the natural map 
\[H^1(\Q, A_n[\lambda_n]) \longrightarrow \prod_{p \leq \infty} H^1(\Q_p, A_n)[\lambda_n].\]
Equivalently, the Selmer group $\Sel_{\lambda_n}(A_n)$ consists of those $\alpha \in H^1(\Q, A_n[\lambda_n])$ whose restrictions $\mathrm{res}_p(\alpha)$ lie in the image of the map $\widehat{A}_n(\Q_p) \to H^1(\Q_p, A_n[\lambda_n])$  for all primes $p \leq \infty$.  We have an exact sequence
\[0 \to \widehat{A}_n(\Q)/\lambda_n(A_n(\Q)) \to \Sel_{\lambda_n}(A_n) \to \Sha(A_n)[\lambda_n] \to 0,\]
where $\Sha(A_n)[\lambda_n]$ is the $\lambda_n$-torsion subgroup of the Tate-Shafarevich group of $A_n$. 

By Proposition \ref{prop:kummerobstruction}, there are inclusions 
\[ \Sel_{\lambda_n}(A_n) \subset H^1(\Q, \Theta(\L_n)) \subset H^1(\Q, A_n[\lambda_n]).\]  
If $v \in Y(\Q)_{dn}$ corresponds, under the bijection of Theorem \ref{theorem:thetaparamgeneral}, to an element of $\Sel_{\lambda_n}(A_n)$, we say that $v$ is {\it Selmer}. Equivalently, $v  \in Y(\Q)_{dn}$ is Selmer if and only if it is locally soluble at $p$ for every prime $p$.  Write $Y(\Q)_{dn}^\mathrm{sel}$ for the set of Selmer elements $v$ of invariant $dn$, and $Y(\Q)^\mathrm{sel}$ for the set of all Selmer elements. 
\begin{theorem}\label{thm:globalselmerbijection}
For each $n \in \Q^\times$, there is a natural bijection between the $G(F)$-orbits on $Y(F)_{dn}^\mathrm{sel}$ and the group $\Sel_{\lambda_n}(A_n)$. Under this bijection, the identity class in $\Sel_{\lambda_n}(A_n)$ corresponds to the unique reducible orbit in $Y(F)_{dn}$.  
\end{theorem}

\begin{proof}
This follows from Theorem \ref{theorem:thetaparamgeneral}. 
\end{proof}

\begin{theorem}\label{thm:globalintegrality}
There exists a nonzero integer $N$ such that for all nonzero $n \in N\Z$, and for all $v \in Y(\Q)_{dn}^\mathrm{sel}$, we have $G(\Q) v \cap Y(\Z) \neq \emptyset$. 
\end{theorem}

\begin{proof}
If $v \in Y(\Q)_{dn}$, then the orbit $G(\Q)y$ contains some $v' \in Y(\Z)$ if and only if the orbit $G(\Q_p)y$ contains an element of $Y(\Z_p)$, for every prime $p$. This follows from the fact that the class number of $G = \SL_2^2/\mu_2$ is $1$.  Thus, by Theorem \ref{theorem:integrality}, it suffices to show that for any prime $p$, if $v_p(n)$ is large enough (depending on $p$), then all of the $G(\Q_p)$-orbits on $Y(\Q_p)_{nd}$ have representatives in $Y(\Z_p)_{nd}$.

Note that for any given prime $p$, there are only finitely many $G(\Q_p)$-orbits with $A_3$-invariant $nd$, since the set $H^1(\Q_p, \Theta(L_n)) \subset H^1(\Q_p, E_{dn}[2])$ is finite.   There are also only finitely many cube-classes of $n \in \Q_p^\times/\Q_p^{\times 3}$. Thus, we may scale any $v \in Y(\Q_p)_{dn}$ by an appropriate power of $p$, to obtain an integral element of $Y(\Q_p)_{dn'}$ with $n' \equiv n \pmods{\Q_p^{\times 3}}$.  It follows that if $v_p(n)$ is large enough (depending only on $p$), then all of the (finitely many) $G(\Q_p)$-orbits with $A_3$-invariant $dn$ contain representatives in $Y(\Z_p)$.
\end{proof}

By the previous two theorems, in order to estimate $\sum_{|n| < X}\#\Sel_{\lambda_n}(A_n)$, we must determine the number of Selmer $G(\Q)$-orbits lying in $Y(\frac1N\Z)$ and having $A_3$-invariant $nd$ bounded by $X$ in absolute value.  Here, the ``Selmer'' condition is determined by (infinitely many) congruence conditions.  Even though we impose infinitely many congruence conditions to sieve down to the Selmer elements, the following result shows that these conditions do not throw out too many orbits.

\begin{proposition}\label{prop:globalintegralityimpliesselmer}
There exists a nonzero integer $M$ such that if $v \in Y(\frac{1}{M}\Z)_{dn}$ and $dn \in \Z$ is cube-free at all primes $p \nmid M$, then  $v \in Y(\Q_p)^\mathrm{sel}$. 
\end{proposition}
\begin{proof}
This follows from Proposition \ref{prop:cubefreeimpliessoluble}. 
\end{proof}

Most results in this subsection generalize easily to the case where $F$ is any number field (as well as global fields of characteristic $p > 3$).  However, for the analogue of Proposition \ref{prop:globalintegralityimpliesselmer} in this more general setting, modifications are needed to account for nontrivial class group and unit group.  

We turn now to the problem of estimating $\sum_{|n| < X}\#\Sel_{\lambda_n}(A_n)$, which we shall carry out via a suitable weighted count of $G(\Z)$-orbits on $Y(\Z)$ satisfying $|A_3|<X$. 

\section[The \hspace{-.015in}number \hspace{-.015in}of integral \hspace{-.015in}orbits \hspace{-.015in}on \hspace{-.015in}an \hspace{-.015in}invariant \hspace{-.015in}quadric \hspace{-.015in}with \hspace{-.015in}bounded \hspace{-.015in}invariants]{The number of integral orbits on an invariant quadric with bounded invariants
}\label{sec:counting}

Let $V:=2\otimes \Sym_3(2)$, and let $G$ be the image in $\GL(V)$ of $\{(g,h)\in \GL_2\times \GL_2 : \det{g}\cdot (\det{h})^3 = 1\}$. \linebreak Then, as already noted, $G$ acts on $V$; the first $\GL_2$ acts on the first factor via the standard representation, and the second $\GL_2$ acts on the symmetric cube of the standard representation as usual. Note that $G\cong (\SL_2\times \SL_2)/\mu_2$.  

The action of $G(\C)$ on $V(\C)$ has two independent polynomial invariants denoted $A_1$ and $A_3$, and they have degrees 2 and 6, respectively. The degree $2$ invariant has an especially simple formula: \begin{equation}\label{A1formula}
    A_1((F_1,F_2)) = r_1 r_8 - 3r_2 r_7 + 3r_3 r_6 - r_4 r_5,
    \end{equation}
where $F_1(x,y) = \sum_{i=0}^3 \binom3ir_{i+1} x^{3-i} y^i$ and $F_2(x,y) = \sum_{i=0}^3 \binom3i r_{i+5} x^{3-i} y^i$. Up to scaling and translations by $A_1^3$, the invariant $A_3$ is given by the degree 3 invariant $J$ of the covariant binary quartic form given by  the  Jacobian determinant $G = (\partial_xF_1)(\partial_yF_2) - (\partial_yF_1)(\partial_xF_2)$. More precisely, we have
\begin{equation}\label{A3formula}A_3(v) =  \frac1 {54}\Bigl(J(G(v))- A_1(v)^3\Bigr)\end{equation}
for $v = (F_1,F_2) \in V(\C);$
then   $A_3(v)$ is a primitive integer polynomial in $r_1,\ldots,r_8$. The invariants $J(G)$ and $A_3$ are normalized so that for $v = (3xy^2,x^3 + ny^3)$, we have $J(G(v)) = 54n$ and $A_3(v) = n$. For a geometric characterization of these invariants, see Theorem~\ref{2selpar1}, which is \cite[Thm.\ 42]{BhargavaHo2}. The quadric $Y \subset V = 2 \otimes \Sym_3(2)$ defined by $A_1 = 0$ is thus preserved by the action of~$G$. 

In this section, we extend the methods of \cite{BhargavaHo2}, \cite{Ruth},  and \cite{leventthesis}, involving geometry of numbers and the circle method, to give an asymptotic formula, with a power-saving error term, for the number of $G(\Z)$-orbits on $Y(\Z)$ 
such that  $|A_3|<X$ and where $A_3$ satisfies any  acceptable set of congruence conditions. We take the methods of \cite{Ruth} and \cite{leventthesis}  further by expressing our asymptotic formula in terms of real and $p$-adic integrals with respect to $p$-adic and real Haar measures on $G$, respectively, which represent a novel re-interpretation of the singular integral and singular series that arise in the circle method in this case. These expressions will be key in the applications to the average sizes of Selmer groups in  Section~\ref{computing the average of two selmer}.

For every $n \in \Z$, there is a distinguished $G(\Z)$-orbit on $Y(\Z)$ having $A_3$-invariant $n$, namely the orbit of the pair $v_n:=(F_1,F_2) = (3xy^2,x^3 + ny^3)$. Note that $v_n\in Y(\Z)$ is a {\it reducible} element with $|A_3|$-invariant $n$, since  the corresponding binary quartic form $f(w_1,w_2) = \Disc(w_1F_1 - w_2F_2)$ has a rational root at $[1:0]$; one can check that this is the unique such orbit up to $G(\Q)$-equivalence.  Since the reducible orbit is unique, we focus on counting the irreducible ones.

More generally, we wish to count weighted irreducible $G(\Z)$-orbits on $Y(\Z)$ having bounded $A_3$-invariant, where 
the weights are defined by finite or appropriate infinite sets of
congruence conditions. 
A function $\phi:Y({\Z})\to[0,1]\subset \R$ is said to be {\it defined by congruence conditions} if, for all primes $p$, there exist functions $\phi_p:Y({\Z_p})\to[0,1]$ satisfying the following conditions:
\begin{itemize}
\item[(1)] For all $y\in Y(\Z)$, the product $\prod_p\phi_p(y)$ converges to $\phi(y)$;
\item[(2)] For each $p$, the function $\phi_p$ is 
locally constant outside some closed set in $Y({\Z_p})$ of measure zero.
\end{itemize}
We say that such a function $\phi$ is {\it acceptable} if for
sufficiently large primes $p$, we have $\phi_{p}(y)=1$ whenever
$p^2 \nmid A_3(y)$.

Let $N_\phi(Y(\Z);X)$ denote the weighted number of irreducible  $G(\Z)$-orbits of elements $y\in Y(\Z)$ with $|A_3(y)|<X$, where the orbit of each such $y$ is weighted by $\phi(y)$. The purpose of this section is to prove the following theorem:

\begin{theorem}\label{thsqfreetc}
  Let $\phi:Y(\Z)\to[0,1]$ be an acceptable function that is defined by the functions $\phi_{p}:Y({\Z_p})\to[0,1]$. Then  
\begin{equation}
N_\phi(Y(\Z);X)
  = 
  X  \cdot 
 \int_{\scriptstyle{y\in G(\Z) \backslash Y(\R)}\atop\scriptstyle{|A_3(y)|<1}}
dy \cdot \prod_{p}
  \int_{y\in Y({\Z_{p}})}\phi_{p}(y)\,dy\,+\,O_\phi\left(X^{1 - \Omega(1)}\right);
\end{equation}
  here $dy$ is the $G(\R)$-invariant $($resp.\ $G(\Z_p)$-invariant$)$ measure $dr_2\,dr_3\cdots dr_8/(\partial A_1/\partial r_1)$ on $Y(\R)$ $($resp.\ $Y(\Z_p))$,
  where $r_1,\ldots,r_8$ are the coordinates on $V$.
\end{theorem}\noindent
Here $\Omega(1)$ denotes a quantity that is bounded below by a positive absolute~constant.

\subsection{Counting irreducible elements of bounded height}\label{countsec}

In this subsection, we prove the following special case of Theorem~\ref{thsqfreetc} giving the asymptotic number of
$G(\Z)$-equivalence classes of irreducible elements of $Y(\Z)$ having
bounded $A_3$-invariant and satisfying any specified finite set of congruence conditions.

 To state the result, for any $G(\Z)$-invariant set $S\subset Y(\Z)$, let $N(S;X)$
  denote the number of $G(\Z)$-equivalence classes of irreducible
  elements $y\in S$ satisfying $0\neq |A_3(y)|<X$. Let $S_p\subset Y(\Z_p)$ denote the $p$-adic closure of $S$ in $Y(\Z_p)$. We prove:
  
\begin{theorem}\label{thmcount}
Let $S\subset Y(\Z)$ be defined by a finite set of $G(\Z)$-invariant congruence conditions modulo $M$. Then
$$N(S;X)= X\cdot 
\int_{\scriptstyle{y\in G(\Z) \backslash Y(\R)}\atop\scriptstyle{|A_3(y)|<1}}
dy
\cdot 
\prod_p \int_{\scriptstyle{y\in S_p}} dy
+ O\left(M^{O(1)} X^{1 - \Omega(1)}\right).$$
\end{theorem}

\subsubsection{Reduction theory}

Define the (naive) {\it height} $H(v)$ of an element $v\in V(\R)$ by $$H(v) := H(A_1(v), A_3(v)) := \max\left\{|A_1(v)|^{{12}}, |A_3(v)|^4\right\},$$ and the discriminant $\Delta(v)$ of $v$ by $$\Delta(v) := \Delta(A_1(v), A_3(v)) := 16 A_3(v)^3 (A_1(v)^3 - 27 A_3(v)).$$
Let $v(A_1,A_3):=(3x y^2 - A_1 y^3, x^3 + A_3 y^3)$, so that $v(A_1,A_3)$ has invariants $A_1$ and $A_3$. Let
$$L_{\Delta < 0} := \{v(A_1,A_3) : \Delta(A_1, A_3) < 0,\, H(A_1,A_3) = 1\}.$$ Then, as shown in \cite[Section~$5$]{BhargavaHo2},
$L_{\Delta < 0}$ is a fundamental domain for the action of $G(\R)$ on the set of height $1$ elements of $V(\R)_{\Delta < 0}$.

Let $R = \R^+\cdot L_{\Delta < 0}$. Then $R$ is a fundamental domain for the action of $G(\R)$ on $V(\R)_{\Delta < 0}$. For any $v\in V(\R)_{\Delta < 0}$, let $v_R$ denote the unique $G(\R)$-representative of $v$ in $R$, and $v_L$ denote the unique $\R^+\cdot G(\R)$-representative of $v$ in $L_{\Delta < 0}$. Similarly let $\lambda_v\in \R^+$ be such that $v_R = \lambda_v\cdot v_L$.  

Let $R(X) := \{v\in R : H(v)\leq X\}.$  Since $H(\lambda v) = \lambda^{24}H(v)$, the coordinates of any $v\in \lambda L_{\Delta < 0}\subseteq R(X)$ are all $O(\lambda)= O(X^{1/24})$. Hence for any compact $G_0\subseteq G(\R)$, the coefficients of any  $v\in G_0 \cdot\lambda L_{\Delta < 0}\subseteq R(X)$ are all $O_{G_0} (\lambda) = O_{G_0}(X^{1/24})$.

Let $v_\pm := (3xy^2,x^3\pm y^3)\in L_{\Delta < 0}$ be the two points in $L_{\Delta < 0}$ with $A_1 = 0$. 

Let $\mathcal{F}$ be a fundamental domain for the action of $G(\Z)$ on $G(\R)$, as constructed in \cite[\S5.2]{BhargavaHo2}. Thus $\FF$ lies inside a Siegel set; explicitly, if we write $t=(t_1,t_2)$ and $u=(u_1,u_2)$, then $$\FF= \{n_u a_t k:n_u\in N'(t),a_t\in A',k\in K\},$$ where
\begin{align*}
N'(t) &:= \left\{n_u:=(n_{u_1}, n_{u_2})\in G(\R) : u_i\in I({t_i})\right\},
\\A' &:= \left\{a_t:=(a_{t_1}, a_{t_2})\in G(\R) : t_i\geq \sqrt{\frac{\sqrt{3}}{2}}\right\},
\\K &:= \SO_2(\R)\times \SO_2(\R)\subseteq G(\R);
\end{align*}\noindent
here $n_{u_i} := \smalltwobytwo{1}{0}{u_i}{1}$ and $a_{t_i} := \smalltwobytwo{t_i^{-1}}{0}{0}{t_i}$, and $I(t_i)$ is a union of one or two subintervals of
$[-\frac12,\frac12]$ depending only on the value of $t_i$.  

In terms of these Iwasawa coordinates, we define a Haar measure $dg$ on $G(\R)$ by 
\[dg = t_1^{-2}t_2^{-2}
du_1 \,du_2 \,d^\times t_1 \,d^\times t_2\, dk.\]
Here, $dk$ is a Haar measure on $K$ such that $\int_{k\in K} dk = 1$.

\subsubsection{Averaging over fundamental domains}\label{avgsec}

We now extend the averaging method of \cite{Bhargavadensityofquinticfields,BS1,leventthesis} to estimate $N(S; X)$.
Let $$V(\Z)^\irr := \{(F_1, F_2)\in V(\Z) : \Disc(w_1F_1-  w_2F_2) \text{ has nonzero discriminant and no root in $\P^1(\Q)$}\}.$$ 
For any subset $S\subset V(\Z)$, let $S^\irr:=S\cap V(\Z)^\irr$. 
Let $Y(\Z) := \{v\in V(\Z) : A_1(v) = 0\}$.
Our goal in this section is to count the number of $G(\Z)$-orbits on $Y(\Z)^\irr$ with $|A_3|<X$.

Let $\delta\in \R^+$ be a small absolute constant that we will choose at the end of the argument.
Let $\mu_0\in C^\infty(\R)$ be such that $\mu_0\geq 0$ on $\R$, $\mu_0(x) = 1$ if $x\leq 0$ and $\mu_0(x) = 0$ if $x\geq 1$. Let $$\mu_+(x) := \mu_0(X^{\delta^2}(x-1))\, \mu_0(X^{\delta^2}(-x-1)),$$ and let $$\mu_-(x) := \mu_0(X^{\delta^2}(x-1) + 1)\,\mu_0(X^{\delta^2}(-x-1) + 1).$$ Thus $\mu_+(x) = 1$ when $|x|\leq 1$ and $\mu_+(x) = 0$ when $|x|\geq 1 + X^{-\delta^2}$, and similarly $\mu_-(x) = 1$ when $|x|\leq 1 - X^{-\delta^2}$ and $\mu_-(x) = 0$ when $|x|\geq 1$. Note that $||\mu_\pm^{(k)}||_\infty\ll X^{k \delta^2}$, where $\mu_\pm^{(k)}$ denotes the $k$-th derivative of $\mu_\pm$.\footnote{The point of having a $\delta^2$ in the exponent is to have $X^{o(\delta)}$ lost upon each differentiation (so that in total we save $\gg X^{\delta - o(\delta)}$ each time we integrate by parts) during the repeated integration by parts in the proof of Proposition \ref{the main term to integrate}.} Note that $\mu_\pm$ depends on the parameter $X$, but our notation suppresses this.

Let $\alpha\in C_c^\infty(G(\R))$ be  a smooth, compactly supported, and $K$-invariant function such that $\int_{G(\R)} \alpha = 1$. Let $\beta\in C_c^\infty(L_{\Delta < 0})$ be   such that $\beta(v_{\pm}) = \frac{1}{2}$, and such that both $\supp({\beta})$ and $\beta^{-1}(1/2)\subseteq L_{\Delta < 0}$ are unions of two small compact intervals containing $v_\pm$. For $v\in V(\R)$, define
\begin{align*}
\nu(v) &:= \sum_{\lambda_v\cdot g\cdot v_L = v} \alpha(g)\,\beta(v_L)
=\sum_{g\cdot v_R = v} \alpha(g)\,\beta(v_L)
\end{align*}\noindent
and $$w_\pm(v;X) := \mu_\pm\left(\frac{A_3(v)}{X}\right)\nu(v).$$
Note that $w_\pm(v;X) = w_\pm(X^{-\frac{1}{6}} v; 1)$.

\begin{remark}
{
\em
We will suppress the dependence of all implicit constants on $\alpha$, $\beta$, and $\mu_0$.
}
\end{remark}

For a $G(\Z)$-invariant set $S\subset Y(\Z)$, we define 
$N_\pm(S;X) := \sum_{v\in G(\Z)\backslash S^\irr} \mu_\pm\left(\frac{A_3(v)}{X}\right)$. Thus 
\[N_-(S;X)\leq N(S;X)\leq N_+(S;X).\]

Now because the defining sum over $G(\Z)\backslash S^\irr$ in the definition of $N_\pm(S;X)$ may be computed using any fundamental domain of $G(\Z)\backslash V(\R)$, it follows that $$N_\pm(S;X) = \sum_{v\in \mathcal{F} h\cdot R\cap S^\irr} \mu_\pm\left(\frac{A_3(v)}{X}\right) \beta(v_L)$$ for all $h\in G(\R),$ where $\mathcal{F} h\cdot R$ records multiplicity (i.e.\ is a multiset) and we have used that $\beta(v_L) = \frac{1}{2} = \frac{1}{\#|\Stab_{G(\R)}(v_+)|} = \frac{1}{\#|\Stab_{G(\R)}(v_-)|}$ for all $v\in S^\irr\subseteq Y(\Z)^\irr$.

Averaging this equality over $h\in G(\R)$ with respect to $\alpha(h) dh$, we find:
\begin{align*}
N_\pm(S;X) &= \int_{h\in G(\R)} dh\, \sum_{v\in \mathcal{F} h\cdot R\cap S^\irr} \mu_\pm\left(\frac{A_3(v)}{X}\right) \alpha(h) \beta(v_L)
\\&= \int_{h\in G(\R)} dh\, \sum_{v\in S^\irr} \mu_\pm\left(\frac{A_3(v)}{X}\right) \alpha(h) \beta(v_L)\cdot \#|\{g\in \mathcal{F} : gh\cdot v_R = v\}|
\\&= \sum_{v\in S^\irr} \mu_\pm\left(\frac{A_3(v)}{X}\right) \beta(v_L)\int_{h\in G(\R)} dh\, \alpha(h)\cdot \#|\{g\in \mathcal{F} : gh\cdot v_R = v\}|
\\&= \sum_{v\in S^\irr} \mu_\pm\left(\frac{A_3(v)}{X}\right) \beta(v_L)\sum_{\gamma\cdot v_R = v} \int_{h\in G(\R)} dh\, \alpha(h)\cdot \#|\{g\in \mathcal{F} : gh = \gamma\}|
\\&= \sum_{v\in S^\irr} \mu_\pm\left(\frac{A_3(v)}{X}\right) \beta(v_L)\sum_{\gamma\cdot v_R = v} \int_{h\in G(\R)} dh\, \alpha(h)\cdot \#|\{g\in \mathcal{F} : h = g^{-1}\cdot \gamma\}|
\\&= \sum_{v\in S^\irr} \mu_\pm\left(\frac{A_3(v)}{X}\right) \beta(v_L) \sum_{\gamma\cdot v_R = v} \int_{h\in \mathcal{F}^{-1} \gamma} dh\, \alpha(h)
\\&= \sum_{v\in S^\irr} \mu_\pm\left(\frac{A_3(v)}{X}\right) \beta(v_L) \sum_{\gamma\cdot v_R = v} \int_{h\in \mathcal{F}} dh\, \alpha(h^{-1} \gamma)
\\&= \int_{h\in \mathcal{F}} dh\, \sum_{v\in S^\irr} \mu_\pm\left(\frac{A_3(v)}{X}\right) \sum_{\gamma\cdot v_R = v} \alpha(h^{-1} \gamma) \beta(v_L)
\\&= \int_{h\in \mathcal{F}} dh\, \sum_{v\in S^\irr} \mu_\pm\left(\frac{A_3(v)}{X}\right) \sum_{\gamma\cdot v_R = h^{-1} v} \alpha(\gamma) \beta(v_L)
\\&= \int_{h\in \mathcal{F}} dh\, \sum_{v\in S^\irr} w_\pm(h^{-1} v; X).
\end{align*}

Thus
\begin{equation}\label{keyexp}
N_\pm(S; X) = \int_{g\in \mathcal{F}} \sum_{v\in S\cap Y(\Z)^\irr} w_\pm(g^{-1} v; X)\;dg.
\end{equation}
We use (\ref{keyexp}) as the definition of $N_\pm(S;X)$ even if $S\subset V(\Z)$ is not contained in $Y(\Z)$ and even if $S$ is not $G(\Z)$-invariant. Note that in all cases, $N_\pm(S;X) = N_\pm(S \cap Y(\Z); X)$. 

For $S\subseteq V(\Z)$ and $u = (u_1,u_2) \in I(t)$, let $$P_\pm(u,t,X;S) := \sum_{v\in S\cap Y(\Z)} w_\pm(a_t^{-1} n_u^{-1}\cdot v;X).$$
Then (\ref{keyexp}) may be re-expressed as
\begin{align}\label{keyexp2}
N_\pm(S; X) &=
\int_{g = n_ua_t \in N'(t)A'}
P_\pm({u}, {t}, X; S^\irr)
\;dg\, \\
&=\int_{t_1,t_2=\sqrt{\frac{\sqrt3}2}}^\infty
\int_{\scriptstyle{u_1\in I(t_1)}\atop\scriptstyle{u_2\in I(t_2)}}
P_\pm({u}, {t}, X; S^\irr)\;
t_1^{-2}t_2^{-2}
du_1 \,du_2 \,d^\times t_1 \,d^\times t_2\,.
\end{align}

\subsubsection{A sufficient condition for reducibility}

The following lemma, which is \cite[Lemma~6.2]{BhargavaHo2}, gives sufficient conditions for an element in $V(\Z)$ to be reducible. 

\begin{lemma}\label{hyper3red}
Let $v=(F_1,F_2)=(r_1,\ldots,r_8)\in V(\Q)$ be an element such that either $r_1=r_2=0$ or $r_1=r_5=0$. Then $v$ is reducible.
\end{lemma}

\begin{proof}
In case (i), we see that $w_2$ is a factor of $\Disc(w_1 F_1 - w_2 F_2)$.  In case (ii), by replacing the cubic form $F_1$ by a suitable $\Q$-linear combination of $F_{1}$ and $F_{2}$, we may transform $v$ by an element of $G(\Q)$ so that $r_{2}$ is zero.  Since $r_{1}$ will remain zero, we are then in case (i).  Hence $v$ is reducible in either case.
\end{proof}

\subsubsection{Cutting off the cusp}

\begin{lemma}\label{the side lengths of the relevant box lemma}
Let $v = (r_1, \ldots, r_8)\in V(\R)$ be such that $w_+(a_t^{-1} n_u^{-1} v; X)\neq 0$. Then 
\begin{equation}
\begin{array}{llll}
|r_1|\ll t_1^{-1} t_2^{-3} X^{\frac{1}{6}},&
|r_2|\ll t_1^{-1} t_2^{-1} X^{\frac{1}{6}},&
|r_3|\ll t_1^{-1} t_2 X^{\frac{1}{6}},&
|r_4|\ll t_1^{-1} t_2^{3} X^{\frac{1}{6}}\\[.1in]
|r_5|\ll t_1 t_2^{-3} X^{\frac{1}{6}},&
|r_6|\ll t_1 t_2^{-1} X^{\frac{1}{6}},&
|r_7|\ll t_1 t_2 X^{\frac{1}{6}},&
|r_8|\ll t_1 t_2^{3} X^{\frac{1}{6}}.
\label{the side lengths of the relevant box}
\end{array}
\end{equation}
\end{lemma}\noindent

\begin{proof}
This follows by computing the action of $a_t$ on each coordinate of $V(\R)$ and noting that $n_u$ lies in a compact set.
\end{proof}

For example, we have $P_\pm(u,t,X; S) = \sum_{v\in S\cap Y(\Z) : ||v||_\infty\ll t_1 t_2^3 X^{1/6}} w_\pm(a_t^{-1} n_u^{-1} v; X)$, where $||\cdot||_\infty$ denotes the usual $L^\infty$-norm.

\begin{proposition}\label{deep}
Let $V_0 := \left\{\left(r_1 X^3 + \cdots + r_4 Y^3, r_5 X^3 + \cdots + r_8 Y^3\right)\in V : r_1 r_8 = 0\right\}$.
Then 
$$N_\pm(V_0(\Z);X)\ll_\epsilon X^{8/9+\epsilon}.$$
\end{proposition}

\begin{proof}

We first claim that $P_\pm(u,t,X;V_0(\Z)^\irr)\ll_\epsilon t_1 t_2^3 X^{5/6+\epsilon}$. By Lemma \ref{the side lengths of the relevant box lemma} it suffices to show that the number of irreducible $v\in V_0(\Z)$ satisfying the inequalities (\ref{the side lengths of the relevant box}) in the conclusion of Lemma \ref{the side lengths of the relevant box lemma} is $\ll_\epsilon t_1 t_2^3 X^{5/6 + \epsilon}$.
 
We will prove the claim for the subspace where $r_1$ vanishes---the identical argument produces an even stronger bound for the subspace where $r_8$ vanishes. So let us assume that $r_1 = 0$.

By Lemma~\ref{hyper3red},  if $v\in V(\Z)^\irr$, then $r_2,r_5\neq 0$, and so we must have $t_1^{-1}t_2^{-1} X^{1/6}\gg 1$ and $t_1t_2^{-3} X^{1/6}\gg 1$.  Hence the number of irreducible integer points
satisfying (\ref{the side lengths of the relevant box}) will be nonzero~only~if
\begin{equation}\label{condt1}
t_1t_2\ll X^{1/6} \text{ \, and \, } t_1^{-1}t_2^3\ll X^{1/6}.
\end{equation}
Suppose that (\ref{the side lengths of the relevant box}) holds and the $A_1$-invariant vanishes.  The number of possibilities for the variables $(r_3,r_4,r_5,r_6,r_8)$ is
$$
\ll t_1^{-1} t_2 X^{1/6}
\cdot
t_1^{-1} t_2^{3} X^{1/6}
\cdot
t_1 t_2^{-3}X^{1/6}
\cdot
t_1 t_2^{-1}X^{1/6}
\cdot
t_1 t_2^{3}X^{1/6}
= t_1t_2^3 X^{5/6}.
$$
Once these five variables $(r_3,r_4,r_5,r_6,r_8)$ have been fixed, 
then the condition $A_1=0$ also fixes the value of $r_2 r_7$.
If this value is nonzero, then we conclude that the number of possibilities for the pair $(r_2,r_7)$ is  at most $O_\epsilon(X^\epsilon)$. Hence, the number of irreducible integer points which satisfy (\ref{the side lengths of the relevant box}), have vanishing $A_1$-invariant, and are such that $r_2r_7\neq 0$ is at most $O_\epsilon(t_1t_2^3X^{5/6+\epsilon})$.

If $r_2r_7=0$ (i.e., $r_7=0$), then 
the above estimate on the number of pairs $(r_2,r_7)$ does not apply; but then we could have run the identical argument by fixing all variables except $(r_4,r_5)$, assuming $r_4r_5\neq 0$. If $r_7 = 0$ and $r_4 r_5 = 0$ (i.e., $r_4 = 0$), then the condition $A_1=0$ is equivalent to $r_3r_6=0$. By Lemma \ref{hyper3red},  $r_4=0$ and irreducibility forces $r_3\neq 0$, and so $r_6 = 0$.  Thus the number of possibilities for $(r_1,\ldots,r_8)$ (given that $r_1=r_7 = r_4 = r_6 = 0$) is
\[\ll t_1^{-1} t_2^{-1} X^{1/6}
\cdot
t_1^{-1} t_2 X^{1/6}
\cdot
t_1 t_2^{-3}X^{1/6}
\cdot 
t_1 t_2^{3}X^{1/6} =  
 X^{2/3}.
\]
Combining all cases, we see that the number of irreducible integer points satisfying (\ref{the side lengths of the relevant box}) with $r_1 r_8 = 0$ and vanishing $A_1$-invariant is $O_\epsilon\left(t_1t_2^3 X^{5/6+\epsilon}\right)$. By Lemma \ref{the side lengths of the relevant box lemma}, we then conclude that $P_\pm(u,t,X;V_0(\Z)^\irr)\ll_\epsilon t_1 t_2^3 X^{5/6+\epsilon}$, as claimed.

Therefore, by the definition of $N_\pm(V_0(\Z);X)$, we have 
\begin{equation}\label{redcomp}
\begin{array}{lcl}
N_\pm(V_0(\Z);X)&\!\!\!\ll\!\!\!& \displaystyle\int_{t_2=\sqrt{\frac{\sqrt3}2}}^{X^{1/12}}
\int_{t_1=\sqrt{\frac{\sqrt3}2}}^{X^{1/6}/t_2} 
\int_{\scriptstyle{u_1\in I(t_1)}\atop\scriptstyle{u_2\in I(t_2)}}
P_\pm(u,t,X;V_0(\Z)^\irr)
du_1 \,du_2 \,\frac{d t_1}{t_1^3} \,\frac{d t_2}{t_2^3}\,
\\[.225in]&\!\!\!\ll_\epsilon\!\!\!& \displaystyle
\int_{t_2=\sqrt{\frac{\sqrt3}2}}^{X^{1/12}}
\int_{t_1=\max\left(\sqrt{\frac{\sqrt3}2}, \frac{t_2^3}{X^{1/6}}\right)}^{X^{1/6}/t_2} \:
t_1t_2^3\,X^{5/6+\epsilon}\:\frac{d t_1}{t_1^3} \,\frac{d t_2}{t_2^3}
\\[.225in]&\!\!\!\ll_\epsilon\!\!\!& \displaystyle
\int_{t_2=\sqrt{\frac{\sqrt3}2}}^{X^{1/18}}
\int_{t_1=\sqrt{\frac{\sqrt3}2}}^{X^{1/6}/t_2} \:
t_1t_2^3\,X^{5/6+\epsilon}\:\frac{d t_1}{t_1^3} \,\frac{d t_2}{t_2^3} + \int_{t_2=X^{1/18}}^{X^{1/12}}
\int_{t_1=\frac{t_2^3}{X^{1/6}}}^{X^{1/6}/t_2} \:
t_1t_2^3\,X^{5/6+\epsilon}\:\frac{d t_1}{t_1^3} \,\frac{d t_2}{t_2^3}\\[.225in]
&\!\!\!\ll_\epsilon\!\!\!& X^{8/9 + \epsilon},
\end{array}
\end{equation}
as desired.
\end{proof}

\begin{proposition}\label{shallow}
$$\int_{X^\delta\ll ||\vec{t}||_\infty\ll X^{1/6}} \int_{||\vec{u}||_\infty\ll 1} P_\pm(u,t,X;Y(\Z)^\irr) du_1 du_2 \frac{dt_1}{t_1^3} \frac{dt_2}{t_2^3}\ll_\epsilon X^{1 - 2\delta + \epsilon}.$$
\end{proposition}

\begin{proof}
We follow the proof of Proposition \ref{deep}, except that we may now assume $r_1 r_8\neq 0$. 
The number of possibilities for the six coordinates $(r_2,\ldots,r_7)$ for a point $(r_1,\ldots,r_8)\in Y(\Z)^\irr$ satisfying (\ref{the side lengths of the relevant box}) is
\[
\ll 
t_1^{-1} t_2^{-1} X^{1/6}\cdot t_1^{-1} t_2 X^{1/6}\cdot t_1^{-1} t_2^3 X^{1/6}\cdot t_1 t_2^{-3} X^{1/6}\cdot t_1 t_2^{-1} X^{1/6}\cdot t_1 t_2 X^{1/6} = X.
\]
Once these six variables have been fixed, then the condition $A_1=0$ fixes the nonzero value of $r_1r_8$, and thus $r_1$ and $r_8$ are determined up to at most $O_\epsilon(X^\epsilon)$ possibilities. Therefore, 
\[
P_\pm(u,t,X;Y(\Z)^\irr)\ll_\epsilon 
 X^{1+\epsilon},
\]
yielding the desired result upon integration.
\end{proof}

\subsubsection{A change-of-variable formula}
For each $r \in \R$, let $v_r = (3xy^2,x^3 + ry^3) \in Y(\R)$.  
\begin{proposition}\label{vjac}
There exists a rational constant $\mathcal J$ such that, for any $\psi\in L^1(Y(\R), dy)$, we have
\begin{equation}\label{Jac}
\frac{|\mathcal J|}2\int_{\R}
\int_{G(\R)}\psi(g\cdot v_{A_3})\,dg\,dA_3=\int_{Y(\R)}\psi(y)dy.
\end{equation}
\end{proposition}

\noindent This can be verified by an explicit Jacobian calculation.

A more conceptual proof can also be given as follows. It was proven in 
\cite[Proposition~3.10]{BS1}
that under the local identification  $G(\R)\times\R\times\R\to V(\R)$ (onto its image, which contains $Y(\R)_{\Delta\neq 0}$) given by $(g,A_1,A_3)\mapsto g\cdot s(A_1,A_3)$, where $s: \R^2\to V(\R)$ is a smooth section of the invariants map $V(\R)\to\R^2$, 
we have an equality of differential forms $dg\wedge dA_1\wedge dA_3=c\,dv$ on $V(\R)$; here $c$ is a rational constant and  $dv=dr_1\wedge \cdots \wedge dr_8$ is the Euclidean volume on $V$.  If $D_1=dg\wedge dA_3$ and $D_2=c\,dy=c\,dr_1\wedge \cdots \wedge dr_7/(\partial A_1/\partial r_8)$, then $D_1\wedge dA_1=c D_2\wedge dA_1$, as both are equal to $c\,dv$. We conclude by Lemma~\ref{lem:diff} below that $D_1=-cD_2$ as differential forms on $Y(\R)_{\Delta\neq 0}$. Proposition~\ref{vjac} follows.

\begin{lemma}\label{lem:diff}
Let $M$ be a manifold. Let $k\in \Z^+$. Let $f\in C^\infty(M)$ be such that $df$ is nowhere-vanishing on $M$. Let $\omega\in \Omega^k(M)$ be a $k$-form on $M$. Then: $df\wedge \omega = 0$ if and only if there is an $\widetilde{\omega}\in \Omega^{k-1}(M)$ such that $\omega = df\wedge \widetilde{\omega}$. In particular, if $df\wedge \alpha = df\wedge \beta$, then $\alpha\vert_{\{f = 0\}} = \beta\vert_{\{f = 0\}}$.
\end{lemma}

\begin{proof}
For the forward direction, let $X$ be a vector field on $M$ such that the contraction $i_X(df)\in C^\infty(M)$ vanishes nowhere (e.g., choose a metric and dualize $df$). Then, because
\[
0 = i_X(df\wedge \omega)
= i_X(df)\cdot \omega - df\wedge i_X(\omega),
\]
\noindent
it follows that $\omega = df\wedge \left(\frac{i_X(\omega)}{i_X(df)}\right)$.
The reverse direction is evident by antisymmetry since $df$ is a $1$-form, and the last sentence is then  evident given the equivalence. 
\end{proof}

We may use Proposition \ref{vjac} 
to give a convenient expression for the volume of $\{y\in G(\Z)\backslash Y(\R): |A_3(y)| < 1\}$ with respect to the measure $dy$:
\begin{equation}
\int_{y\in G(\Z)\backslash Y(\R): |A_3(y)| < 1}\!dy=
\frac{|\mathcal J|}2\int_{-1}^1\int_{\FF}dg\,dA_3 
\label{volexp} ={|\mathcal J|}  \Vol(G(\Z)\backslash G(\R)).
\end{equation}

A similar Jacobian calculation / conceptual proof yields the following more general change-of-measure formula:

\begin{proposition}\label{vjac2}
  Let $K$ be $\R$, $\C$, or $\Z_p$ for some prime $p$, and let $\psi\in L^1(Y(K), dy)$. Then there exists a rational
  constant $\J$, independent of $K$ and $\psi$, such that
\begin{equation}
    \int_{Y(K)}\psi(y)dy=|\J|
\int_{\substack{0\neq A_3\in K}}\bigg(\sum_{y\in G(K) \backslash Y_{A_3}(K)}\frac{1}{\#\Stab_{G(K)}(y)}\int_{g\in G(K)}\psi(g\cdot y)dg\bigg)dA_3, \vspace{-.1in}
  \end{equation}
  where $Y_{A_3}(K)$ denotes the set of elements in
  $Y(K)$ having invariant $A_3$.
\end{proposition}

\noindent 
 Proposition~\ref{vjac2} follows from \cite[Proposition~3.12]{BS1} using Lemma~\ref{lem:diff} just as Proposition~\ref{vjac} was deduced from \cite[Proposition~3.10]{BS1}.

\subsubsection{The main body}

To count points in $Y(\Z)$ in the main bodies of our fundamental domains, 
we use the circle/smoothed delta symbol method.
For our application, unlike Heath-Brown~\cite{HB} and  previous treatments of the circle method, we require an estimate of the weighted number of integer points on the quadric $Y(\R)$ in skew boxes, and we require knowledge of the dependence of the error term on the skewness of the box  (i.e., on the parameter $t$). The other key contribution of our treatment is the expression of the singular integral and singular series both in terms of integrals over $\R$ and over $\Z_p$, respectively, with respect to the canonical measure $dy$.

We prove:

\begin{proposition}\label{the main term to integrate}
Let $S\subset Y(\Z)$ be defined by congruence conditions modulo $M$. Then
\[
P_\pm(u,t,X;S)
\!=\! X \cdot\!\int_{y\in Y(\R)} \!\!\!\!\!w_\pm(y;1) 
dy \cdot
\prod_p\!
\int_{\scriptstyle{y\in S_p}}\!\!\!
dy+ O_\epsilon(||t||_\infty^{16} M^{O(1)} X^{2/3 + \epsilon})
\]
\end{proposition}

\begin{proof}
We follow Heath-Brown~\cite{HB}, but, to handle weighted counts in skew regions, we keep careful track of the dependence of the error terms on the skewness parameter $t$. 

We first compute $P_\pm(u,t,X; \vec{v}_0 + M\cdot V(\Z))$ for $\vec{v}_0\in S$, and then will sum over a set of representatives of $S\pmods{M}$ to conclude. So let $\vec{v}_0\in S$, and let $\sigma(\vec{v}) := \vec{v}_0 + M\cdot \vec{v}$. Let~$\lambda := X^{\frac{1}{6}}$. Let $N\in \Z^+$ with $N\asymp \delta^{-2}$.

By Heath-Brown's \cite[Theorem $2$, (1.2)]{HB}, i.e., the key identity of the smoothed delta symbol method, we have 
\begin{align*}
&P_\pm(u,t,X; \vec{v}_0 + M\cdot V(\Z))
=\sum_{\vec{v}\in V(\Z) : A_1(\sigma(\vec{v})) = 0} w_\pm(a_t^{-1} n_u^{-1}\cdot  \sigma(v);X) \\[.05in] &= (1 + O_N(\lambda^{-N}))\lambda^{-2} \sum_{q\geq 1} q^{-8} \!\!\sum_{\vec{c}\in V(\Z)^*}\!\! \bigg(\sum_{u\in (\Z/q)^\times} \sum_{\vec{v}\in V(\Z/q)} e_q(u A_1(\sigma(\vec{v})) + \vec{c}\cdot \vec{v})\bigg)\\&\quad\quad\quad\quad\quad\quad\quad\quad\quad \cdot \bigg(\int_{V(\R)} d\vec{v}\, w_\pm(a_t^{-1} n_u^{-1}\cdot \sigma(\vec{v}); X)\cdot h\bigg(\frac{q}{\lambda}, \frac{A_1(\sigma(\vec{v}))}{\lambda^2}\bigg)\cdot e_q(-\vec{c}\cdot \vec{v})\bigg),
\end{align*}\noindent
where $e(x) := \exp(2\pi i x), e_q(x) := e(x/q)$, we have used his Theorem $1$ to estimate his $c_Q$ (here his $Q$ is our $\lambda$), and we define $h$ as follows: let $w_0(x) := \begin{cases} e^{-\frac{1}{1-x^2}} & |x| < 1\\[-.05in] 0 & |x|\geq 1\end{cases}$,\, $\widetilde{w}_0(x) := \frac{4w_0(4x - 3)}{\int_\R w_0(x) dx}$, and $h(x,y) := \sum_{k\geq 1} \frac{1}{k\cdot x}\left(\widetilde{w}_0(k\cdot x) - \widetilde{w}_0\left(\frac{|y|}{k\cdot x}\right)\right)$. Evidently $h$ is smooth and satisfies $h(x,y)\ll x^{-1}$ for all $y$ and $h(x,y) = 0$ when $x\gg 1 + |y|$.

By \cite[Lemma $25$]{HB},   
\begin{align}\label{exponential sum bound}
\bigg|\sum_{u\in (\Z/q)^\times} \sum_{\vec{v}\in V(\Z/q)} e_q(uA_1(\sigma(\vec{v})) + \vec{c}\cdot \vec{v})\bigg|\ll q^5
\end{align}\noindent
(we will not use sharper analysis because this will suffice for our purposes).

As for the integral, via the change of variables
\begin{align*}
\vec{v}\mapsto \sigma^{-1}\left(\lambda n_u a_t \cdot \sigma(\vec{v})\right) = \lambda n_u a_t \cdot \vec{v} + \frac{(\lambda n_u a_t - \id)\cdot \vec{v}_0}{M}
\end{align*}\noindent
we find:
\begin{align}\label{integration by parts bound}\nonumber
&\int_{V(\R)}  w_\pm(a_t^{-1} n_u^{-1}\cdot \sigma(\vec{v}); X)
h\left(\frac{q}{\lambda}, \frac{A_1(\sigma(\vec{v}))}{\lambda^2}\right) e_q(-\vec{c}\cdot \vec{v})\,d\vec{v}
\\&= \lambda^8 e_{Mq}\left((\lambda n_u a_t - \id)\cdot \vec{v}_0\right)\int_{V(\R)} w_\pm(\sigma(\vec{v}); 1)
h\left(\frac{q}{\lambda}, A_1(\sigma(\vec{v}))\right) e_q(-\lambda ((n_ua_t)^\dagger \cdot\vec{c})\cdot \vec{v})\,d\vec{v}
\end{align}
 where $g^\dagger$ denotes the transpose of $g$. 

Because $w_\pm(\vec{v}; 1) = 0$ when $||\vec{v}||_\infty\gg 1$, it follows that $w_\pm(\sigma(\vec{v}); 1) h\left(\frac{q}{\lambda}, A_1(\sigma(\vec{v}))\right) = 0$ for all $\vec{v}\in V(\R)$ when $q\gg M\lambda$. Thus we assume without loss of generality that $q\ll M\lambda$.

Applying the identity $\int_\R f(x) e(\xi x) dx = -\frac{1}{2\pi i \xi}\int_\R f'(x)e(\xi x)dx$ for $f\in C^\infty_c(\R)$ (proven via integration by parts) $N$ times, it follows that if $||(n_{\vec{u}} a_{\vec{t}})^\dag\cdot \vec{c}||_\infty\gg \frac{q}{\lambda} ||\vec{t}||_\infty \lambda^\delta$, then
\begin{align*}
\int_{V(\R)} d\vec{v}\, w_\pm(\sigma(\vec{v}); 1) h\left(\frac{q}{\lambda}, A_1(\sigma(\vec{v}))\right) e_q(-\lambda ((n_u a_t)^\dagger \cdot\vec{c})\cdot \vec{v})
\ll_N M^N X^{\delta^2 N} (||\vec{t}||_\infty\lambda^\delta)^{-N} \left(\frac{\lambda}{q}\right),
\end{align*}\noindent
where we have used the bound $(\d_x^k h)(x,y)\ll_k x^{-1-k}$ as in \cite[Lemma 5]{HB} as well as $||\mu_\pm^{(k)}||_\infty\ll X^{k\delta^2}$.

Combining (\ref{exponential sum bound}) and (\ref{integration by parts bound}), the sum over $\vec{c}\neq 0$---which we will see contributes only to the error term---is:
\begin{align*}
&\lambda^{-2}\sum_{q\geq 1} q^{-8}\sum_{0\neq \vec{c}\in V(\Z)^*} \bigg(\sum_{u\in (\Z/q)^\times} \sum_{\vec{v}\in V(\Z/q)} e_q(u A_1(\sigma(\vec{v})) + \vec{c}\cdot \vec{v})\bigg)\\&\quad\quad\quad\quad \cdot \int_{V(\R)} d\vec{v}\, w_\pm(a_t^{-1} n_u^{-1} \sigma(\vec{v}); 1) h\left(\frac{q}{\lambda}, A_1(\sigma(\vec{v}))\right) e_q(-\vec{c}\cdot \vec{v})
\\&\ll O_N\left(M^N \lambda^{O(1)-N \delta}\right)
\\&\quad + \lambda^{6+o(1)}\sum_{\frac{\lambda^{1-\delta}}{||\vec{t}||_\infty^4}\ll q\ll M \lambda} q^{-3} \sum_{0 < ||(n_{\vec{u}} a_{\vec{t}})^\dag\cdot \vec{c}||_\infty\ll \frac{q ||\vec{t}||_\infty}{\lambda^{1-\delta}}} \bigg|\int_{\vec{v}\in V(\R)} d\vec{v}\, w_\pm(\sigma(\vec{v}); 1) h\left(\frac{q}{\lambda}, A_1(\sigma(\vec{v}))\right) e_q(-\lambda ((n_{\vec{u}} a_{\vec{t}})^\dag\cdot \vec{c})\cdot \vec{v})\bigg|,
\end{align*}\noindent
the lower bound on $q$ in the first sum of the second term arising from the fact that $$||(n_{\vec{u}}\cdot a_{\vec{t}})^\dag\cdot \vec{c}||_\infty\gg \frac{||\vec{c}||_\infty}{||\vec{t}||_\infty^4},$$ and the upper bound arising from the fact that $h(x,y) = 0$ when $x\gg 1 + |y|$. Inserting absolute values into the integral, and using $h(x,y)\ll x^{-1}$, the sum over $\vec{c}\neq 0$ is:
\begin{align*}
&\ll O_N\Big(M^N \lambda^{O(1) - N \delta}\Big) + \lambda^{7+o(1)}\sum_{\frac{\lambda^{1-\delta}}{||\vec{t}||_\infty^4}\ll q\ll M \lambda} q^{-4} \#\left|\left\{\vec{c}\in \Z^8 : 0 < ||(n_{\vec{u}} a_{\vec{t}})^\dag\cdot \vec{c}||_\infty\ll \frac{q}{\lambda^{1-\delta}} ||\vec{t}||_\infty\right\}\right|.
\\&\ll O_N\Big(M^N \lambda^{O(1) - N \delta}\Big) + \lambda^{7+o(1)} \sum_{\frac{\lambda^{1-\delta}}{||\vec{t}||_\infty^4}\ll q\ll M \lambda} q^{-4} \prod_{i=0}^1\prod_{j=0}^3 \Bigg(1 + \frac{q t_1^{(-1)^i}\cdot t_2^{-3 + 2j}}{\lambda^{1-\delta}} ||\vec{t}||_\infty\Bigg)
\\&\ll O_N\Big(M^N \lambda^{O(1) - N \delta}\Big) + M^{O(1)} t_1^4 t_2^8 \max\bigg(1, \frac{t_1^2}{t_2}\bigg) \max\bigg(1, \frac{t_1^2}{t_2^3}\bigg) \lambda^{4+O(\delta)}
\\&\ll O_N\Bigl(M^N \lambda^{O(1) - N \delta}\Bigr) + M^{O(1)} ||\vec{t}||_\infty^{16}  \lambda^{4 + O(\delta)}
\\&\ll M^{O(1)} ||\vec{t}||_\infty^{16}  \lambda^{4 + O(\delta)}
\end{align*}\noindent
since $N\asymp \delta^{-2}$ and $\delta\asymp 1$.

Therefore, 
\begin{align*}
&(1 + O_N(\lambda^{-N})) \sum_{\vec{v}\in V(\Z) : A_1(\sigma(\vec{v})) = 0} w_\pm(a_t^{-1} n_u^{-1} \sigma(\vec{v}); X)
\\&= \lambda^{-2} \sum_{q\ll M \lambda} q^{-8} \bigg(\sum_{u\in (\Z/q)^\times} \sum_{\vec{v}\in V(\Z/q)} e_q(u A_1(\sigma(\vec{v})))\bigg) \Bigg(\int_{V(\R)} |d\vec{v}|\, w_\pm(a_t^{-1} n_u^{-1} \sigma(\vec{v}); X) h\bigg(\frac{q}{\lambda}, \frac{A_1(\sigma(\vec{v}))}{\lambda^2}\bigg)\Bigg)
\\&\quad\quad + O\Bigl(M^{O(1)} ||\vec{t}||_\infty^{16}  \lambda^{4 + O(\delta)}\Bigr)
\\&= M^{-8} \lambda^6\sum_{q\ll M \lambda} q^{-8} \biggl(\sum_{u\in (\Z/q)^\times} \sum_{\vec{v}\in V(\Z/q)} e_q(u A_1(\sigma(\vec{v})))\biggr) \Biggl(\int_{V(\R)} |d\vec{v}|\, w_\pm(\vec{v}; 1)h\biggl(\frac{q}{\lambda}, A_1(\vec{v})\biggr)\Biggr)
\\&\quad\quad + O\Bigl(M^{O(1)} ||\vec{t}||_\infty^{16}  \lambda^{4 + O(\delta)}\Bigr),
\end{align*}\noindent
where we have performed the change of variables  $\vec{v}\mapsto \lambda n_u a_t\cdot \sigma^{-1}(\vec{v})$.

Now we apply \cite[Lemma $13$]{HB}.  Since $q\ll M \lambda$, we obtain
\begin{align*}
\int_{V(\R)} |d\vec{v}|\, w_\pm(\vec{v}; 1) h\left(\frac{q}{\lambda}, A_1(\vec{v})\right) = \int_{Y(\R)} |dy|\, w_\pm(y; 1) + O_N\bigg(\bigg(\frac{q}{\lambda}\bigg)^N\bigg),
\end{align*}\noindent
which extracts the singular integral.

Thus
\begin{align*}
&\sum_{q\ll M \lambda} q^{-8} \bigg(\sum_{u\in (\Z/q)^\times} \sum_{\vec{v}\in V(\Z/q)} e_q(u\cdot A_1(\sigma(\vec{v})))\bigg) \Bigg(\int_{V(\R)} |d\vec{v}|\, w_\pm(\vec{v}; 1)\cdot h\bigg(\frac{q}{\lambda}, A_1(\vec{v})\bigg)\Bigg)
\\&= \bigg(\int_{Y(\R)} |dy|\, w_\pm(y; 1)\bigg) \sum_{q\ll M \lambda^{1 - \delta}} q^{-8}\sum_{u\in (\Z/q)^\times} \sum_{\vec{v}\in V(\Z/q)} e_q(u A_1(\sigma(\vec{v})))
+ O\Big(M^{O(1)} \lambda^{4 + O(\delta)}\Big)
\\&= \bigg(\int_{Y(\R)} |dy|\, w_\pm(y; 1)\bigg)\sum_{q\geq 1} q^{-8}\sum_{u\in (\Z/q)^\times} \sum_{\vec{v}\in V(\Z/q)} e_q(u A_1(\sigma(\vec{v})))
+ O\Bigl(M^{O(1)} \lambda^{4 + O(\delta)}\Bigr),
\end{align*}\noindent
where we have used  \cite[Lemma $25$]{HB} twice.

We now re-express the singular series as an analogous $p$-adic integral (see also Schmidt \cite{schmidt} and Lachaud \cite{lachaud}). The function $q\mapsto q^{-8}\sum_{u\in (\Z/q)^\times} \sum_{\vec{v}\in V(\Z/q)} e_q(u\cdot A_1(\sigma(\vec{v})))$ is multiplicative,~whence:
\begin{align*}
\sum_{q\geq 1} q^{-8}\sum_{u\in (\Z/q)^\times} \sum_{\vec{v}\in V(\Z/q)} e_q(u\cdot A_1(\sigma(\vec{v})))
= \prod_p \sum_{k\geq 1} p^{-8k} \sum_{u\in (\Z/p^k)^\times} \sum_{\vec{v}\in V(\Z/p^k)} e_q(u\cdot A_1(\sigma(\vec{v}))).
\end{align*}
Now
\begin{align*}
&\sum_{u\in (\Z/p^k)^\times} \sum_{\vec{v}\in V(\Z/p^k)} e_q(u A_1(\sigma(\vec{v})))
\\&= \phi(p^k)\#|\{\vec{v}\in V(\Z/p^k) : p^k\mid A_1(\sigma(\vec{v}))\}| - p^{k-1}\#|\{\vec{v}\in V(\Z/p^k) : p^{k-1}\mid\mid A_1(\sigma(\vec{v}))\}|
\\&= p^k\#|\{\vec{v}\in V(\Z/p^k) : p^k\mid A_1(\sigma(\vec{v}))\}| - p^{k-1} \#|\{\vec{v}\in V(\Z/p^k) : p^{k-1}\mid A_1(\sigma(\vec{v}))\}|.
\end{align*}
But $A_1(\sigma(\vec{v}))\pmods{p^{k-1}}$ only depends on $\vec{v}\pmods{p^{k-1}}$. Therefore, 
\begin{align*}
&\sum_{u\in (\Z/p^k)^\times} \sum_{\vec{v}\in V(\Z/p^k)} e_q(u\cdot A_1(\sigma(\vec{v})))
\\&= p^k\cdot \#|\{\vec{v}\in V(\Z/p^k) : p^k\mid A_1(\sigma(\vec{v}))\}| - p^{k+7}\cdot \#|\{\vec{v}\in V(\Z/p^{k-1}) : p^{k-1}\mid A_1(\sigma(\vec{v}))\}|.
\end{align*}
Since the sum telescopes,
\begin{align*}
\sum_{k\geq 1} p^{-8k} \sum_{u\in (\Z/p^k)^\times} \sum_{\vec{v}\in V(\Z/p^k)} e_q(u\cdot A_1(\sigma(\vec{v})))
&= \lim_{k\to \infty} p^{-7k}\cdot \#|\{\vec{v}\in V(\Z/p^k) : A_1(\sigma(\vec{v}))\equiv 0\pmods{p^k}\}|
\\&= \lim_{k\to \infty} \frac{1}{p^{-k}}\int_{\vec{v}\in V(\Z_p) : |A_1(\sigma(\vec{v}))|_p\leq p^{-k}} |d\vec{v}|_p
\\&= |M|_p^{-8p}\cdot \lim_{k\to \infty} \frac{1}{p^{-k}}\int_{\substack{\vec{v}\in V(\Z_p) : \vec{v}\equiv \vec{v}_0\pmods{M}, \\|A_1(\vec{v})|_p\leq p^{-k}}} |d\vec{v}|_p,
\end{align*}\noindent
where in the last equality we have made the change of variables $\vec{v}\mapsto \sigma^{-1}(\vec{v})$.

Because $A_1$ is nondegenerate for all $\vec{v}\in V(\Z_p)$, we have $$||(\nabla A_1)(\vec{v})||_\infty\gg ||\vec{v}||_\infty$$ (proven, e.g., via the adjugate). It follows that either $||\vec{v}||_\infty\leq \delta^{-1} p^{-\frac{k}{2}}$, i.e.\ $\vec{v}$ lies in a set of measure $\ll \Bigl(\delta^{-1} p^{-\frac{k}{2}}\Bigr)^8 = \delta^{-8} p^{-4k}$, or else $||(\nabla A_1)(\vec{v})||_\infty\gg \delta^{-1} p^{-\frac{k}{2}}.$

Suppose that $||(\nabla A_1)(\vec{v})||_\infty\gg \delta^{-1} p^{-\frac{k}{2}}$ and furthermore that $|A_1(\vec{v})|_p\leq p^{-k}$. Let $j$ be minimal such that $|(\d_{v_j} A_1)(\vec{v})|_p\gg \delta^{-1} p^{-\frac{k}{2}}$. Then $A_1(\vec{v})\equiv 0\pmods{p (\d_{v_j} A_1)(\vec{v})^2}$, whence by Hensel's lemma there is a unique $\vec{v}'$ with $A_1(\vec{v}') = 0$ and such that $v_i' = v_i$ for $i\neq j$ and $v_j'\equiv v_j\pmods{p (\d_{v_j} A_1)(\vec{v})}$. Thus $||\vec{v}' - \vec{v}||_\infty\leq p^{-1} |(\d_{v_j} A_1)(\vec{v})|_p$ and hence  $|(\d_{v_j} A_1)(\vec{v}')|_p\gg \delta^{-1} p^{-\frac{k}{2}}$ as well.

In reverse, given $\vec{v}'$ such that $A_1(\vec{v}') = 0$, $|(\d_{v_i} A_1)(\vec{v}')|_p\ll \delta^{-1} p^{-\frac{k}{2}}$ for $i < j$, and such that $|(\d_{v_j} A_1)(\vec{v}')|_p\gg \delta^{-1} p^{-\frac{k}{2}}$, the measure of the set of $\vec{v}$ such that $A_1(\vec{v})\equiv 0\pmods{p^k}$, $v_i = v_i'$ for $i\neq j$, and $v_j'\equiv v_j\pmods{p (\d_{v_j} A_1)(\vec{v})}$ is $p^{-k} |(\d_{v_j} A_1)(\vec{v}')|_p^{-1}$, by Taylor expansion.

Thus, for each $1\leq j\leq 8$, via the map $\vec{v} = (v_1, v_2, \ldots, v_8) \mapsto \vec{v}' = (v_1, v_2, \ldots, v_{j-1}, v_j', v_{j+1}, \ldots, v_8)$, we have 
\begin{align*}
&\frac{1}{p^{-k}}\int_{\substack{\vec{v}\in V(\Z_p) : \vec{v}\equiv \vec{v}_0\pmods{M}, \\|A_1(\vec{v})|_p\leq p^{-k},\, |(\d_{v_j} A_1)(\vec{v})|_p\gg \delta^{-1} p^{-\frac{k}{2}}, \\|(\d_{v_i} A_1)(\vec{v})|_p\ll \delta^{-1} p^{-\frac{k}{2}}\, \forall i < j}} |d\vec{v}|_p
\\&= \frac{1}{p^{-k}}\int_{v_1,\ldots, \hat{v}_j, \ldots, v_8\in \Z_p} |dv_1|_p\cdots \widehat{|dv_j|_p}\cdots |dv_8|_p \int_{\substack{v_j\in \Z_p : \vec{v}\equiv \vec{v}_0\pmods{M}, \\|A_1(\vec{v})|_p\leq p^{-k},\, |(\d_{v_j} A_1)(\vec{v})|_p\gg \delta^{-1} p^{-\frac{k}{2}}, \\|(\d_{v_i} A_1)(\vec{v})|_p\ll \delta^{-1} p^{-\frac{k}{2}}\, \forall i < j}} |dv_j|_p
\\&= \int_{\substack{y\in Y(\Z_p) : y\equiv \vec{v}_0\pmods{M},\\|(\d_{v_j} A_1)(y)|_p\gg \delta^{-1} p^{-\frac{k}{2}},\\|(\d_{v_i} A_1)(y)|_p\ll \delta^{-1} p^{-\frac{k}{2}}\, \forall i < j}} \frac{|dv_1|_p\cdots \widehat{|dv_j|_p}\cdots |dv_8|_p}{|(\d_{v_j} A_1)(y)|_p}
= \int_{\substack{y\in Y(\Z_p) : y\equiv \vec{v}_0\pmods{M},\\|(\d_{v_j} A_1)(y)|_p\gg \delta^{-1} p^{-\frac{k}{2}},\\|(\d_{v_i} A_1)(y)|_p\ll \delta^{-1} p^{-\frac{k}{2}}\, \forall i < j}} |dy|_p,
\end{align*}\noindent
where in the last equality we have implicitly used Lemma \ref{lem:diff}.

Since, as mentioned,
\begin{align*}
\frac{1}{p^{-k}}\int_{\substack{\vec{v}\in V(\Z_p) : \vec{v}\equiv \vec{v}_0\pmods{M}, \\|A_1(\vec{v})|_p\leq p^{-k}}} |d\vec{v}|_p
= \sum_{j=1}^8 \frac{1}{p^{-k}}\int_{\substack{\vec{v}\in V(\Z_p) : \vec{v}\equiv \vec{v}_0\pmods{M}, \\|A_1(\vec{v})|_p\leq p^{-k},\, |(\d_{v_j} A_1)(\vec{v})|_p\gg \delta^{-1} p^{-\frac{k}{2}}, \\|(\d_{v_i} A_1)(\vec{v})|_p\ll \delta^{-1} p^{-\frac{k}{2}}\, \forall i < j}} |d\vec{v}|_p
+ O(\delta^{-8} p^{-3k}),
\end{align*}\noindent
it follows that
\[
\lim_{k\to \infty} \frac{1}{p^{-k}}\int_{\substack{\vec{v}\in V(\Z_p) : \vec{v}\equiv \vec{v}_0\pmods{M}, \\|A_1(\vec{v})|_p\leq p^{-k}}} |d\vec{v}|_p
= |M|_p^{-8} \int_{y\in Y(\Z_p) : y\equiv \vec{v}_0\pmods{M}} |dy|_p.
\]

Therefore, on collecting everything together, we have found that:
\begin{align*}
&\quad\,\sum_{\vec{v}\in V(\Z) : A_1(\sigma(\vec{v})) = 0} w_\pm(a_t^{-1} n_u^{-1}\cdot \sigma(\vec{v}); X)
\\[.045in] &= M^{-8} \lambda^6  \bigg(\int_{Y(\R)} |dy|\, w_\pm(y; 1)\bigg) \prod_p |M|_p^{-8} \int_{y\in Y(\Z_p) : y\equiv \vec{v}_0\pmods{M}} |dy|_p
+ O\Big(M^{O(1)} ||\vec{t}||_\infty^{16} \lambda^{4 + O(\delta)}\Big)
\\[.025in]&= \lambda^6 \bigg(\int_{Y(\R)} |dy|\, w_\pm(y; 1)\bigg) \prod_p \int_{y\in Y(\Z_p) : y\equiv \vec{v}_0\pmods{M}} |dy|_p
+ O\Big(M^{O(1)} ||\vec{t}||_\infty^{16} \lambda^{4 + O(\delta)}\Big).
\end{align*}

Finally, we sum over representatives $\vec{v}_0$ of $S\pmods{M}\subseteq Y(\Z/M)$ to conclude that
\begin{align*}
\sum_{\vec{v}\in S} w_\pm(a_t^{-1} n_u^{-1}\cdot \vec{v}; X)
&= \lambda^6 \int_{y\in Y(\R)} |dy|\, w_\pm(y; 1) \prod_p \int_{y\in S_p} |dy|_p + O\Bigl(M^{O(1)}||\vec{t}||_\infty^{16} \lambda^{4 + O(\delta)}\Bigr).
\end{align*}\noindent
Choosing $\delta\asymp 1$ sufficiently small gives the claim.
\end{proof}

\subsubsection{Estimates on reducibility in the main body}

Let $g$ be the completely multiplicative function vanishing outside the squarefree integers such that $$g(p) = \frac{\int_{y\in Y(\Z_p)\cap V(\Z_p)^\irr} dy}{\int_{y\in Y(\Z_p)} dy}$$
for all primes $p$. Note that $0\leq g(p)\leq 1$. Let $h$ be the completely multiplicative function vanishing outside the squarefree integers such that $h(p) := \frac{g(p)}{1 - g(p)}$ if $g(p) < 1$ and $h(p) := 2$ otherwise.

\begin{lemma}\label{density of reducible mod p fellows}
Let $p > 3$ be a prime. Then 
\[g(p)\geq \begin{cases}
\frac{1}{4} + O\!\left(\frac{1}{p}\right) & \mbox{ if } p\equiv 1\pmods{3},\\
\frac{1}{2} + O\!\left(\frac{1}{p}\right) & \mbox{ if } p\equiv 2\pmods{3}.
\end{cases}\]
\end{lemma}

\begin{proof}
By Proposition \ref{vjac2}, we have 
$$1 - g(p) = \frac{\int_{0\neq A_3\in \Z_p} \left(\sum_{y\in G(\Z_p)\backslash Y_{A_3}(\Z_p)^\red} \frac{1}{\#|\Stab_{G(\Z_p)}(y)|}\right) dA_3}{\int_{0\neq A_3\in \Z_p} \left(\sum_{y\in G(\Z_p)\backslash Y_{A_3}(\Z_p)} \frac{1}{\#|\Stab_{G(\Z_p)}(y)|}\right) dA_3},$$ where $Y_{A_3}(\Z_p)^\red := Y_{A_3}(\Z_p)\cap V(\Z_p)^\red$.

For $n \in \Z_p^\times$, let $E_n$ be the elliptic curve $E_n \colon y^2 = x^3 + 16n^2$. For any $v \in Y_n(\Z_p)$, we have 
\[\mathrm{Stab}_{G(\Z_p)}(y) \simeq E_n[2](\Q_p),\] 
and there is a bijection between $G(\Z_p)\backslash Y_n(\Z_p)$ and $E_n(\Q_p)/2E_n(\Q_p)$, by Proposition \ref{prop:cubefreeimpliessoluble}.  Under this bijection, the unique reducible orbit corresponds to the identity element.

If $p \equiv 1 \pmods{3}$ and $16n^2$ is a cube in $\Z_p^\times$, then 
\[\#E_n(\Q_p)/2E_n(\Q_p) = \#E_n[2](\Q_p) = 4,\]
so for such $n$ there are three irreducible orbits for every reducible one.  If $16n^2$ is not a cube, then $\#E_n(\Q_p)/2E_n(\Q_p) = 1$  and there is only the reducible orbit. Since $16n^2$ is a cube for one third of all $n \in \Z_p^{\times}$, we have
\[1 - g(p) = \frac13\cdot \frac14 + \frac23 \cdot 1 + O\Bigg(\frac{1}{p}\Bigg)\]
and hence $g(p) \geq \frac14 + O\Big(\frac{1}{p}\Big)$.

If $p \equiv 2\pmod3$ and $p >2$, then 
\[\#E_n(\Q_p)/2E_n(\Q_p) = \#E_n[2](\Q_p) = 2,\]
so for each $n$, there are two orbits and one of them is reducible. Thus $g(p) \geq \frac12 + O\Big(\frac{1}{p}\Big)$. 
\end{proof}

We now define the usual Selberg sieve weights.
Let $\eta\in \R^+$ with $\eta\asymp 1$ (certainly $\eta = 10^{-10}$ will suffice). Let $R := X^\eta$. Let $D := R^2$. Let $Q := \prod_{p\leq R} p$. Let $J := \sum_{m\leq R} h(m)$. Let $\rho_e := \frac{1}{J}\cdot \frac{\mu(e)}{g(e)}\cdot \sum_{e\mid m\mid P : m\leq R} h(m)$; thus, e.g., $\rho_1 = 1$, $\rho_e = 0$ when $e > R$, and $|\rho_e|\leq 1$.\footnote{This is because $$|\rho_e| = \frac{1}{J}\sum_{e\mid m\mid P : m\leq R} \sum_{e'\mid e} \frac{h(m)}{h(e')} = \frac{1}{J}\sum_{m'\mid Q : m'\leq \frac{R}{e}, (m,e) = 1} \sum_{e''\mid e} h(m'\cdot e'')\leq 1.$$} Let $\lambda_d := \sum_{[e,e'] = d} \rho_e\cdot \rho_{e'}$; thus e.g.\ $|\lambda_m|\leq d_3(m)$, $\lambda_m = 0$ when $m > D$, and $\sum_d \lambda_d\cdot g(d) = \frac{1}{J}$.\footnote{Indeed\begin{align*} \sum_d \lambda_d\cdot g(d) &= \frac{1}{J^2}\sum_{e,e'} \frac{\mu(e)\mu(e')}{g((e,e'))}\biggl(\sum_{e\mid m\mid Q : m\leq R} h(m)\biggr)\biggl(\sum_{e'\mid m'\mid Q : m'\leq R} h(m')\biggr)\\&= \frac{1}{J^2}\sum_{m,m'\mid Q : m, m'\leq R} h(m) h(m')\sum_{e\mid m} \mu(e)\sum_{e''\mid (e,m')} \frac{\mu(e'')}{g(e'')}\sum_{e'''\mid \frac{m'}{(e,m')}} \mu(e'''),\end{align*} and $\sum_{e''\mid m'} \frac{\mu(e'')}{g(e'')} = \frac{\mu(m')}{h(m')}$.}

\begin{proposition}\label{reduciblebody}
\begin{equation}\label{redexp}
\int_{1\ll ||\vec{t}||_\infty\ll X^\delta} \int_{||\vec{u}||_\infty\ll 1} P_\pm({u}, {t}, X; Y(\Z)^\red)\,t_1^{-2} t_2^{-2} du_1\, du_2\, d^\times t_1\, d^\times t_2
\ll_\delta X^{1 - \Omega(1) + O(\delta)}.
\end{equation}
\end{proposition}

\begin{proof}
Let $\lambda := X^{\frac{1}{6}}$. Evidently $$\bigg(\sum_{e\mid Q} \rho_e\cdot \mathbbm{1}_{\cap_{p\mid e} V(\Z_p)^\irr}(y)\bigg)^2\geq \mathbbm{1}_{\cap_{p\leq R} V(\Z_p)^\red}(y)$$ (if $y\not\in \bigcap_{p\leq R} V(\Z_p)^\red$ there is nothing to prove, and if $y\in \bigcap_{p\leq R} V(\Z_p)^\red$ then both sides are $1$).
Summing over $y\in Y(\Z)$, we obtain
\begin{align*}
P_\pm(u,t,X; \bigcap_{p\leq R} V(\Z_p)^\red) &= \sum_{y\in Y(\Z)\cap \bigcap_{p\leq R} V(\Z_p)^\red} w_\pm(a_t^{-1} n_u^{-1}\cdot y; X)
\\&\leq \sum_{m\mid Q} \lambda_m \sum_{y\in Y(\Z)\cap \bigcap_{p\mid m} V(\Z_p)^\irr} w_\pm(a_t^{-1} n_u^{-1}\cdot y; X).
\end{align*}
By Theorem \ref{the main term to integrate},
\begin{align*}
P_\pm(u,t,X; \bigcap_{p\mid m} V(\Z_p)^\irr) &\!=\! X \!\int_{y\in Y(\R)}\!\!\!\!\!\! w_\pm(y; 1) dy \, \prod_{p\nmid m}\! \int_{y\in Y(\Z_p)} \!\!\!\!\!dy \, \prod_{p\mid m} \!\int_{y\in Y(\Z_p)\cap V(\Z_p)^\irr} \!\!\!\!\!\!\!\!dy
\!+\! O\Big(||\vec{t}||_\infty^{16} m^{O(1)} X^{\frac{2}{3}}\Big).
\end{align*}\noindent
Thus, on using $|\lambda_m|\leq d_3(m)$ and $\lambda_m = 0$ when $m > D$, we find:
\begin{align*}
&P_\pm(u,t,X; \bigcap_{p\leq R} V(\Z_p)^\red)
\\&\leq X\cdot \int_{y\in Y(\R)} w_\pm(y; 1) dy\cdot \prod_p \int_{y\in Y(\Z_p)} dy\cdot \sum_{m\mid Q} \lambda_m g(m) + O\bigg(X^{\frac{2}{3}}||\vec{t}||_\infty^{16} \sum_{m\leq D} d_3(m) m^{O(1)}\bigg)
\\&= \frac{X}{J}\cdot \int_{y\in Y(\R)} w_\pm(y; 1) dy\cdot \prod_p \int_{y\in Y(\Z_p)} dy + O\big(X^{\frac{2}{3} + O(\eta)} ||\vec{t}||_\infty^{16}\big).
\end{align*}

Finally, because $g(p)\gg 1$ when $p\gg 1$ (Lemma \ref{density of reducible mod p fellows}) it follows that
\[
J = \sum_{m\leq R} h(m)\geq \sum_{m\leq R : (m, O(1)) = 1} O(1)^{-\#|\{p\mid m : p\gg 1\}|}
\gg \frac{R}{\log{R}},
\]\noindent
for example, whence we conclude from $$P_\pm(u,t,X; Y(\Z)^\red)\leq P_\pm(u,t,X; Y(\Z)\cap \bigcap_{p\leq R} V(\Z_p)^\red)$$ that $$P_\pm(u,t,X; Y(\Z)^\red)\ll X^{1-\eta + o(1)} + X^{\frac{2}{3} + O(\eta)} ||\vec{t}||_\infty^{16}.$$ Integrating over $\vec{u}$ and $\vec{t}$ then gives the claim.
\end{proof}

\subsubsection{Putting together the cusp and main body}

We now prove Theorem \ref{thmcount}.

\begin{proof}[Proof of Theorem $\ref{thmcount}$.]
Since $N_-(S;X)\leq N(S;X)\leq N_+(S;X)$ it suffices to prove the same asymptotic for $N_\pm(S;X)$ instead.

By (\ref{keyexp2}),
\begin{align*}
&N_\pm(S;X) \\&= \int_{t_1,t_2=\sqrt{\frac{\sqrt3}2}}^\infty
\int_{\scriptstyle{u_1\in I(t_1)}\atop\scriptstyle{u_2\in I(t_2)}}
P_\pm(u,t,X; S^\irr)
t_1^{-2}t_2^{-2}
du_1 \,du_2 \,d^\times t_1 \,d^\times t_2
\\&= \int_{t_1,t_2=\sqrt{\frac{\sqrt3}2}}^{X^{O(1)}}
\int_{\scriptstyle{u_1\in I(t_1)}\atop\scriptstyle{u_2\in I(t_2)}}
P_\pm({u}, {t}, X; S^\irr)\;
t_1^{-2}t_2^{-2}
du_1 \,du_2 \,d^\times t_1 \,d^\times t_2
\\&= \Biggl(\int_{1\ll ||\vec{t}||_\infty < X^\delta}\int_{\scriptstyle{u_1\in I(t_1)}\atop\scriptstyle{u_2\in I(t_2)}} + \int_{X^\delta\leq ||\vec{t}||_\infty\ll X^{O(1)}}\int_{\scriptstyle{u_1\in I(t_1)}\atop\scriptstyle{u_2\in I(t_2)}}\Biggr) P_\pm({u}, {t}, X; S^\irr)\;
t_1^{-2}t_2^{-2}
du_1 \,du_2 \,d^\times t_1 \,d^\times t_2,
\end{align*}\noindent
where in the second equality we used that if $||\vec{t}||_\infty\gg X^{\frac{1}{6}}$ then by (\ref{condt1}) there are no irreducible integer points $y\in Y(\Z)$ with $w_\pm(a_t^{-1} n_u^{-1}\cdot y; X)\neq 0$.

By Proposition \ref{shallow} the second integral is $\ll X^{1 - \Omega(1)}$.

Let us further write the first integral as:
\begin{align*}
&\int_{1\ll ||\vec{t}||_\infty < X^\delta}\int_{\scriptstyle{u_1\in I(t_1)}\atop\scriptstyle{u_2\in I(t_2)}} P_\pm({u}, {t}, X; S^\irr) t_1^{-2}t_2^{-2}
du_1 \,du_2 \,d^\times t_1 \,d^\times t_2
\\&= \int_{1\ll ||\vec{t}||_\infty < X^\delta}\int_{\scriptstyle{u_1\in I(t_1)}\atop\scriptstyle{u_2\in I(t_2)}} P_\pm({u}, {t}, X; S) t_1^{-2}t_2^{-2}
du_1 \,du_2 \,d^\times t_1 \,d^\times t_2
\\&\quad\quad - \int_{1\ll ||\vec{t}||_\infty < X^\delta}\int_{\scriptstyle{u_1\in I(t_1)}\atop\scriptstyle{u_2\in I(t_2)}} P_\pm({u}, {t}, X; S^\red) t_1^{-2}t_2^{-2}
du_1 \,du_2 \,d^\times t_1 \,d^\times t_2.
\end{align*}
By Proposition \ref{reduciblebody} the second integral is $\ll X^{1 - \Omega(1)}$.
As for the first integral, applying Proposition~\ref{the main term to integrate} and observing that the resulting error terms also integrate to $\ll M^{O(1)} X^{1 - \Omega(1)}$ and that the main term is independent of $t$ and $u$, we need only observe that
\begin{align*}
&\int_{1\ll ||\vec{t}||_\infty < X^\delta}\int_{\scriptstyle{u_1\in I(t_1)}\atop\scriptstyle{u_2\in I(t_2)}} t_1^{-2}t_2^{-2}
dy\, du_1 \,du_2 \,d^\times t_1 \,d^\times t_2
\\&= \Biggl(\int_{t_1, t_2 = \sqrt{\frac{\sqrt{3}}{2}}}^\infty \int_{\scriptstyle{u_1\in I(t_1)}\atop\scriptstyle{u_2\in I(t_2)}} - \int_{||\vec{t}||_\infty\geq X^\delta} \int_{\scriptstyle{u_1\in I(t_1)}\atop\scriptstyle{u_2\in I(t_2)}}\Biggr) t_1^{-2}t_2^{-2}
dy\, du_1 \,du_2 \,d^\times t_1 \,d^\times t_2
\\&= \Vol(G(\Z)\backslash G(\R)) + O(X^{1 - \Omega(1)})
\end{align*}\noindent
and that (by Proposition \ref{vjac})
\begin{align*}
&\int_{y\in Y(\R)} w_\pm(y; 1) dy
\\&= \frac{|\mathcal{J}|}{2}\int_{g\in G(\R)} \int_{A_3\in \R} w_\pm(g\cdot v(0, A_3); 1) dA_3\, dg
\\&= \frac{|\mathcal{J}|}{2}\int_{g\in G(\R)} \int_{A_3\in \R} \mu_\pm(A_3) \sum_{hg\cdot v(0,A_3) = v(0,A_3)} \alpha(h)\beta(v_{\sgn(A_3)}) dA_3\, dg
\\&= \frac{|\mathcal{J}|}{4}\int_{g\in G(\R)} \Biggl(\int_{A_3\in \R^+} \mu_\pm(A_3) \sum_{h\in \Stab_{G(\R)}(v_+)} \alpha(h g^{-1}) dA_3 + \int_{A_3\in \R^-} \mu_\pm(A_3) \sum_{h\in \Stab_{G(\R)}(v_-)} \alpha(h g^{-1}) dA_3\Biggr)\, dg
\\&= \frac{|\mathcal{J}|}{4}\Biggl(\int_{A_3\in \R^+} \mu_\pm(A_3) \sum_{h\in \Stab_{G(\R)}(v_+)} + \int_{A_3\in \R^-} \mu_\pm(A_3) \sum_{h\in \Stab_{G(\R)}(v_-)}\Biggr) \int_{g\in G(\R)} \alpha(h g^{-1}) dg\, dA_3
\\&= \frac{|\mathcal{J}|}{2}\int_{A_3\in \R} \mu_\pm(A_3)
\\&= |\mathcal{J}| + O(X^{-\delta^2}),
\end{align*}\noindent
where we have used that $\int_{g\in G(\R)} \alpha(h g^{-1}) dg = 1$ via the change of variables $g\mapsto g^{-1} h$.

It follows that
\begin{align*}
\int_{1\ll ||\vec{t}||_\infty < X^\delta}\int_{\scriptstyle{u_1\in I(t_1)}\atop\scriptstyle{u_2\in I(t_2)}} \int_{y\in Y(\R)} w_\pm(y; 1) t_1^{-2}t_2^{-2}
dy\, du_1 \,du_2 \,d^\times t_1 \,d^\times t_2
&= |\mathcal{J}|\cdot \Vol(G(\Z)\backslash G(\R)) + O(X^{1 - \Omega(1)})
\\[-.05in]&= \int_{\scriptstyle{\substack{y\in G(\Z)\backslash Y(\R), \\|A_3(y)|<1}}} dy + O(X^{1 - \Omega(1)}),
\end{align*}\noindent
and so by Proposition \ref{the main term to integrate} we are done.
\end{proof}

\subsection{A uniformity estimate}

We will make use of the following theorem of Browning--Heath-Brown \cite[Theorem $1.1$]{browning-heath-brown}.

\begin{theorem}[Browning--Heath-Brown]\label{browning heath brown}
Let $X\subseteq \P^m$ be a hypersurface defined over $\Q$ by a quadratic form of rank at least $5$. Let $Z\subseteq X$ be a codimension $2$ subvariety defined over $\Q$, let $\mathscr{Z}$ be its scheme-theoretic closure in $\P^m_\Z$, and let $Z_p := \mathscr{Z}\otimes_\Z \F_p$. Then for any $\epsilon > 0$ there exists a constant $c_{\epsilon,X,Z} > 0$ depending only on $X$, $Z$, and $\epsilon$, such that the number of $x\in X(\Q)$ of height $H_{\P^m}(x) < B$ which specialize to a point in $Z_p(\F_p)$ for some $p > M$ is at most $$c_{\eps,X,Z} B^\eps \left(\frac{B^{m-1}}{M\log{M}} + B^{m-1-1/m}\right).$$
\end{theorem}\noindent
Here $H_{\P^m}([x_0, \ldots, x_m]) := \prod_v \max_i |x_i|_v$ as usual.

We now prove the desired uniformity estimate. 
Let $W_p(Y)$ denote the set of $y\in Y(\Z)$ such that $p^2 \mid A_3(y)$.    

\begin{proposition}\label{uniprop}
Suppose $M\in \Z^+$ with $M\gg 1$. Then $$N(\bigcup_{p > M} W_p(Y); X)\ll \frac{X^{1 + O(\delta^3)}}{M\log{M}} + X^{1 - \Omega(1)}.$$
\end{proposition}

\begin{proof}
By Proposition \ref{shallow},
\begin{align*}
&N(\bigcup_{p > M} W_p(Y); X)
\\&\leq N_+(\bigcup_{p > M} W_p(Y); X)
\\&=\int_{t_1,t_2=\sqrt{\frac{\sqrt3}2}}^\infty
\int_{\scriptstyle{u_1\in I(t_1)}\atop\scriptstyle{u_2\in I(t_2)}}
P_+({u}, {t}, X; \bigcup_{p > M} W_p(Y)^\irr)\;
t_1^{-2}t_2^{-2}
du_1 \,du_2 \,d^\times t_1 \,d^\times t_2
\\&= \Biggl(\int_{\sqrt{\frac{\sqrt{3}}{2}}\leq ||\vec{t}||_\infty < X^{\delta^3}} + \int_{||\vec{t}|| > X^{\delta^3}}\Biggr)\int_{\scriptstyle{u_1\in I(t_1)}\atop\scriptstyle{u_2\in I(t_2)}}
P_+({u}, {t}, X; \bigcup_{p > M} W_p(Y)^\irr)\;
t_1^{-2}t_2^{-2}
du_1 \,du_2 \,d^\times t_1 \,d^\times t_2
\\&= \int_{\sqrt{\frac{\sqrt{3}}{2}}\leq ||\vec{t}||_\infty < X^{\delta^3}}\int_{\scriptstyle{u_1\in I(t_1)}\atop\scriptstyle{u_2\in I(t_2)}}
P_+({u}, {t}, X; \bigcup_{p > M} W_p(Y)^\irr)\;
t_1^{-2}t_2^{-2}
du_1 \,du_2 \,d^\times t_1 \,d^\times t_2
+ O(X^{1 - \Omega(1)}).
\end{align*}

Thus it suffices to prove that, for $1\ll t_i\ll X^{\delta^3}$, $$P_+(u,t,X; \bigcup_{p > M} W_p(Y)^\irr)\ll \frac{X^{1 + O(\delta^3)}}{M\log{M}} + X^{1 - \Omega(1)},$$ because inserting this bound into the above integral yields the claim. Note also that for such $\vec{t}=(t_1,t_2)$, by Lemma \ref{the side lengths of the relevant box lemma} if $w_+(a_t^{-1} n_u^{-1}\cdot v; X)\neq 0$, then, writing $v=(r_1, \ldots, r_8)$, we have $H_{\P^7}([r_1 : \cdots : r_8])\ll X^{1/6 + O(\delta^3)}$.

Let now $v = (F_1,F_2)=(r_1,\ldots,r_8) \in V(\Z)$ be such that $w_+(a_t^{-1} n_u^{-1}\cdot v; X)\neq 0$ and $p > M$ be such that $p^2\mid A_3(v)$. Let also $f = \Disc(w_1F_1(x,y) - w_2F_2(x,y)) \in \Z[w_1,w_2]$. Since $f$ is a binary quartic form, we may refer to its invariants $I(f)$ and $J(f)$ \cite{BS1}.  By a computation, we have that  $A_1(v) \mid I(f)$.  If $v \in Y(\Z)$, then $A_1(v) = 0$ and hence $I(f) = 0$.  It then follows from degree considerations that $J(f)$ is a nonzero rational constant times $A_3(v)^2$.  Since $p^2 \mid A_3(v)$, it follows that $p^4 \mid J(f)$.

Since $I(f) = 0$, the discriminant of $f(w_1,w_2)$ is a nonzero rational constant times $J(f)^2$, so is divisible by $p^8$. In particular,  $f(w_1,w_2)$ has a repeated root over $\F_p$.  However,  because $I(aw_1^4 + bw_1^3w_2 + cw_1^2w_2^2) = c^2$, we see that any repeated root of $f(w_1,w_2)$ must actually be at least a triple root since $I(f) = 0$.  On the other hand, if $f(w_1,w_2)$ has a quadruple root modulo $p$ or is a multiple of $p$, then $v$ reduces to an $\F_p$-point on a fixed  codimension $2$ subvariety $Z\subseteq Y$ defined over $\Z$ and we may apply Theorem \ref{browning heath brown} (with $B = X^{1/6 + O(\delta^3)}$) to bound the number of such $v$.

So it remains to consider the case where $f(w_1,w_2)$ has splitting type $(1^3 1)$.  By a change of basis over $\Z_p$ which is uniquely determined up to $(w_1,w_2)\mapsto (t w_1, t^{-1} w_2)$ mod $p$, we may assume that $f(w_1,w_2) =bw_1^3w_2 + dw_1w_2^3 + ew_2^4$ with $p\nmid b$, $p\mid d$, $p \mid e$. Since $I(f) = -3bd = 0$, we see that $d =0$.  We then see that $J(f) = -27b^2e$, whence, because the discriminant of $f$ is a multiple of $J(f)^2$, we find that $p^4 \mid e = \Disc(F_2)$. Now, if $p \mid F_2$, then we may again use Theorem \ref{browning heath brown} as before. So we may assume that $F_2$ is primitive over $\Z_p$, and hence splits as $(1^2 1)$ or $(1^3)$ over $\F_p$.  

If $F_2$ has splitting type $(1^21)$, then, after a second (independent) change of basis over $\Z_p$, we have $F_2(x,y) = 3r_6x^2y + r_8y^3$ with $p\nmid r_6$ and $p\mid r_8$.
Since $0 = A_1(v)\equiv 3r_6r_3\pmods{p^4}$, it follows that $p^4 \mid r_3$. We also compute $0 = d\equiv 4 r_4r_6^3\pmods{p^4}$, so that $p^4 \mid r_4$. Hence 
\[ w_1F_1(x,y) -  w_2F_2(x,y)\equiv x^2(r_1 w_1 x + 3(r_2 w_1 - r_6 w_2) y)\pmods{p}\] 
and so $f(w_1,w_2) = \Disc( w_1F_1(x,y) - w_2F_2(X,Y))\equiv 0\pmods{p}$, which is not of type $(1^31)$, a contradiction.

The final case is when $F_2(x,y) = r_5x^3 + 3r_6x^2y + 3r_7xy^2 + r_8y^3$ is $(1^3)$ at $p$. After a $\Z_p$-change of basis, we may assume $p\nmid r_5$, $r_6=0$, and $p \mid r_7,r_8$.  That $p^4 \mid \Disc(F_2)$ means that $p^2 \mid r_8$ and thus $p^2 \mid r_7$.  Then $A_1(v)=0$ implies that $p^2\mid r_4$. Since $f$ is primitive and $F_2(x,y)\equiv r_5 x^3\pmods{p}$, we have $p\nmid r_3$. We may again change basis over $\Z_p$ to ensure that moreover $r_2 = 0$. (This determines the mod $p$ reduction of our basis over $\Z_p$ up to transformations of the form $(x,y)\mapsto (t x, t^{-1} y)$.)  Since $0 = a = \Disc(F_1)\equiv 4 r_1 r_3^3\pmods{p^4}$, it follows that $p^4\mid r_1$. Now let $\tau(v) = \gamma v$, where $\gamma = (\diag(\sqrt{p}, 1/\sqrt{p}), \diag(\sqrt{p}, 1/\sqrt{p}))$. Explicitly, the action of $\gamma$ is given by $$(F_1,F_2)\mapsto (\sqrt{p}F_1(\sqrt{p}x,y/\sqrt{p}), F_2(\sqrt{p}x,y/\sqrt{p})/\sqrt{p}),$$ or in terms of coefficients: $(r_1, \ldots, r_8)\mapsto (p^2 r_1, pr_2, r_3, r_4/p, pr_5,r_6,r_7/p,r_8/p^2)$.  The pair $\tau(v)$ has the same invariants, but the associated binary quartic has been replaced by $(p^2 a, pb,c,d/p,e/p^2)$, which is divisible by $p$, and indeed is $p$ times a $(1^3 1)$ binary quartic. Note also that $$(p^2 r_1, pr_2, r_3, r_4/p, pr_5,r_6,r_7/p,r_8/p^2)\equiv (0, 0, r_3, 0, 0, 0, 0, r_8/p^2)\pmods{p},$$ so that the first binary cubic in this pair is also of type $(1^2 1)$.

Now note that the above argument did not rely on the exact vanishing of $a$, $c$, $d$, $r_6$, etc., but rather needed only that they be divisible by, e.g., $p^{10}$. Thus (by mod-$p^{10}$ weak approximation), to produce this transformation, we may have first chosen changes of basis over $\Z$ rather than over $\Z_p$.

We claim that the association above $v \mapsto \tau(v)$ gives a well-defined $A_3$-preserving injection from $G(\Z)$-orbits of $v \in Y(\Z)$ with $f$ of splitting type $(1^31)$ and $F_2$ of type $(1^3)$ to a set of $G(\Z)$-orbits $Y(\Z)$ with binary quartic equal to $p$ times a $(1^3 1)$ form and $F_1$ a $(1^2 1)$ form. To see that $\tau$ is well-defined at the level of $G(\Z)$-orbits, note that all choices made above in changing basis over $\Z$
differ by elements of $G(\Z)$ that are congruent to a diagonal matrix mod $p$. Since, for $t\in \Z_p^\times$,
\begin{equation}\label{integrality of a conjugate}
\diag(\sqrt{p},1/\sqrt{p})^{-1}\cdot (\diag(t,t^{-1}) + p\cdot M_2(\Z_p))\cdot \diag(\sqrt{p},1/\sqrt{p})\subseteq \GL_2(\Z_p),
\end{equation}
these give the same $G(\Z)$-orbit (integrality at $p$ is the only thing to be checked). To see that $\tau$ is injective, suppose we have $v$ and ${v'}$ such that $\tau({v'}) = g\cdot \tau(v)$ with $g = (g_1, g_2)\in G(\Z)$.  Because of the uniqueness up to $(w_1,w_2)\mapsto (t w_1, t^{-1} w_2)$ of the change of basis over $\F_p$ taking a $(1^3 1)$ mod-$p$ binary quartic (namely the binary quartic of $\tau(v)$ divided by $p$) to a multiple of $w_1^3 w_2\pmods{p}$, it follows that $g_2\pmods{p}$ is diagonal, whence by \eqref{integrality of a conjugate} we may assume without loss of generality (by changing ${v'}$ by an element of $\{\id\}\times \SL_2(\Z)$) that $g_2 = \id$ and hence $f = \widetilde{f}$. Similarly, because of \eqref{integrality of a conjugate} and the uniqueness up to $(x,y)\mapsto (t x, t^{-1} y)$ of the change of basis over $\F_p$ taking a $(1^2 1)$ mod-$p$ binary cubic (namely the first binary cubic of $\tau(v)$) to a multiple of $x y^2\pmods{p}$, it follows that $g_1\pmods{p}$ is diagonal, whence by \eqref{integrality of a conjugate} we conclude that ${v'}\in G(\Z)\cdot v$, and so $\tau$ is indeed injective at the level of $G(\Z)$-orbits.

The desired result now follows, since we have already bounded the set of $v\in Y(\Z)$ for which $f$ is a multiple of $p$, $p^2 \mid A_3(v)$, and $|A_3(v)| < X$. 
\end{proof}

\subsection{Proof of the main counting theorem}

We first state the following weighted version of
Theorem \ref{thmcount}.

\begin{theorem}\label{cong3}
Let $\phi_p : Y(\Z_p)\to \R$ be locally constant, $G(\Z)$-invariant, and such that $\phi_p = 1$ for all but finitely many $p$.
Let
  $N_\phi(Y(\Z);X)$ denote the weighted number of irreducible
  $G(\Z)$-orbits in $Y(\Z)$ having $|A_3|$ bounded by $X$,
  where each orbit $G(\Z)\cdot y$ is counted with weight
  $\phi(y):=\prod_p \phi_p(y)$. Then 
\begin{equation}
N_\phi(Y(\Z);X)
  = X\cdot \int_{\scriptstyle{y\in G(\Z) \backslash Y(\R)}\atop\scriptstyle{|A_3(y)|<1}}
dy
\cdot 
  \prod_p \int_{y\in Y({\Z_{p}})} \phi_{p}(y)\,dy+ O_\phi\Big(X^{1 - \Omega(1)}\Big).
\end{equation}
\end{theorem}

\begin{proof}
We apply Theorem \ref{thmcount} to the level sets of $\phi$.
\end{proof}

To prove Theorem~\ref{thsqfreetc}, we must extend Theorem~\ref{cong3} to weight functions $\phi=\prod \phi_p$ that are acceptable but nontrivial at infinitely many primes.  It is for this extension that the uniformity estimate in Proposition~\ref{uniprop} is needed. 

\begin{proof}[Proof of Theorem $\ref{thsqfreetc}$.]
We may repeat the argument in \cite[\S2.7]{BS1}, but with the uniformity estimate \cite[Theorem~2.13]{BS1} there replaced by Proposition~\ref{uniprop}. 
\end{proof}

\section{The average size of the $2$-Selmer group in a cubic twist family
\label{computing the average of two selmer}}
Let $(A, \L)$ be a polarized abelian variety over $\Q$ with a $\mu_3$-action, as in Section \ref{subsec:A-torsors}.  For each integer $n \geq 1$, let $(A_n, \L_n)$ be the corresponding cubic twist and let $\lambda_n \colon A_n \to \widehat{A}_n$ be the corresponding polarization. Recall that we assume
\begin{enumerate}[$(1)$]
    \item The $\mu_3$-action has isolated fixed points (which is automatic if $A$ is simple), and 
    \item $[-1]^*\L \simeq \L$, and
    \item $\dim_\Q H^0(A, \L) = 2$, so that each $\lambda_n$ is a $(2,2)$-isogeny.
\end{enumerate}

\begin{definition}
{\em 
A subset $\Sigma \subset \Z$ is {\it defined by congruence conditions} if there are open subsets $\Sigma_p \subset \Z_p$ such that $\Sigma = \bigcap_p \Sigma_p$.\footnote{More precisely, $\Sigma = \Z\cap \bigcap_p \Sigma_p\subseteq \widehat{\Z}$.} 
}
\end{definition}

\begin{definition}
{\em 
A subset $\Sigma = \bigcap_p \Sigma_p \subset \Z$ defined by congruence conditions is {\it acceptable} if
\begin{enumerate}[$(1)$]
    \item each $\Sigma_p$ is nonempty and open, and
    \item for all but finitely many primes $p$, the set $\Sigma_p$ contains all $n \in \Z_p$ with $v_p(n) \leq 1$.  
\end{enumerate}
}
\end{definition}

The following theorem gives Theorem \ref{theorem:avg2Sel} as a special case.  
\begin{theorem}\label{theorem:2selcubictwists}
Let $\Sigma \subset \Z$ be acceptable.  Then $\mathrm{avg}_{n \in \Sigma} \#\Sel_{\lambda_n}(A_n) = 3$. 
\end{theorem}

\begin{proof}
We fix some $d \in \Z$ such that $\mathrm{Stab}_G(v_d) \simeq A[\lambda]$, as in Lemma \ref{lem:muaction}. Note that we are free to replace $d$ by $dt^3$, for any nonzero $t \in \Z$.  Recall from Section \ref{subsec:localfields} the notion of a locally soluble $v \in Y(\Q_p)$, relative to $(A, L)$.  By Theorems \ref{thm:globalselmerbijection} and \ref{thm:globalintegrality}, we may choose $d$ so that for every nonzero $n \in \Z$, there is a bijection between the Selmer group $\Sel_{\lambda_n}(A_n)$ and the $G(\Q)$-orbits of locally soluble elements $v \in Y(\Z)$ with $A_3(v) = dn$.

To compute $\mathrm{avg}_{n \in \Sigma} \#\Sel_{\lambda_n}(A_n)$, it is therefore enough to estimate the function $N_{\widetilde \phi}(Y(\Z);X)$, where $\widetilde \varphi \colon Y(\Z) \to [0,1]$ is defined as follows. For $v \in Y(\Z)$, let $\widetilde m(v)$ be the number of $G(\Z)$-orbits in the $G(\Q)$-orbit of $v$. Then we define $\widetilde\varphi(v) = 1/\widetilde m(v)$ if $v$ is everywhere locally soluble and $A_3(v) \in d \Sigma$, otherwise $\widetilde\varphi(v) = 0$. 

As is usual in these types of arguments, it is more convenient to replace $\widetilde \phi$ with the slightly different function  $\phi(v)$ which is defined in the same way except we replace $\widetilde m(v)$ with 
\[m(v) = \sum_{v' \in O(v)}\dfrac{\#\Stab_{G(\Q)}(v)}{\#\Stab_{G(\Z)}(v')},\]
where $O(v)$ is a set of representatives for the $G(\Z)$-orbits in the same $G(\Q)$-orbit as $y$.   Notice that $m(v) = \widetilde m(v)$ whenever $\Stab_{G(\Q)}(v)$ is trivial. This switch is therefore justified by the fact that the number of everywhere locally soluble $G(\Q)$-orbits on $Y(\Z)$ with $|A_3(v)| < X$ and  nontrivial stabilizer $\Stab_{G(\Q)}(v)$ is $O(X^{1/3})$.  Indeed, if $A_3(v) = dn$, then $\Stab_{G(\Q)}(v) = \Stab_{G(\Q)}(v_{dn}) \simeq E_{dn}[2](\Q)$ which is nontrivial if and only if $d^2n^2$ is a cube in $\Q^\times$.  There are $O(X^{1/3})$ such values of $n$ and the number of locally soluble $G(\Q)$-orbits is the same for each one (since the corresponding elliptic curves are isomorphic). Thus the total number of such orbits is $O(X^{1/3})$ and will be negligible when we average over $n \in \Sigma$ with $|n| < X$.

To invoke our general counting result Theorem \ref{thsqfreetc}, we must first show that $\phi$ is an acceptable function defined by congruence conditions. We have $m(v) = \prod_p m_p(v)$ where 
\[m_p(v) = \sum_{v' \in O_p(v)}\dfrac{\#\Stab_{G(\Q_p)}(v)}{\#\Stab_{G(\Z_p)}(v')},\]
where $O_p(v)$ is a set of representatives for the $G(\Z_p)$-orbits in the same $G(\Q_p)$-orbit as $y$. The proof is as in \cite[Prop.\ 3.6]{BS1}, using the fact that $G$ has class number $1$.  From this expression we see that $\phi$ is defined by congruence conditions.  The acceptability of $\phi$ follows from Proposition \ref{prop:globalintegralityimpliesselmer} and the acceptability of the set $\Sigma$.    

Thus, by Theorem \ref{thsqfreetc}, we have 
\begin{equation}\label{eq:average}
N_\phi(Y(\Z);X)
  = 
  X  \cdot \frac 12
 \int_{\scriptstyle{y\in G(\Z) \backslash Y(\R)}\atop\scriptstyle{|A_3(y)|<1}}
dy \;\prod_{p}
  \int_{y\in Y({\Z_{p}})}\phi_{p}(y)\,dy\,+\,O\Big(X^{1 - \Omega(1)}\Big).
  \end{equation}
By Proposition \ref{vjac2}, we have 
\begin{equation}\label{intexp}
    \int_{y \in Y(\Z_p)} \phi_p(y)dy = |\mathcal{J}|_p \cdot \mathrm{Vol}( G(\Z_p))\int_{n \in \Sigma_p} \sum_{\sigma \in \widehat{A}_n(\Q_p)/\lambda_n(A_n(\Q_p))} \dfrac{1}{\#A_n[\lambda_n](\Q_p)} dn,
\end{equation}
using the bijection between locally soluble orbits with $A_3(y) = dn$ and the group $\widehat{A}_n(\Q_p)/\lambda_n(A_n(\Q_p))$, as well as the isomorphism $\Stab_{G(\Q_p)}(v_m) \simeq A_n[\lambda_n](\Q_p)$ of Theorem \ref{theorem:thetaparamgeneral}.

Combining (\ref{eq:average}) and (\ref{intexp}),  we obtain
\[N_\phi(Y(\Z);X) = |\mathcal{J}|\mathrm{Vol}(G(\Z)\backslash G(\R))\prod_p\Bigg( |\mathcal{J}|_p\mathrm{Vol}(G(\Z_p))\int_{n \in \Sigma_p} c_p(\lambda_n) dn\Bigg)X + O\Big(X^{1 - \Omega(1)}\Big)\]
where
\[c_p(\lambda_n) := \dfrac{\#\widehat A_n(\Q_p)/\lambda_n(A_n(\Q_p))}{\#A_n[\lambda_n](\Q_p)}.
\]
 For finite $p \neq 2$ we have $c_p(\lambda_n) =  c_p(\widehat A_n)/c_p(A_n) = 1$ by \cite[Prop.\ 3.1]{shnidmanRM} and \cite[Prop.\ 4.3]{Lorenzini}. In fact, since $\widehat \lambda_n = \lambda_n$, we have the more general formula 
 \begin{equation}\label{eq:selmerratio}
 c_p(\lambda_n) = |2|_p^{-1}
 \end{equation}
 for all $p\leq \infty$ \cite[Lem.\ 7.1]{laga}.  

It follows that
\begin{align*}
N_\phi(Y(\Z);X) &= |\mathcal{J}|\mathrm{Vol}(G(\Z)\backslash G(\R)) X\prod_p \left(|\mathcal{J}/2|_p \cdot \mathrm{Vol}(G(\Z_p))\cdot \mathrm{Vol}(\Sigma_p) \right)X + O\Big(X^{1 - \Omega(1)}\Big)\\
&= 2 \cdot \mathrm{Vol}(\Sigma)\cdot X \cdot \mathrm{Vol}(G(\Z)\backslash G(\R))\prod_p \mathrm{Vol}(G(\Z_p)) + O\Big(X^{1 - \Omega(1)}\Big) \\
&= 4\cdot \mathrm{Vol}(\Sigma) \cdot X + O\Big(X^{1 - \Omega(1)}\Big), 
\end{align*}
where $\mathrm{Vol}(\Sigma)$ is the natural density of $\Sigma \subset \Z$, and we have used that the Tamagawa number of $G$ is~$2$.  
We conclude that
\[    \mathrm{avg}_{n \in \Sigma} \#\Sel_{\lambda_n}(A_n) = 1 + \lim_{X \to \infty} \dfrac{N_\phi(Y(\Z); X)}{\mathrm{Vol}(\Sigma \cap [-X,X])}
    = 1 + 2 = 3,\]
    as desired.
\end{proof}

We will also require the following variant of Theorem \ref{theorem:2selcubictwists}, where we impose additional local conditions on the Selmer elements. 

\begin{theorem}\label{thm:equidistribution}
Fix a prime $p$ and suppose $\Sigma \subset \Z$ is an acceptable subset such that $\widehat A_n(\Q_p)/\lambda_n A_n(\Q_p)$ has constant size $2^k$ for all $n \in \Sigma$.  Let $\Sel_s(A_n) \subset \Sel_{\lambda_n}(A_n)$ be the subgroup of Selmer elements which are locally trivial at $p$. Then  $\mathrm{avg}_{n \in \Sigma} \#\Sel_s(A_n) = 1 + 2^{1-k}$.
\end{theorem}

\begin{proof}
The proof is the same as in Theorem \ref{theorem:2selcubictwists}, except we tweak the local weight function $\phi_p$ so that $\phi_p(y) = 0$ unless $y$ is in the reducible $G(\Q_p)$-orbit with $A_3$-invariant $A_3(y)$, which corresponds to the identity element of $\widehat A_n(\Q_p)/\lambda_n A_n(\Q_p)$ under the bijection of Theorem \ref{theorem:thetaparamgeneral}.  Since $\widehat A_n(\Q_p)/\lambda_n A_n(\Q_p)$ has size $2^k$ for all $n \in \Sigma$, this has the effect of multiplying the Euler factor at $p$ by $2^{-k}$, and leaving all other Euler factors the same. So the proof gives $\mathrm{avg}_{n \in \Sigma} \#\Sel_s(A_n) = 1 + 2 \cdot 2^{-k} = 1 + 2^{1-k}$.
\end{proof}

Using similar arguments, one can prove a more general equidistribution theorem as in \cite[Thm.~9]{BhargavaSkinner}, but for the applications in this paper, Theorems \ref{theorem:2selcubictwists} and \ref{thm:equidistribution} will suffice.

\section{The root number in any cubic twist family is equidistributed}\label{sec:root numbers}

Let $d$ and $n$ be nonzero integers, and let $E_{d,n} \colon y^2 = x^3 + dn^2$.  Write $w_{d,n}\in \{\pm 1\}$ for the root number of $E_{d,n}$; thus the functional equation for the completed $L$-function of $E_{d,n}$ reads \[L(E_{d,n},s) = w_{d,n} \cdot  L(E_{d,n}, 2-s).\]  
The purpose of this section is to prove Theorems \ref{thm:rootnumberequidistribution} and \ref{thm:rootnumbercountableunion}.

To prove these theorems we make use of known explicit formulas for the roots numbers $w_{d,n}$.  For each $0\neq d\in \Z$, let $f_d: \Z^+\to \{\pm 1\}$ be the multiplicative function such that $f_d(p^k) = 1$ for all primes $p\mid 6d$ and $k\in \N$, such that $f_d(p^k) = f_d(p^{k-3})$ for all $p$ and $k\geq 3$, and such that $f(p^2) = f(p) = \chi_{-3}(p)$, for all primes $p \nmid 6d$, where $\chi_{-3}(p) = \left(\frac{-3}{p}\right)$.

The following is a corollary of  V\'{a}rilly-Alvarado's \cite[Prop.\ 4.4]{VA11}, drawing on formulas of Rohrlich.
\begin{proposition}\label{prop:rootnumberformula}
Let $0\neq d\in \Z$. Then there is a function $g_d: (\Z/9)^\times\times (\Z/3)^{\omega(6d)}\to \{\pm 1\}$ such that $$w_{d,n} = g_d\left(\left(\frac{n}{3^{v_3(n)}}\right)^2\bmod{9}, (v_p(n)\bmod{3})_{p\mid 6d})\right)\cdot f_d(n)$$ for all $n\in \Z^+$.
\end{proposition}

Proposition \ref{prop:rootnumberformula} shows that if we ignore  a factor coming from primes dividing $6d$, the root number of $E_{d,n}$ agrees with the evaluation of the multiplicative function $f_d$ at $n$. As a corollary, we deduce the following result, which is a more precise version of Theorem \ref{thm:rootnumbercountableunion}. 

\begin{corollary}\label{cor:rootnumberlevelsets}
Let $\Sigma\subseteq \Z^+$, and let $\gamma\in \{\pm 1\}$. For each $(s, \nu, a)\in \Z^+\times \Z^+\times (\Z/9)^\times$ with $s\mid (6d)^\infty$ and $\nu$ squarefull and coprime to $6d$, there is an $\eps_{(s, \nu, a)}\in \{\pm 1\}$ such that
\begin{align*}
\left\{n\in \Sigma \,\vert\, w_{d,n} = \gamma\right\} = \disjcup_{\substack{(s, \nu, a)\in \Z^+\times \Z^+\times (\Z/9)^\times :\\ s\mid (6d)^\infty, \\(\nu, 6d) = 1, \\\nu\text{ squarefull}}} \left\{s\cdot \nu\cdot t \,\left\vert\, \substack{s\cdot \nu\cdot t\in \Sigma, \\t\text{ squarefree, } \\t^2\equiv a\pmods{9}, \\(t, 6d\cdot \nu) = 1, \\t\equiv \gamma\cdot \eps_{(s, \nu, a)}\pmods{3}}\right.\right\}
\end{align*}
and
\begin{align*}
\left\{n\in \Sigma \,\vert\, w_{d,n^2} = \gamma\right\} = \disjcup_{\substack{(s, \nu, a)\in \Z^+\times \Z^+\times (\Z/9)^\times :\\ s\mid (6d)^\infty, \\(\nu, 6d) = 1, \\\nu\text{ squarefull}}} \left\{s\cdot \nu\cdot t \,\left\vert\, \substack{s\cdot \nu\cdot t\in \Sigma, \\t\text{ squarefree, } \\t^2\equiv a\pmods{9}, \\(t, 6d\cdot \nu) = 1, \\t\equiv \gamma\cdot \eps_{(s^2, \nu^2, a^2)}\pmods{3}}\right.\right\}.
\end{align*}
\end{corollary}

\begin{proof}
For each such $(s,\nu,a)$, let 
\[\eps_{(s,\nu,a)} := g_d\left(\left(\frac{s}{3^{v_3(s)}}\right)^2\cdot \nu^2\cdot a\bmod{9}, (v_p(s)\bmod{3})_{p\mid 6d}\right)\cdot f_d(\nu).\]
Now, each $n\in \Sigma$ can be written uniquely as $n = s\nu t$ where $\nu t$ is prime to $6d$, and $t$ (resp.\ $\nu$) is  the ``squarefree part'' (resp.\ ``squarefull part'') of $\nu t$. In particular $(t, 6d \nu) = 1$. Setting $a := t^2\pmods{9}$,  Proposition \ref{prop:rootnumberformula} gives
\begin{align*}
w_{d,n} = g_d\left(\left(\frac{s}{3^{v_3(s)}}\right)^2 \nu^2 a\bmod{9}, (v_p(s)\bmod{3})_{p\mid 6d}\right) f_d(\nu) f_d(t)
= \eps_{(s,\nu,a)}f_d(t)
\end{align*}\noindent
and
\begin{align*}
w_{d,n^2} = g_d\left(\left(\frac{s}{3^{v_3(s)}}\right)^4 \nu^4 a^2\bmod{9}, (2v_p(s)\bmod{3})_{p\mid 6d}\right) f_d(\nu^2) f_d(t^2)
= \eps_{(s^2,\nu^2,a^2)} f_d(t),
\end{align*}\noindent
where we have used that $f_d(t^2) = f_d(t)$ since $(t, 6d) = 1$. This completes the proof, since $f_d(t) = \chi_{-3}(t)\equiv t\pmods{3}$ for $t$ squarefree and prime to $6d$.
\end{proof}

We may now deduce the following special case of Theorem \ref{thm:rootnumberequidistribution} (with much better error term).

\begin{remark}
{\em
For the remainder of this section, we restrict without loss of generality to $n\in \Z^+$.
}
\end{remark}

\begin{theorem}\label{equidistribution of root numbers away from three}
Let $m\in \Z^+$, and  let $r\in \Z/m$. Then there are constants $c_{d,r}, \widetilde{c}_{d,r}\in \R$ such that
$$\sum_{n\leq X : n\equiv r\pmods{m}} w_{d,n} = c_{d,r}\cdot X + O(d^{O(1)} m^{O(1)} X^{1 - \Omega(1)})$$ and $$\sum_{n\leq X : n\equiv r\pmods{m}} w_{d,n^2} = \widetilde{c}_{d,r}\cdot X + O(d^{O(1)} m^{O(1)} X^{1 - \Omega(1)}).$$ Moreover, if $3 \nmid m$, then $c_{d,r} = \widetilde{c}_{d,r} = 0$.
In particular, for $d$ fixed and $n\to \infty$ over all of $\Z^+$ $($or any arithmetic progression with common difference not divisible by $3)$, the root numbers of $E_{d,n}$ and $E_{d,n^2}$ uniformly distribute in $\{\pm 1\}$.
\end{theorem}\noindent
The constants $c_{d,r}$ and $\widetilde{c}_{d,r}$ are easy enough to determine but we will not specify it here.

\begin{proof}
For notational convenience we will only treat the case of $3\nmid m$---from the argument it will be clear how to proceed when $3\mid m$.\footnote{The point is that when $3\mid m$ then exactly one of the conditions $t\equiv \pm \eps_{(s,\nu,a)}\pmods{3}$ may contradict the congruence $s\nu t\equiv r\pmods{3}$, producing a term of order $X$ because the sum over $t$ corresponding to $\gamma = +1$ no longer cancels the sum corresponding to $\gamma = -1$ to leading order.}

Of course $$\sum_{n\leq X : n\equiv r\pmods{m}} w_{d,n} = \sum_{\gamma\in \{\pm 1\}} \gamma \sum_{n\leq X : n\equiv r\pmods{m}, w_{d,n} = \gamma} 1.$$
By Corollary \ref{cor:rootnumberlevelsets}, it follows that
\begin{align*}
&\sum_{\gamma\in \{\pm 1\}} \gamma \sum_{n\leq X : n\equiv r\pmods{m}, w_{d,n} = \gamma} 1
\\[.125in]&= \sum_{a\in (\Z/9)^\times} \sum_{s\leq X : s\mid (6d)^\infty} \sum_{\substack{\nu\leq \frac{X}{s} : \\(\nu, 6d) = 1, \\(6d\cdot \nu, r, \frac{m}{(m,s\nu)}) = 1, \\\nu\text{ squarefull}}} \sum_{\gamma\in \{\pm 1\}} \gamma \sum_{\substack{t\leq \frac{X}{s\cdot \nu} : \\(t, 6d\cdot \nu) = 1, \\t^2\equiv a\pmods{9}, \\t\equiv \gamma\cdot \eps_{(s,\nu,a)}\pmods{3}, \\s\nu t\equiv r\pmods{m}}} \mu^2(t)
\\[.085in]&= \sum_{a\in (\Z/9)^\times} \sum_{s\leq X^\delta : s\mid (6d)^\infty} \sum_{\substack{\nu\leq X^\delta : \\(\nu, 6d) = 1, \\(6d\cdot \nu, r, \frac{m}{(m,s\nu)}) = 1, \\\nu\text{ squarefull}}} \sum_{\gamma\in \{\pm 1\}} \gamma \sum_{\substack{t\leq \frac{X}{s\cdot \nu} : \\(t, 6d\cdot \nu) = 1, \\t^2\equiv a\pmods{9}, \\t\equiv \gamma\cdot \eps_{(s,\nu,a)}\pmods{3}, \\s\nu t\equiv r\pmods{m}}} \mu^2(t) + O(X^{1 - \Omega(1)})
\\[.085in]&= \sum_{a\in (\Z/9)^\times} \sum_{s\leq X^\delta : s\mid (6d)^\infty} \sum_{\substack{\nu\leq X^\delta : \\(\nu, 6d) = 1, \\(6d\cdot \nu, r, \frac{m}{(m,s\nu)}) = 1, \\\nu\text{ squarefull}}} \sum_{\gamma\in \{\pm 1\}} \gamma \sum_{\substack{t\leq \frac{X}{s\cdot \nu} : \\(t, 6d\cdot \nu) = 1, \\t^2\equiv a\pmods{9}, \\t\equiv \gamma\cdot \eps_{(s,\nu,a)}\pmods{3}, \\t\equiv r\cdot \left(\frac{s\nu}{(m,s\nu)}\right)^{-1}\pmods{\frac{m}{(m,s\nu)}}}} \mu^2(t) + O(X^{1 - \Omega(1)}).
\end{align*}\noindent
Note that the conditions $(t, 6d\cdot \nu) = 1, t^2\equiv a\pmods{9}, t\equiv \gamma\cdot \eps_{(s,\nu,a)}\pmods{3}, t\equiv r\cdot \left(\frac{s\nu}{(m,s\nu)}\right)^{-1}\pmods{\frac{m}{(m,s\nu)}}$ can be written as at most $\ll d X^\delta$ congruences modulo $[9, 6d\cdot \nu, m]\ll dm X^\delta$.

Now, for $x\in \Z/y$, if $(x,y)$ is not squarefree, then there are no squarefree $n\equiv x\pmods{y}$. Otherwise,
\begin{align*}
\sum_{n\leq X : n\equiv x\pmods{y}} \mu^2(n) &= \sum_{n\leq X : n\equiv x\pmods{y}} \sum_{d^2\mid n} \mu(d)
\\&= \sum_{d\leq \sqrt{X} : (y,d^2)\mid (x,y)} \mu(d) \sum_{e\leq \frac{X}{d^2} : \frac{d^2}{(y,d^2)}\cdot e\equiv x\pmods{\frac{y}{(y,d^2)}}} 1
\\&= \sum_{f\mid (x,y) : (f, \frac{y}{f}) = 1} \mu(f) \sum_{g\leq \sqrt{\frac{X}{f}} : (g,y) = 1} \mu(g) \sum_{e\leq \frac{X}{f^2 g^2} : e\equiv \frac{x}{f}\cdot f^{-1} g^{-2}\pmods{\frac{y}{f}}} 1
\\&= \sum_{f\mid (x,y) : (f, \frac{y}{f}) = 1} \mu(f) \sum_{g\leq \sqrt{\frac{X}{f}} : (g,y) = 1} \mu(g) \left(\frac{X}{f g^2 y} + O(1)\right)
\\&= \frac{X}{y} \sum_{f\mid (x,y) : (f, \frac{y}{f}) = 1} \frac{\mu(f)}{f} \sum_{g\leq \sqrt{\frac{X}{f}} : (g,y) = 1} \frac{\mu(g)}{g^2} + O\left(O(1)^{\frac{\sqrt{\log{y}}}{\log\log{y}}}\sqrt{X} + O(1)^{\frac{\log{y}}{\log\log{y}}}\right)
\\&= \frac{X}{y} \sum_{f\mid (x,y) : (f, \frac{y}{f}) = 1} \frac{\mu(f)}{f} \sum_{g\geq 1 : (g,y) = 1} \frac{\mu(g)}{g^2} + O\left(O(1)^{\frac{\sqrt{\log{y}}}{\log\log{y}}} \sqrt{X} + O(1)^{\frac{\log{y}}{\log\log{y}}}\right)
\\&= \frac{6}{\pi^2}\cdot \frac{X}{y}\cdot \prod_{p\mid (x,y) : p^2\nmid y} \left(1 - \frac{1}{p}\right)\cdot \prod_{p\mid y} \left(1 - \frac{1}{p^2}\right)^{-1} + O\left(O(1)^{\frac{\sqrt{\log{y}}}{\log\log{y}}} \sqrt{X} + O(1)^{\frac{\log{y}}{\log\log{y}}}\right)
\end{align*}\noindent
(of course one can be more precise). Consequently (since the leading terms match---this is where we use that $3\nmid m$),
\begin{align*}
&\sum_{\gamma\in \{\pm 1\}} \gamma \!\!\!\!\!\!\!\!\!\!
\!\!\!\!\!\!\!\!\sum_{\substack{t\leq \frac{X}{s\nu} : \\(t, 6d\nu) = 1, \\t^2\equiv a\pmods{9}, \\t\equiv \gamma \eps_{(s,\nu,a)}\pmods{3}, \\t\equiv r \bigl(\frac{s\nu}{(m,s\nu)}\bigr)^{-1}\!\pmods{\frac{m}{(m,s\nu)}}}} \!\!\!\!\!\!\!\!
\!\!\!\!\!\!\!\!\!\mu^2(t)
\;\:=\!\!\ \sum_{\substack{t\leq \frac{X}{s \nu} : \\(t, 6d \nu) = 1, \\t^2\equiv a\pmods{9}, \\t\equiv \eps_{(s,\nu,a)}\pmods{3}, \\t\equiv r \bigl(\frac{s\nu}{(m,s\nu)}\bigr)^{-1}\!\pmods{\frac{m}{(m,s\nu)}}}} \!\!\!\!\!\!\!\!
\!\!\!\!\!\!\!\!\!\mu^2(t)\;- \!\!\sum_{\substack{t\leq \frac{X}{s \nu} : \\(t, 6d\nu) = 1, \\t^2\equiv a\pmods{9}, \\t\equiv -\eps_{(s,\nu,a)}\pmods{3}, \\t\equiv r \bigl(\frac{s\nu}{(m,s\nu)}\bigr)^{-1}\!\pmods{\frac{m}{(m,s\nu)}}}} \!\!\!\!\!\!\!\!
\!\!\!\!\!\!\!\!
\!\mu^2(t)
\ll d^{O(1)} m^{O(1)} X^{1/2 + O(\delta)}.
\end{align*}\noindent
Summing this over $s,\nu\leq X^\delta$ and $a\in (\Z/9)^\times$, we conclude that $$\sum_{n\leq X : n\equiv r\pmods{m}} w_{d,n}\ll d^{O(1)} m^{O(1)} X^{1/2 + O(\delta)} + X^{1 - \Omega(1)},$$ as desired. Finally, after replacing $\eps_{(s,\nu,a)}$ by $\eps_{(s^2, \nu^2, a^2)}$ above precisely the same argument proves $$\sum_{n\leq X : n\equiv r\pmods{m}} w_{d,n^2}\ll d^{O(1)} m^{O(1)} X^{1/2 + O(\delta)} + X^{1 - \Omega(1)},$$ and we are done.
\end{proof}

We may now prove Theorem \ref{thm:rootnumberequidistribution}.

\begin{proof}[Proof of Theorem $\ref{thm:rootnumberequidistribution}$.]
Let $\eps > 0$. Let $m\in \Z^+$ and $A\subseteq \Z/m$ be such that $3\nmid m$ and the symmetric difference $\{n : n\pmods{m}\in A\} \Delta \Sigma$ has density $\leq \eps$ (we may use e.g.\ $m := \prod_{\substack{p\leq T, \\p\neq 3}} p^T$ with $T\gg_{\Sigma, \eps} 1$). Thus $$\sum_{n\in \Sigma : n\leq X} w_{d,n} = \sum_{r\in A} \sum_{n\leq X : n\equiv r\pmods{m}} w_{d,n} + O_S(\eps X)$$ and similarly for $\sum_{n\in \Sigma : n\leq X} w_{d,n^2}$. We conclude by applying Theorem \ref{equidistribution of root numbers away from three} and taking $\eps\to 0$ sufficiently slowly with $X$.
\end{proof}

\section{Cubic twists having ranks 0 and 1}\label{sec:ranks}

Fix a nonzero $d\in\Z$, and let $E_{d,n} \colon y^2 = x^3 + dn^2$ be the corresponding cubic twist family of elliptic curves, with $n\in\Z$ varying.  Theorem \ref{theorem:avg2Sel} does not quite imply that a positive proportion of twists $E_{d,n}$ have $2$-Selmer rank 0 (and hence Mordell--Weil rank 0); for example, it is consistent with the possibility that asymptotically half of the curves $E_{d,n}$ satisfy $\#\Sel_2(E_{d,n}) = 2$ and half satisfy $\#\Sel_2(E_{d,n}) = 4$.  This hypothetical distribution is also consistent with the fact that the parity of $\dim_{\F_2} \Sel_2(E_{d,n})$ is equidistributed in these families.

In this section, we apply the results on root numbers from the previous section and the $p$-parity theorem to 
prove the existence of twists having ranks $0$ and $1$, respectively.

\begin{theorem}\label{avgwithrootnumberfixed}
Fix a nonzero integer $d$ and a sign $w\in \{\pm 1\}$. Then the average size of the $2$-Selmer group of those elliptic curves in $E_{d,n}$ $($resp.\ $E_{d,n^2})$ having root number $w$ is $3$.
\end{theorem}

\begin{proof}
This follows from Theorem \ref{thm:rootnumbercountableunion}, or rather, its more precise version Corollary \ref{cor:rootnumberlevelsets}, which expresses the set of elliptic curves of root number $+1$ (resp.\ $-1$) as the union of acceptable families. The average size of $\Sel_2(E_{d,n})$ is $3$ on each such acceptable family.  Using the uniformity estimate Proposition \ref{uniprop}, one shows as in \cite[\S 6.4]{j=0} that the average is still $3$ when we average over the union of all of these families as well.
\end{proof}

Theorems~\ref{thm:rootnumberequidistribution} and  \ref{avgwithrootnumberfixed} together give Theorem~\ref{rootnumbers}.
We now prove the existence of many curves in any cubic twist family $E_{d,n}$ having $2$-Selmer rank $0$ and $2$-Selmer rank $1$.

\begin{theorem}\label{rankzerothm}
Fix an integer $d \neq 0$ and let $E_{d,n} \colon y^2 = x^3 + dn^2$ be the corresponding family of cubic twists. 
The proportion of $n$ such that $\Sel_2(E_{d,n}) = 0$ $($resp.\ $\Sel_2(E_{d,n^2}) = 0)$ is at least $1/6$, and the proportion of $n$ such that $\dim_{\F_2} \Sel_2(E_{d,n}) = 1$ $($resp.\ $\dim_{\F_2}\Sel_2(E_{d,n^2}) = 1)$ is at least $5/12$. 
\end{theorem}

\begin{proof}
For $w \in \{\pm 1\}$, let $\Sigma(w)$ be the set of integers $n$ such that $E_{d,n}$ (resp.\ $E_{d,n^2}$) has root number $w$.  By the $p$-parity Theorem \cite{DDparity}, the parity of the $\F_2$-rank of $\Sel_2(E_{d,n})$ (resp.\ $\Sel_2(E_{d,n^2})$) is constant on $\Sigma(w)$ and is even if and only if $w = 1$.  If the parity is even, then at least $\frac13$ of $m \in \Sigma(w)$ have $\F_2$-rank 0, as otherwise the average size of $\Sel_2(E_{d,n})$ (resp.\ $\Sel_2(E_{d,n^2})$) would be larger than $\frac{2}{3}\cdot 4 + \frac{1}{3}\cdot 1 = 3$.  Similarly, if the parity is odd, then at least $\frac56$ of $m \in \Sigma(w)$ have $\F_2$-rank equal to $1$, as otherwise the average size of $\Sel_2(E_{d,n})$ (resp.\ $\Sel_2(E_{d,n^2})$) would be larger than $\frac{1}{6}\cdot 8 + \frac{5}{6} \cdot 2 = 3$. 
\end{proof}

For any elliptic curve $E/\Q$, if $\Sel_2(E) = 0$, then $\rk \, E(\Q) = 0$. This follows from the usual exact sequence
\[0 \to E(\Q)/2E(\Q) \to \Sel_2(E) \to \Sha(E)[2] \to 0.\]
It is conjectured that $\Sha(E)[2]$ has even $\F_2$-dimension; this is a consequence of the conjectural finiteness of $\Sha(E)$ and properties of the Cassels-Tate pairing $\Sha(E) \times \Sha(E) \to \Q/\Z$. If this is the case, and if $E(\Q)[2]= 0$, we see that $\dim_{\F_2}\Sel_2(E) = 1$ implies $\rk\, E(\Q) = 1$. Thus, the following is a corollary of Theorem \ref{rankzerothm}.

\begin{corollary}\label{cor:rank proportions}
At least $1/6$ of the elliptic curves $E_{d,n}$ have algebraic rank $0$.  If $\Sha(E_{d,n})$ is finite for $100\%$ of $n$, then at least $5/12$ of the elliptic curves $E_{d,n}$ have algebraic rank $1$.  
\end{corollary}

We can prove an unconditional but weaker version of the second assertion of Corollary~\ref{cor:rank proportions} in certain circumstances, using the $p$-converse theorem of Burungale and Skinner, which is Corollary~\ref{cor:pcm} of the Appendix. This result requires elliptic curves with good reduction at $2$.

\begin{remark}{\em 
For a fixed $d$, there may not exist any cubic twists $E_{d,n}$ with good reduction at $2$.  By Tate's algorithm, if $dn^2$ is sixth-power-free, then $E_{d,n} \colon y^2 = x^3 + dn^2$ has good reduction at $2$ if and only if there exists an odd integer $D$ such that $dn^2 \equiv 16D^2 \pmods{64}$.  It follows that there exist $n$ such that $E_{d,n}$ has good reduction at $2$ if and only if the $2$-adic valuation $v_2(d)$ of $d$ is even and $d = 2^{v_2(d)}D$ with $D \equiv 1 \pmods 4$. In particular, if $d=-432$, then there are many cubic twists with good reduction at $2$.\footnote{Indeed, the model $E_{-432,n} \colon x^3 + y^3 = n$ visibly has good reduction at $2$ when $n$ is an odd integer.}
}
\end{remark} 

\begin{theorem}\label{rankonethm}
Fix a nonzero integer $d$.
Among the elliptic curves in the cubic twist family $E_{d,n}$ $($resp.\ $E_{d,n^2})$ that have good reduction at $2$ and root number $-1$, a proportion of at least $1/3$ have rank $1$.
\end{theorem}

\begin{proof}
By the remark, we may assume that $\Q(\sqrt{d})$ is unramified at $2$, otherwise there are no twists of good reduction.  Let $\alpha_i(X)$ (resp.\ $\beta_i(X)$) denote the proportion of curves in the family $E_{d,n}$ with $|n|\leq X$ having good reduction at 2 and root number $-1$ such that $\#\Sel_2(E_{d,n})= 2^i$ and such that $\Sel_2(E_{d,n})$ maps trivially (resp.\ nontrivially) to $E(\Q_2)/2E(\Q_2)$.  Then we have already seen that $\alpha_1(X)+\beta_1(X)\geq 5/6$.  Since the average size of the $2$-Selmer group in this family is $3$, we have
\begin{equation}\label{ab1}
\sum_i 2^i \alpha_i(X)
+\sum_i 2^i \beta_i(X) = 3 + o_{X\to \infty}(1).
\end{equation}

Now the elliptic curve $E_{d,n}$ has complex multiplication by the field $K = \Q(\sqrt{-3})$, and the prime~$2$ is inert in $\O_K$. It follows that $E_{d,n}$ has supersingular reduction at $2$ and hence $E_{d,n}[2](\Q_2) =0$.\footnote{More directly: the polynomial $x^3 + dn^2$ has no roots over $\Q_2$ if $dn^2 \equiv 16D^2 \pmods{64}$ with $D$ odd.} Hence 
\[\#E_{d,n}(\Q_2)/2E_{d,n}(\Q_2) = 2\#E_{d,n}[2](\Q_2) = 2,\]
by (\ref{eq:selmerratio}). 
If $E_{d,1}$ has good reduction at $2$, and $2\mid n$ but $v_2(n) \not\equiv 0\pmods 3$, then  $E_{d,n}$ has bad reduction at $2$ since $\Q(\sqrt[3]{n})$ is ramified at $2$.  Thus, as $E_{d,n}$ varies within the subfamily of good reduction curves, all the curves are isomorphic over $\Q_2$ (since $\Z_2^\times= \Z_2^{\times3}$). 
By Theorem \ref{thm:equidistribution}, half of all nontrivial $2$-Selmer elements in this family remain nontrivial over~$\Q_2$.
Thus
\begin{equation}\label{ab2}
\frac12 \sum_i 2^i \beta_i(X)\leq 1 + o_{X\to \infty}(1).
\end{equation}
Subtracting twice (\ref{ab2}) from (\ref{ab1}), we conclude that
\begin{equation}\label{ab3}
\sum_i 2^i \alpha_i(X)\leq 1 + o_{X\to \infty}(1).
\end{equation}
In particular, $2\alpha_1(X)\leq 1 + o_{X\to \infty}(1)$. Therefore, $\alpha_1(X)\leq 1/2 + o_{X\to \infty}(1)$ and so $\beta_1(X)\geq 1/3 + o_{X\to \infty}(1)$. 
The elliptic curves whose density is given by $\beta_1(X)$ all have algebraic rank $1$ by Burungale and Skinner's Corollary~\ref{cor:pcm} in the Appendix.
\end{proof}

\begin{proof}[Proof of Theorems~$\ref{main}$--$\ref{main2}$]
When $d = -432 = -2^4\cdot 3^3$, the curve $E_{d,n}$ has good reduction at $2$ whenever $n$
has $2$-adic valuation that is a multiple of $3$; this set has density $4/7$ (among all integers, and also among all cubefree integers, and similarly for squares).
Imposing the root number $-1$ condition is again a density $\frac12$ condition by Theorem \ref{thm:rootnumberequidistribution}. Thus, Theorem \ref{rankonethm} guarantees that at least $\frac13 \cdot \frac47 \cdot \frac12 = \frac{2}{21}$ of cubic twists $E_{16,n}$ (resp.\ $E_{16,n^2}$) have algebraic rank 1. Together with Corollary~\ref{cor:rank proportions}, this gives Theorem \ref{main}.  The proofs of Theorems \ref{mainsquare} and \ref{main2} are similar, using the fact that $E_{d,n}$ has good reduction at $2$ if and only if $dn^2 = 2^{6k+4}D$ with $D\equiv 1\pmods 4$ and~$k \geq 0$. 
\end{proof}

\begin{proof}[Proof of Theorem~$\ref{main3}$]
Since the average size of the $2$-Selmer group across the curves $E_{d,n}$ of positive root number is $3$, and for even integers $r\geq 0$ we have the inequality $$\frac32r + 1\leq 2^{r},$$ 
it follows that 
\[\avg_n\, \rk(E_{d,n})\leq 
\avg_n \frac23(2^{\dim_{\F_2} \Sel_2(E_{d,n})}-1) = \avg_n\frac23(\#\Sel_2(E_{d,n})-1)=\frac23(3-1)=\frac43\]
across the curves $E_{d,n}$ having root number $+1$; here we have used the fact that, by the $p$-parity theorem, curves with root number $+1$ have even $2$-Selmer rank. 

Similarly, since the average size of the $2$-Selmer group across the curves $E_{d,n}$ of negative root number is $3$, and for odd integers $r\geq 0$ we have the inequality $$3r - 1\leq 2^{r},$$ 
it follows that 
\[\avg_n\, \rk(E_{d,n})\leq 
\avg_n\frac13(2^{\dim_{\F_2}\Sel_2(E_{d,n})}+1)\leq \avg_n\frac13(\#\Sel_2(E_{d,n})+1)=\frac13(3+1)=\frac43\]
across the curves $E_{d,n}$ having root number $-1$; here we have used the fact that curves $-1$ have odd $2$-Selmer rank. Thus, across all curves $E_{d,n}$, the average rank is bounded above by $4/3.$

The lower bound follows from Theorem \ref{main2}.

Finally, note that the same argument applies verbatim to the family $E_{d,n^2}$ as well.
\end{proof}


\section{A higher-dimensional example: cubic twists of Prym surfaces}\label{sec:Pryms}
We give examples of cubic twist families of some geometrically simple abelian surfaces. We then combine Theorem \ref{theorem:2selcubictwists} with Pantazis' bigonal construction to prove Theorem \ref{thm: bielliptic Picard curves}.

In the proof of Theorem \ref{theorem:thetaparamgeneral}, 
we observed that if $(A,\L)$ is a degree $2$ polarized abelian variety with $\mu_3$-action, then there is an elliptic curve $E$ and an isomorphism $\eta \colon A[\lambda] \simeq E[2]$ such that the abelian variety $B = (A \times E)/\Gamma_\eta$ is principally polarized.  If $A$ is an abelian surface, then we have the following explicit construction of such $A$ and $B$.

Consider plane quartic curves over $\Q$ with affine model $C \colon y^3 = x^4 + ax^2 + b$, for some $a,b \in \Q$. Suppose that $b(a^2 - 4b) \neq 0$, so that $C$ is smooth.  Such a curve admits a $\mu_3$-action over $\Q$, generated by the order 3 automorphism $(x,y) \mapsto (x,\zeta_3y)$. We consider the cubic twist family
\[C_n \colon ny^3 = x^4 + ax^2 + b.\]
Another model for $C_n$ is $y^3 = x^4 + an^4x^2 + bn^8$.

Let $\pi \colon C \to E$ be the degree two map to the the elliptic curve $E \colon y^3 = x^2 + ax + b$. The evident automorphism of order 3 means $E$ has $j$-invariant 0, and indeed the short Weierstrass model for $E$ is $y^2 = x^3 + 16(a^2 - 4b)$.  Similarly, the curve $C_n$ is a double cover of $E_n \colon y^2 = x^3 + 16m^2(a^2 - 4b)$, and the latter is the cubic twist family of elliptic curves $E_{d,n}$, with $d = 16(a^2 - 4b)$ in our notation.

The involution generating $\Aut(C/E)$ is $\tau(x,y)= (-x,y)$, so $\pi$ is ramified at the three points where $x = 0$ and the unique point $\infty$ at infinity.  Let $J = \mathrm{Jac}(C)$ and let $A := \ker(J \to E)$. 

\begin{lemma}\label{lemma: prym}
The map $\pi^* \colon \Pic^0(E) \to \Pic^0(C) \simeq J$ is injective and $A$ is an abelian variety.
\end{lemma}

\begin{proof}
The kernel of $\pi^*$ is $2$-torsion, and it is nontrivial if and only if $\pi$ is unramified. Since $\pi$ is ramified, the map $\pi^*$ is injective, and it follows by duality that the kernel of $\pi_*\colon J \to E$ is connected and hence an abelian variety. 
\end{proof}

The abelian surface $A$ is an example of a {\it Prym variety}.  It may alternatively be described as the subgroup of degree zero divisor classes in $J$ on which $\tau$ acts as $-1$.  For generic parameters $a, b \in \Q$, the surface $A$ is absolutely simple, so there is no obvious way to reduce the study of the Mordell-Weil group $A(\Q)$ to rational points on elliptic curves.

By Lemma \ref{lemma: prym}, we may view $E \simeq \Pic^0(E)$ inside $J$, and it follows immediately that $E \cap A = E[2]$.  This subgroup plays a role in the geometry of $A$, as we explain. First, let the theta divisor $\theta \in \mathrm{Div}(J)$ be the image of the map $C^{(2)} \to J$ sending an effective divisor $D$ of degree two to $D - 2\infty$. The line bundle $\O_J(\theta)$ determines a principal polarization on $J$, and its restriction to $A$ is an ample line bundle $\L$. The corresponding polarization $\lambda \colon A \to \widehat A$ has degree $4$ and its kernel is precisely $E[2] \subset A$; for more details see \cite{MumfordPrym, lagashnidman}.

The order $3$ automorphism on $J$ preserves $A$ and preserves the theta divisor $\theta$ as well. It follows that $(A,L)$ is a polarized abelian surface with $\mu_3$-action. In particular, the abelian variety $A$ has cubic twists, which are simply the Prym varieties $A_n$ attached to the curves $C_n$.  
Let $\lambda_n$ be the degree 4 polarization $A_n \to \widehat{A}_n$, the $n$-th cubic  twist of $\lambda = \lambda_L \colon A \to \widehat{A}$.

\begin{corollary}\label{cor:avglambdaselsize}
 We have $\avg_n \#\Sel_\lambda(A_n) = 3$ and $\avg_n\dim_{\F_2} \Sel_\lambda(A_n) \leq 1.5$.
\end{corollary}

\begin{proof}
This follows from Theorem \ref{theorem:2selcubictwists}, once we observe that $\dim_\Q H^0(A, \L_A) = \sqrt{\deg(\lambda)} = 2$.
\end{proof}

Next, we leverage our understanding of the Selmer groups $\Sel_{\lambda_n}(A_n)$ to deduce information about the $\Sel_2(A_n)$, and hence the ranks of $A_n(\Q)$.  Let $\widetilde\lambda_n \colon \widehat A_n \to A_n$ be the isogeny (over $\Q$) such that $\widetilde \lambda_n \circ \lambda_n = [2]$. Beware that $\widetilde\lambda_n$ is not the dual of $\lambda_n$, as $\lambda_n$ is self-dual whereas $A_m$ is generally not. To study $\widetilde \lambda_n$ we use a beautiful special case of Pantazis' bigonal construction:

\begin{proposition}\label{prop:bigonal}
Recall $d = 16(a^2 - 4b)$.  The surface $\widehat A$ is the Prym attached to the genus three curve $\widehat C \colon y^3 = x^4 + 8ax^2 + d$. Moreover, the map $\widetilde \lambda \colon \widehat A \to A$ is the natural polarization on $\widehat{A}$.
\end{proposition}

\begin{proof}
This is a special case of \cite[Thm.\ 3.14]{laga}. 
\end{proof}

\begin{proof}[Proof of Theorem $\ref{thm: bielliptic Picard curves}$.]
The first sentence of Theorem \ref{thm: bielliptic Picard curves} is Corollary \ref{cor:avglambdaselsize} above. The average $\F_2$-rank of $\Sel_\lambda(A_n)$ is at most 1.5 by Corollary \ref{cor:avglambdaselsize}. By Proposition \ref{prop:bigonal} and Corollary \ref{cor:avglambdaselsize}, the average $\F_2$-rank of $\Sel_{\widetilde\lambda}(\widehat A_n)$ is also at most 1.5.  Since $\widetilde \lambda \circ \lambda  = [2]$, it follows that the average $\F_2$-dimension of $\Sel_2(A_n)$ is at most $1.5 + 1.5 = 3$, and hence the average rank of $A_n$ is at most $3$.  Since $\rk\, J_n = \rk\, A_n + \rk\, E_n$, and since the average rank of $E_n = E_{d,n}$ is at most $4/3$ by Theorem~\ref{main2}, we conclude that the average rank of $J_n$ is at most $3 + 4/3 = 13/3$. 
\end{proof}

\section[The \hspace{-.0125in}average \hspace{-.0125in}size \hspace{-.0125in}of \hspace{-.0125in}the \hspace{-.0125in}$3$-Selmer \hspace{-.0125in}group \hspace{-.0125in}in \hspace{-.0125in}a \hspace{-.0125in}general \hspace{-.0125in}cubic \hspace{-.0125in}twist \hspace{-.0125in}family \hspace{-.0125in}is \hspace{-.0125in}infinite]{The average size of the $3$-Selmer group in a general cubic twist family is infinite}

\begin{proof}[Proof of Theorem $\ref{theorem:3Sel}$]
Since $d$ is fixed, we write $E_n = E_{d,n}$, and let $E'_n = E_{-3d,3n}$.  There is a 3-isogeny $\phi_n \colon E_n \to E'_n$, whose base change to $\Q(\sqrt{-3})$ becomes multiplication by $\sqrt{-3}$ \cite{j=0}.
The kernel of the natural map $\Sel_{\phi_n}(E_n) \to \Sel_3(E_n)$ is $E'_n[\widehat{\phi}_n](\Q)/\phi(E_n[3](\Q))$, whose size is at most~$3$. Thus it suffices to show that the average size of $\Sel_{\phi_n}(E_n)$ is unbounded as $n \to \infty$.

Combining the Greenberg--Wiles formula \cite[8.7.9]{NSW} and \cite[Prop.\ 3.1]{shnidmanRM}, we have
 \[\#\Sel_{\phi_n}(E_n) \gg_d c(E_n')/c(E_n),\]
 where $c(E) = \prod_p c_p(E)$ is the product of all the Tamagawa numbers of $E$. The ratios
 \[c_p(E_n')/c_p(E_n)\] are uniformly bounded (from above and below), independent of both $m$ and $p$ (and even $d$).  This is a general fact about $\ell$-isogenies of abelian varieties of a given dimension, but it follows easily from Tate's algorithm in this case, especially since $E$ has everywhere potentially good reduction.  Thus, we can safely ignore finitely many primes, and we have    
 \[\#\Sel_{\phi_n}(E_n) \gg_d \prod_{p > 3d}\dfrac{c_p(E_n')}{c_p(E_n)}.\] 
Let $\chi = \left(\frac{d}{\cdot}\right)$ be the quadratic character cutting out the field $\Q(\sqrt{d})$.  For $p > 3d$, we have \cite[Prop.~34]{j=0}: 
\[\dfrac{c_p(E_n')}{c_p(E_n)} = 
\begin{cases}
3^{-\chi(n)} & \mbox{ if } p \equiv 2\pmods3 \mbox{ and } p \mid n,\\
1 & \mbox{ otherwise.}
\end{cases}\]
Now let $\alpha(n)$ (resp.\ $\beta(n)$) be the number of primes $p \equiv 2 \pmods{3}$ dividing $n$ such that $\chi(n) = -1$ (resp.\ $\chi(n) = 1$).   Then 
\begin{equation}\label{eq:selbound}
\#\Sel_{\phi_n}(E_n) \gg_d 3^{\alpha(n) - \beta(n)}.
\end{equation}

To estimate the sum $\sum_{n < X} \#\Sel_{\phi_n}(E_n)$, we use the following result of Selberg--Delange type.

\begin{theorem}[{\cite[Prop.~4]{fms}}]\label{odoni}
Let $f$ be a multiplicative real valued function on the natural numbers. Suppose that there exist constants $u$ and $v$ such that $0 \leq f(p^r) \leq ur^v$ for all primes $p$ and all positive integers $r$. Suppose also that there exist real numbers $\xi > 0$ and $0 < \beta < 1$ such that
\[\sum_{p < X} f(p) = \xi\cdot \dfrac{X}{\log X} + O\left(\dfrac{X}{(\log X)^{1 + \beta}}\right)\]
as $X \to \infty$.   Then there is an explicit constant $C_f$ such that 
\[\sum_{n \leq X} f(n) = C_f\cdot X (\log X)^{\xi - 1} + O_f\left(X(\log X)^{\xi - 1 - \beta}\right)\]
as $X \to \infty$.  
\end{theorem}

To apply Theorem \ref{odoni}, first consider the case that $-3d$ is not a square.  Then half of primes are congruent to $2\pmods3$ and among those, half of them satisfy $\chi(m) = 1$.  Thus   
\begin{align*}
\sum_{p \leq X} 3^{\alpha(p) - \beta(p)} 
&= \left(\frac{1}{2}  + 3\cdot \frac{1}{4} + \frac{1}{3}\cdot \frac{1}{4}\right) X\left(\log X\right)^{-1} + O\Big(X\left(\log X\right)^{-2}\Big)\\
&= \frac{4}{3}\cdot X(\log X)^{-1} + O\Big(X(\log X)^{-2}\Big).    
\end{align*}
It then follows from (\ref{eq:selbound})  and Theorem \ref{odoni} that there is an explicit constant $C_d\in \R^+$ such that 
\[\sum_{n < X} \#\Sel_{\phi_n}(E_n) \gg_d \sum_{n < X} 3^{\alpha(n) - \beta(n)} = (C_d + o(1))\cdot X\left(\log X\right)^{1/3},\]
which proves Theorem \ref{theorem:3Sel} in this case.  If $-3d$ is a square, then a similar argument shows that
\[\sum_{n < X} \#\Sel_{\phi_n}(E_n) \gg_d (1 + o(1))\cdot X\log X,\]
which completes the proof. 
\end{proof}

\begin{remark}{\em 
There is a single cubic twist family that is not covered by Theorem \ref{theorem:3Sel}, namely, $E_{1,n} \colon y^2 = x^3 + n^2$.  In this case, Chan~\cite{chan} has proven that the average size of $\Sel_{\phi_n}(E_{1,n})$ is equal to 1. 
}  
\end{remark}

\begin{remark}{\em 
It seems likely that our lower bound is close to sharp. That is, aside from the two exceptional cubic twist families $E_{1,n}$ and $E_{-3,n}$, we expect that the main term will be on the order of $X(\log X)^{1/3}$.}     
\end{remark}

\subsection*{Acknowledgements}
We thank Asher Auel, Tim Browning, Ashay Burungale, Roger Heath-Brown, Wei Ho, Kevin Hughes, Jef Laga, Jonathan Love, Arul Shankar, and Christopher Skinner for helpful conversations. 
We also thank Ashay Burungale and Christopher Skinner for the beautiful results proven in the Appendix below on $p$-converse theorems for primes $p$ of supersingular reduction for  elliptic curves with complex multiplication.
The first author was supported by the National Science Foundation (grant DMS-$2002109$) and the Society of Fellows.
The second author was supported by a Simons Investigator Grant and NSF grant~DMS-1001828. 
The third author was supported by the Israel Science Foundation (grant No. 2301/20).

\phantom{\cite{shankar-tsimerman}}

\bibliographystyle{abbrv}
\bibliography{references}

\begin{thebibliography}{10}

\bibitem{ABScubics}
L.~{Alp\" oge}, M.~{Bhargava}, and A.~{Shnidman}.
\newblock A positive proportion of cubic fields are not monogenic yet have no
  local obstruction to being so.
\newblock {\em arXiv:2011.01186}, 2020.

\bibitem{leventcircle}
L.~Alp\"oge.
\newblock \href{https://arxiv.org/abs/2110.03947}{Quadrics in arithmetic
  statistics}.
\newblock {\em arXiv:2110.03947}, 2021.

\bibitem{leventthesis}
L.~H.~A. Alp\"{o}ge.
\newblock {\em
  \href{https://dataspace.princeton.edu/handle/88435/dsp01qv33s071x}{Points on
  {C}urves}}.
\newblock ProQuest LLC, Ann Arbor, MI, 2020.
\newblock Thesis (Ph.D.)--Princeton University.

\bibitem{bektemirovmazursteinwatkins}
B.~Bektemirov, B.~Mazur, W.~Stein, and M.~Watkins.
\newblock Average ranks of elliptic curves: tension between data and
  conjecture.
\newblock {\em Bull. Amer. Math. Soc. (N.S.)}, 44(2):233--254, 2007.

\bibitem{Bhargavadensityofquinticfields}
M.~Bhargava.
\newblock The density of discriminants of quintic rings and fields.
\newblock {\em Ann. of Math. (2)}, 172(3):1559--1591, 2010.

\bibitem{j=0}
M.~Bhargava, N.~Elkies, and A.~Shnidman.
\newblock The average size of the 3-isogeny {S}elmer groups of elliptic curves
  {$y^2=x^3+k$}.
\newblock {\em J. Lond. Math. Soc. (2)}, 101(1):299--327, 2020.

\bibitem{BG}
M.~Bhargava and B.~H. Gross.
\newblock Arithmetic invariant theory.
\newblock In {\em Symmetry: representation theory and its applications}, volume
  257 of {\em Progr. Math.}, pages 33--54. Birkh\"{a}user/Springer, New York,
  2014.

\bibitem{BhargavaHo}
M.~Bhargava and W.~Ho.
\newblock Coregular spaces and genus one curves.
\newblock {\em Camb. J. Math.}, 4(1):1--119, 2016.

\bibitem{BhargavaHo2}
M.~Bhargava and W.~Ho.
\newblock On average sizes of selmer groups and ranks in families of elliptic
  curves having marked points.
\newblock {\em arXiv:2207.03309}, 2022.

\bibitem{bkls}
M.~Bhargava, Z.~Klagsbrun, R.~J. Lemke~Oliver, and A.~Shnidman.
\newblock 3-isogeny {S}elmer groups and ranks of abelian varieties in quadratic
  twist families over a number field.
\newblock {\em Duke Math. J.}, 168(15):2951--2989, 2019.

\bibitem{BS1}
M.~Bhargava and A.~Shankar.
\newblock Binary quartic forms having bounded invariants, and the boundedness
  of the average rank of elliptic curves.
\newblock {\em Ann. of Math. (2)}, 181(1):191--242, 2015.

\bibitem{BShn}
M.~Bhargava and A.~Shnidman.
\newblock On the number of cubic orders of bounded discriminant having
  automorphism group {$C_3$}, and related problems.
\newblock {\em Algebra Number Theory}, 8(1):53--88, 2014.

\bibitem{BhargavaSkinner}
M.~Bhargava and C.~Skinner.
\newblock A positive proportion of elliptic curves over {$\Bbb Q$} have rank
  one.
\newblock {\em J. Ramanujan Math. Soc.}, 29(2):221--242, 2014.

\bibitem{browning-heath-brown}
T.~D. Browning and R.~Heath-Brown.
\newblock The geometric sieve for quadrics.
\newblock {\em Forum Math.}, 33(1):147--165, 2021.

\bibitem{CesnaviciusFlat}
K.~{\v{C}}esnavi\v{c}ius.
\newblock
  \href{https://www.imo.universite-paris-saclay.fr/~cesnavicius/selmer-flat.pdf}{Selmer
  groups as flat cohomology groups}.
\newblock {\em J. Ramanujan Math. Soc.}, 31(1):31--61, 2016.

\bibitem{chan}
S.~Chan.
\newblock The $3$-isogeny selmer groups of the elliptic curves $y^2=x^3+n^2$.
\newblock {\em arXiv:2211.06062}, 2022.

\bibitem{1983icmproceedings}
Z.~Ciesielski and C.~a. Olech, editors.
\newblock {\em Proceedings of the international congress of mathematicians.
  {V}ol. 1, 2}. PWN---Polish Scientific Publishers, Warsaw; North-Holland
  Publishing Co., Amsterdam, 1984.
\newblock Held in Warsaw, August 16--24, 1983.

\bibitem{DV18}
S.~Dasgupta and J.~Voight.
\newblock Sylvester's problem and mock {H}eegner points.
\newblock {\em Proc. Amer. Math. Soc.}, 146(8):3257--3273, 2018.

\bibitem{Dickson}
L.~E. Dickson.
\newblock {\em History of the theory of numbers. {V}ol. {II}: {D}iophantine
  analysis}.
\newblock Chelsea Publishing Co., New York, 1966.

\bibitem{DDparity}
T.~Dokchitser and V.~Dokchitser.
\newblock Regulator constants and the parity conjecture.
\newblock {\em Invent. Math.}, 178(1):23--71, 2009.

\bibitem{fms}
S.~Finch, G.~Martin, and P.~Sebah.
\newblock Roots of unity and nullity modulo {$n$}.
\newblock {\em Proc. Amer. Math. Soc.}, 138(8):2729--2743, 2010.

\bibitem{goldfeld}
D.~Goldfeld.
\newblock Conjectures on elliptic curves over quadratic fields.
\newblock In {\em Number theory, {C}arbondale 1979 ({P}roc. {S}outhern
  {I}llinois {C}onf., {S}outhern {I}llinois {U}niv., {C}arbondale, {I}ll.,
  1979)}, volume 751 of {\em Lecture Notes in Math.}, pages 108--118. Springer,
  Berlin, 1979.

\bibitem{Heath-Brown}
D.~R. Heath-Brown.
\newblock The size of {S}elmer groups for the congruent number problem. {II}.
\newblock {\em Invent. Math.}, 118(2):331--370, 1994.
\newblock With an appendix by P. Monsky.

\bibitem{HB}
D.~R. Heath-Brown.
\newblock A new form of the circle method, and its application to quadratic
  forms.
\newblock {\em J. Reine Angew. Math.}, 481:149--206, 1996.

\bibitem{kane}
D.~Kane.
\newblock On the ranks of the 2-{S}elmer groups of twists of a given elliptic
  curve.
\newblock {\em Algebra Number Theory}, 7(5):1253--1279, 2013.

\bibitem{KaneThorne}
D.~M. Kane and J.~A. Thorne.
\newblock On the {$\phi$}-{S}elmer groups of the elliptic curves
  {$y^2=x^3-Dx$}.
\newblock {\em Math. Proc. Cambridge Philos. Soc.}, 163(1):71--93, 2017.

\bibitem{katzsarnak}
N.~M. Katz and P.~Sarnak.
\newblock Zeroes of zeta functions and symmetry.
\newblock {\em Bull. Amer. Math. Soc. (N.S.)}, 36(1):1--26, 1999.

\bibitem{Kim07}
B.~D. Kim.
\newblock The parity conjecture for elliptic curves at supersingular reduction
  primes.
\newblock {\em Compos. Math.}, 143(1):47--72, 2007.

\bibitem{KMR}
Z.~Klagsbrun, B.~Mazur, and K.~Rubin.
\newblock A {M}arkov model for {S}elmer ranks in families of twists.
\newblock {\em Compos. Math.}, 150(7):1077--1106, 2014.

\bibitem{kriz}
D.~Kriz.
\newblock \href{https://arxiv.org/pdf/2002.04767.pdf}{{S}upersingular main
  conjectures, {S}ylvester's conjecture and {G}oldfeld's conjecture}.
\newblock {\em arXiv:2002.04767}, 2020.

\bibitem{lachaud}
G.~Lachaud.
\newblock Une pr\'{e}sentation ad\'{e}lique de la s\'{e}rie singuli\`ere et du
  probl\`eme de {W}aring.
\newblock {\em Enseign. Math. (2)}, 28(1-2):139--169, 1982.

\bibitem{laga}
J.~Laga.
\newblock \href{https://arxiv.org/abs/2101.07658}{{Arithmetic statistics of
  Prym surfaces}}.
\newblock {\em arXiv:2107.06803}, 2021.

\bibitem{lagashnidman}
J.~Laga and A.~Shnidman.
\newblock The geometry and arithmetic of bielliptic {P}icard curves.
\newblock Arxiv preprint, available at
  \url{https://arxiv.org/abs/2308.15297v2}, 2023+.

\bibitem{Lorenzini}
D.~Lorenzini.
\newblock \href{https://doi.org/10.5802/aif.2664}{Torsion and {T}amagawa
  numbers}.
\newblock {\em Ann. Inst. Fourier (Grenoble)}, 61(5):1995--2037 (2012), 2011.

\bibitem{Monsky}
P.~Monsky.
\newblock Generalizing the {B}irch-{S}tephens theorem. {I}. {M}odular curves.
\newblock {\em Math. Z.}, 221(3):415--420, 1996.

\bibitem{MumfordPrym}
D.~Mumford.
\newblock Prym varieties. {I}.
\newblock In {\em Contributions to analysis (a collection of papers dedicated
  to {L}ipman {B}ers)}, pages 325--350. 1974.

\bibitem{MumfordAV}
D.~Mumford.
\newblock {\em Abelian varieties}, volume~5 of {\em Tata Institute of
  Fundamental Research Studies in Mathematics}.
\newblock Published for the Tata Institute of Fundamental Research, Bombay; by
  Hindustan Book Agency, New Delhi, 2008.
\newblock With appendices by C. P. Ramanujam and Yuri Manin, Corrected reprint
  of the second (1974) edition.

\bibitem{Nekovarparity}
J.~Nekov\'{a}\v{r}.
\newblock On the parity of ranks of {S}elmer groups. {IV}.
\newblock {\em Compos. Math.}, 145(6):1351--1359, 2009.
\newblock With an appendix by Jean-Pierre Wintenberger.

\bibitem{NSW}
J.~Neukirch, A.~Schmidt, and K.~Wingberg.
\newblock {\em Cohomology of number fields}, volume 323 of {\em Grundlehren der
  Mathematischen Wissenschaften [Fundamental Principles of Mathematical
  Sciences]}.
\newblock Springer-Verlag, Berlin, second edition, 2008.

\bibitem{poonenrains}
B.~Poonen and E.~Rains.
\newblock Random maximal isotropic subspaces and {S}elmer groups.
\newblock {\em J. Amer. Math. Soc.}, 25(1):245--269, 2012.

\bibitem{RubinSilverberg}
K.~Rubin and A.~Silverberg.
\newblock Ranks of elliptic curves in families of quadratic twists.
\newblock {\em Experiment. Math.}, 9(4):583--590, 2000.

\bibitem{Ruth}
S.~Ruth.
\newblock {A} bound on the average rank of $j$-invariant 0 elliptic curves.
\newblock {\em {P}rinceton {U}niversity Ph.D. thesis}, 2013.

\bibitem{schmidt}
W.~M. Schmidt.
\newblock The density of integer points on homogeneous varieties.
\newblock {\em Acta Math.}, 154(3-4):243--296, 1985.

\bibitem{Selmer}
E.~S. Selmer.
\newblock The {D}iophantine equation {$ax^3+by^3+cz^3=0$}.
\newblock {\em Acta Math.}, 85:203--362 (1 plate), 1951.

\bibitem{shankar-tsimerman}
A.~Shankar and J.~Tsimerman.
\newblock Counting {$S_5$}-fields with a power saving error term.
\newblock {\em Forum Math. Sigma}, 2:Paper No. e13, 8, 2014.

\bibitem{shnidmanRM}
A.~Shnidman.
\newblock \href{https://doi.org/10.1093/imrn/rnz185}{Quadratic twists of
  {A}belian varieties with real multiplication}.
\newblock {\em Int. Math. Res. Not. IMRN}, (5):3267--3298, 2021.

\bibitem{ShnidmanWeissSexticTwist}
A.~{Shnidman} and A.~{Weiss}.
\newblock {Ranks of abelian varieties in cyclotomic twist families}.
\newblock {\em arXiv:2107.06803}, 2021.

\bibitem{Silverbergsurvey}
A.~Silverberg.
\newblock The distribution of ranks in families of quadratic twists of elliptic
  curves.
\newblock In {\em Ranks of elliptic curves and random matrix theory}, volume
  341 of {\em London Math. Soc. Lecture Note Ser.}, pages 171--176. Cambridge
  Univ. Press, Cambridge, 2007.

\bibitem{SilvermanAEC}
J.~H. Silverman.
\newblock {\em The arithmetic of elliptic curves}, volume 106 of {\em Graduate
  Texts in Mathematics}.
\newblock Springer, Dordrecht, second edition, 2009.

\bibitem{Smith-thesis}
A.~Smith.
\newblock {\em \href{https://nrs.harvard.edu/URN-3:HUL.INSTREPOS:37365902}{
  $\ell^\infty$-Selmer Groups in Degree $\ell$ Twist Families}}.
\newblock PhD thesis, Harvard University, Graduate School of Arts \& Sciences,
  2020.

\bibitem{sd}
P.~Swinnerton-Dyer.
\newblock The effect of twisting on the 2-{S}elmer group.
\newblock {\em Math. Proc. Cambridge Philos. Soc.}, 145(3):513--526, 2008.

\bibitem{Sylvester}
J.~J. Sylvester.
\newblock On {C}ertain {T}ernary {C}ubic-{F}orm {E}quations.
\newblock {\em Amer. J. Math.}, 2(4):357--393, 1879.

\bibitem{VA11}
A.~V\'{a}rilly-Alvarado.
\newblock Density of rational points on isotrivial rational elliptic surfaces.
\newblock {\em Algebra Number Theory}, 5(5):659--690, 2011.

\bibitem{watkins}
M.~Watkins.
\newblock Rank distribution in a family of cubic twists.
\newblock In {\em Ranks of elliptic curves and random matrix theory}, volume
  341 of {\em London Math. Soc. Lecture Note Ser.}, pages 237--246. Cambridge
  Univ. Press, Cambridge, 2007.

\bibitem{Yin22}
H.~Yin.
\newblock On the {$8$} case of the {S}ylvester conjecture.
\newblock {\em Trans. Amer. Math. Soc.}, 375(4):2705--2728, 2022.

\bibitem{zagierkramarz}
D.~Zagier and G.~Kramarz.
\newblock Numerical investigations related to the {$L$}-series of certain
  elliptic curves.
\newblock {\em J. Indian Math. Soc. (N.S.)}, 52:51--69 (1988), 1987.

\end{thebibliography}


\begin{thebibliography}{XXXX}
\bibitem{BS} M. Bhargava and C. Skinner, \emph{A positive proportion of elliptic curves over $\BQ$ have rank one}, J. Ramanujan Math. Soc. 29 (2014), no. 2, 221--242.
\bibitem{BCST} A. Burungale, F. Castella, C. Skinner, and Y. Tian, \emph{$p^\infty$-Selmer groups and rational points on CM elliptic curves},
Annales Math, Quebec, Special issue in honor of Bernadette Perrin-Riou (to appear).
\bibitem{BKO} A. Burungale, S. Kobayashi, and K. Ota, \emph{$p$-adic $L$-functions and rational points on CM elliptic curves at inert primes}, preprint.  
\bibitem{BST} A. Burungale, C. Skinner,  and Y. Tian, \emph{Elliptic curves and Beilinson--Kato elements: rank one aspects}, preprint.
\bibitem{BT1} A. Burungale and Y. Tian, \emph{$p$-converse to a theorem of Gross--Zagier, Kolyvagin and Rubin}, Invent. Math. 220 (2020), no. 1, 211--253.
\bibitem{BT2} A. Burungale and Y. Tian, \emph{A rank zero $p$-converse to a theorem of Gross--Zagier, Kolyvagin and Rubin}, preprint. 
 \bibitem{K} K. Kato, \emph{$p$-adic Hodge theory and values of zeta functions of modular forms}, Cohomologies $p$-adiques et applications arithm\'etiques. III. Ast\'erisque No. 295 (2004), ix, 117--290. 
 \bibitem{Ko} S. Kobayashi, \emph{The $p$-adic Gross--Zagier formula for elliptic curves at supersingular primes}, Invent. Math. 191 (2013), no. 3, 527--629.
  \bibitem{PR} B. Perrin-Riou, \emph{Fonctions L $p$-adiques d'une courbe elliptique et points rationnels}, Ann. Inst. Fourier (Grenoble) 43 (1993), no. 4, 945--995.   
\bibitem{Ru2} K. Rubin, \emph{$p$-adic variants of the Birch and Swinnerton-Dyer conjecture for elliptic curves with complex multiplication}, $p$-adic monodromy and the Birch and Swinnerton-Dyer conjecture (Boston, MA, 1991), 71--80, Contemp. Math., 165, Amer. Math. Soc., Providence, RI, 1994.
\bibitem{Se} J.-P. Serre, \emph{Propri\'et\'es galoisiennes des points d'ordre fini des courbes elliptiques}, Invent. Math. 15 (1972), no. 4, 259--331.

\bibitem{Sk} C. Skinner, \emph{A converse to a theorem of Gross, Zagier and Kolyvagin},  Ann. of Math. (2) 191 (2020), no. 2, 329--354.
\bibitem{Yu} Q. Yu, \emph{$p$-Converse to a Theorem of Gross-Zagier, Kolyvagin, and Rubin for Small Primes}, 2021 Caltech Thesis, \texttt{https://resolver.caltech.edu/CaltechTHESIS:06022021-000654935}.

\end{thebibliography}

\Addresses

\addtocontents{toc}{\protect\setcounter{tocdepth}{0}}

\pagebreak

\appendix 

\section[A $p$-converse theorem for CM elliptic curves]{A $p$-converse theorem for CM elliptic curves \\(by Ashay Burungale and Christopher Skinner)} 

\addtocontents{toc}{\protect\setcounter{tocdepth}{1}}

\addcontentsline{toc}{section}{\vspace{.05in}Appendix A: A $p$-converse theorem for CM elliptic curves
}

In this appendix we explain a proof of:

\begin{theorem}\label{theorem:pcm}
Let $E$ be a CM elliptic curve over $\BQ$ and let $p$ be a prime of supersingular reduction for $E$. 
If
\begin{itemize}
\item[{\rm (a)}] $\corank_{\BZ_p}\Sel_{p^\infty}(E)=1$ and 
\item[{\rm (b)}] the localisation map $\Sel_{p^\infty}(E) \stackrel{\sim}{\rightarrow} E(\BQ_p)\otimes_{\BZ_p}\BQ_p/\BZ_p$ is surjective, 
\end{itemize}
then 
$$
\ord_{s=1}L(E,s)=1=\mathrm{rank}_\BZ E(\BQ).
$$
\end{theorem}

As a consequence we deduce:
\begin{corollary}\label{cor:pcm}
Let $E$ be a CM elliptic curve over $\BQ$ and let $p$ be a prime of supersingular reduction for $E$. If
\begin{itemize}
\item[{\rm (a)}] $\Sel_{p}(E) \simeq \BZ/p\BZ$ and
\item[{\rm (b)}] the localisation map $\Sel_{p}(E) \ra E(\BQ_{p})/pE(\BQ_p)$ is nonzero,
\end{itemize}
then $\ord_{s=1}L(E,s)=1 = \rank_\BZ E(\BQ)$.
\end{corollary}

In the case of good ordinary reduction we actually have a stronger result:
\begin{theorem}\label{theorem:pcm2}
Let $E$ be a CM elliptic curve over $\BQ$ and let $p$ be a prime of good ordinary reduction for $E$. Then 
$$
\corank_{\BZ_{p}}\Sel_{p^{\infty}}(E)=1  \implies \ord_{s=1}L(E,s)=1=\mathrm{rank}_\BZ\, E(\BQ).
$$
\end{theorem}
\noindent This theorem follows immediately from the main results of \cite{BT1}, \cite{BCST}, and \cite{Yu}.

As a consequence of Theorem \ref{theorem:pcm2} we have:
\begin{corollary}\label{cor:pcm2}
Let $E$ be a CM elliptic curve over $\BQ$ and let $p$ be a prime of good ordinary reduction for $E$. If $\Sel_{p}(E)/\mathrm{im}(E[p](\BQ))\simeq \BZ/p\BZ$, then $\ord_{s=1}L(E,s)=1 = \rank_\BZ E(\BQ)$.
\end{corollary}

\noindent In the above corollary, $\mathrm{im}(E[p](\BQ))$ is the image of $E[p](\BQ)$ under the Kummer map.

Before embarking on the proof of Theorem \ref{theorem:pcm}, we make a few remarks about these results.

\begin{remark}\label{rmk:pcm}\hfill{\rm 
\begin{itemize}
\item[{\rm (i)}] We emphasize that all these results allow for $p=2$. This is, of course, crucial for the application of Corollary $\ref{cor:pcm}$ in the main body of this paper. 
\item[{\rm (ii)}] In both the theorems and corollaries the finiteness of $\Sha(E)$ $($that is, $\#\Sha(E)<\infty)$ can be added to the final conclusion.
\item[{\rm (iii)}] Theorem $\ref{theorem:pcm}$ is the culmination of a number of prior results, especially {\rm \cite{Ru2}, \cite{BKO},} and {\rm \cite{BT1}}.
\item[{\rm (iv)}] A similar converse for non-CM curves was proved in {\rm \cite{Sk}}, with similar application $($cf.~{\rm \cite{BS}}$)$. 
\end{itemize}}
\end{remark}

\subsection{Proof of Theorem \ref{theorem:pcm}}

The proof of Theorem \ref{theorem:pcm} ties together a number of results on the Iwasawa theory of elliptic curves, especially curves with CM, which we now recall.

\subsubsection{Kato's main conjecture}
Let $E$ be an elliptic curve over $\BQ$. 
For a prime $p$, let $T$ denote the $p$-adic Tate module of $E$ and $V=T\otimes_{\BZ_{p}}\BQ_{p}$. 

Let $\BQ_\infty$ be the cyclotomic $\BZ_{p}$-extension of $\BQ$, $\Gamma=\Gal(\BQ_{\infty}/\BQ)$ and 
$\Lambda=\BZ_{p}[\![\Gamma]\!]$. Fix a topological generator $\gamma\in \Gamma$.  
For a finitely-generated $\Lambda$- or $\Lambda\otimes_{\BZ_p}\BQ_p$-module $M$ let $\xi(M)$ denote its characteristic ideal
(which should be clear from context). 

Let $S_{\st}(E) \subset H^1(\BZ[\frac{1}{p}],T\otimes_{\BZ_p}\Lambda^*)$ be the strict Selmer group of $E$ over $\BQ_\infty$ (the subgroup of classes that are trivial at $p$). Here $\Lambda^*$ is the Pontryagin dual of $\Lambda$ with $G_\BQ$-acting by the inverse of the canonical character $G_\BQ\twoheadrightarrow \Gamma \subset \Lambda^\times$.
Let $X_{\st}(E)$ be the Pontryagin dual of $S_\st(E)$. It is one of the main results of Kato \cite[Thm.~12.4]{K} that $X_\st(E)$ (in the guise of $H^2(\BZ[\frac{1}{p}],T\otimes_{\BZ_p}\Lambda)$) is a finitely-generated torsion $\Lambda$-module. 
Let ${\bf{z}}_{E}\in H^{1}(\BZ[\frac{1}{p}],T\otimes_{\BZ_p}\Lambda) \otimes_{\BZ_{p}}\BQ_{p}$ denote the Beilinson--Kato element \cite{K} and let 
$z_{E} \in H^{1}(\BQ,V)$ be its image under the specialisation $\gamma \mapsto 1$. 
The following special case of \cite[Conj.~12.10]{K} is proved in \cite{BT2}:
\begin{theorem}\label{theorem:Kato} 
Let $E$ be a CM elliptic curve over $\BQ$ and $p$ any prime. 
Then 
$$
\xi\big{(}(H^{1}(\BZ[\tfrac{1}{p}], T \otimes_{\BZ_{p}} \Lambda)\otimes_{\BZ_p} \BQ_{p})/ (\Lambda \otimes \BQ_{p}) \cdot{{\bf{z}}_{E}}\big{)}
=\xi(X_{\st}(E)\otimes_{\BZ_p}\BQ_p).
$$
\end{theorem}

\begin{remark}\label{rmk:Kato}{\rm For primes of ordinary reduction the same result is due to Kato and Rubin, at least if $p\nmid \#\mathcal{O}_K^\times$.}
\end{remark}

\subsubsection{Perrin-Riou's Conjecture} 
Let $E$ be an elliptic curve over $\BQ$ and $p$ a prime. 
Let $H^{1}_{\mathrm{f}}(\BQ_{p},V) \subset H^{1}(\BQ_{p},V)$ denote the subgroup arising from the Kummer image of $E(\BQ_p)$.
By Kato's explicit reciprocity law \cite[Thm.\ 12.5]{K}, 
$$
\loc_{p}(z_{E}) \in H^{1}_{\mathrm{f}}(\BQ_{p},V) \iff L(E,1)=0.
$$
If $L(E,1)=0$, then Perrin-Riou \cite{PR} conjectured $z_{E}$ to be closely linked with the arithmetic of $E$. The following theorem, proved in \cite{BKO} and \cite{Ko},
is evidence for this in the supersingular case:

\begin{theorem}\label{theorem:PR}
Let $E$ be an elliptic curve over $\BQ$ and $p$ a prime of supersingular reduction.
If $L(E,1)=0$, then
there exist $P\in E(\BQ)$ and $c_P \in \BQ^{\times}$ with the following properties. 
\begin{itemize}
\item[$(a)$] We have 
\[
\log(\mathrm{loc}_{p}(z_{E})) = c_{P}  \left(\frac{p+1-a_{p}(E)}{p}\right) \cdot \log(P)^{2} 
\]
for  $\log:H^{1}_{{\mathrm{f}}}(\BQ_{p},V) \ra \BQ_{p}$ the logarithm map associated to the N\'eron differential. 
\item[$(b)$] The point $P$ is non-torsion if and only if $\ord_{s=1}L(E,s)=1.$ 
\end{itemize}
\end{theorem}

\begin{remark}\label{rmk:PR} {\rm For $p$ a prime of good ordinary reduction the same result is known $($cf.~\cite{BST} for $p\geq 5$ and even for $p\geq 3$ under an assumption that the conductor of $E$ is suitably minimal at all primes ramified in the CM field---a condition that can be relaxed$)$. This allows for a uniform proof of Theorem \ref{theorem:pcm} for all primes of good reduction. Of course, the stronger result Theorem \ref{theorem:pcm2} is known, but the existing proofs run along different lines than our proof of Theorem \ref{theorem:pcm}.} 
\end{remark}

\subsubsection{Putting the pieces together} Let $E$ be as in Theorem \ref{theorem:pcm}.
Let $\Sel_{\st}(E)\subset \Sel_{p^\infty}(E)$ be the strict Selmer group, consisting of those classes that vanish under $\loc_p$.
By the assumption, $\Sel_{\st}(E)$ is finite.  The same is then true of $X_{\st}(E)/(\gamma-1)X_{\st}(E)$ (which naturally surjects onto the Pontryagin dual of $\Sel_\st(E)$ with finite kernel).  
Kato proved that $H^{1}(\BZ[\frac{1}{p}],T\otimes_{\BZ_{p}}\Lambda)\otimes_{\BZ_{p}}\BQ_{p}$ 
is a free $\Lambda\otimes_{\BZ_{p}}\BQ_{p}$-module of rank one \cite[Thm.~12.4]{K}. 
So it follows from Theorem \ref{theorem:Kato} and the finiteness of $X_{\st}(E)/(\gamma-1)X_{\st}(E)$ that 
$$
0 \neq z_{E} \in H^{1}(\BQ,V). 
$$
Since $\Sel_{\st}(E)$ is finite, it further follows that
$$
 0 \neq \loc_{p}(z_E) \in H^1(\BQ_{p}, V).
$$
The hypothesis that $\corank_{\BZ_p}\Sel_{p^\infty}(E)=1$ implies that $L(E,1)=0$ (this just follows from the Gross--Zagier and Kolyvagin theorem or from the parity conjecture). It then follows from Theorem \ref{theorem:PR} that $\rank_{\BZ}E(\BQ)\geq 1$ and $\ord_{s=1}L(E,s) = 1$. That $\rank_{\BZ}E(\BQ)=1$ now follows from $\corank_{\BZ_p}\Sel_{p^\infty}(E)=1$. This completes the proof of Theorem \ref{theorem:pcm}.

We now briefly explain how Corollary \ref{cor:pcm} follows from Theorem \ref{theorem:pcm}.  
We first note that since $E$ has good supersingular reduction at $p$, $E[p](\BQ_p) = 0$ by \cite[Prop.~12(d)]{Se}. Hence $E[p](\BQ)=0$. The Cassels--Tate pairing implies that $\Sel_{p^\infty}(E) \cong (\BQ_p/\BZ_p)^r \oplus M \oplus M$, for some integer $r\geq 0$ and some finitely-generated torsion $\BZ_p$-module $M$. As $\Sel_p(E)/\mathrm{im}(E[p](\BQ))=\Sel_{p^\infty}(E)[p]$, condition (a) of the corollary then implies that $\Sel_{p^\infty}(E) \cong \BQ_p/\BZ_p$, so condition (a) of the theorem holds. As $E(\BQ_p)\otimes_{\BZ_p}\BQ_p/\BZ_p \cong \BQ_p/\BZ_p$ and the natural map $E(\BQ_p)/pE(\BQ_p) \twoheadrightarrow (E(\BQ_p)\otimes_{\BZ_p}\BQ_p/\BZ_p)[p]$ is an isomorphism (since $E[p](\BQ_p) = 0$), 
condition (b) of the corollary then implies condition (b) of the theorem.

\medskip
\noindent{\em Acknowledgements.} 
We are grateful to Levent Alp\"oge, Manjul Bhargava, and Ari Shnidman for suggesting the problem. Thanks are also due to Shinichi Kobayashi, Kazuto Ota, Karl Rubin and Ye Tian for helpful discussions. 
The work of C.S. was partially supported by the Simons Investigator Grant \#376203 from the Simons Foundation and
 the National Science Foundation Grant DMS-1901985, and that of A.B. by the National Science Foundation Grant DMS-2001409.

\AddressesA

\end{document}